\documentclass[a4paper,oneside,11pt]{article}
\usepackage{amsmath}
\usepackage{amssymb}
\usepackage{amsthm}
\usepackage{MnSymbol}
\usepackage{enumitem}
\usepackage{bm}
\usepackage{graphicx}
\usepackage[margin = 3cm]{geometry}
\usepackage{color}

\newcommand{\N}{\mathbb{N}}
\newcommand{\R}{\mathbb{R}}
\newcommand{\Y}{\mathbb{Y}}
\newcommand{\JJ}{\mathcal{J}}

\newcommand{\bH}{\bm{H}}

\newcommand{\BB}{\mathcal{B}}
\newcommand{\GG}{\mathcal{G}}
\newcommand{\HH}{\mathcal{H}}
\newcommand{\LL}{\bm{L}}
\newcommand{\K}{\mathbb{K}}
\newcommand{\e}[1]{#1_{\eps}}
\newcommand{\supp}[1]{\mathrm{supp}(#1)}
\newcommand{\dHaus}{\, \mathrm{d} \mathcal{H}^{d-1} \,}
\newcommand{\Haus}{\mathcal{H}^{d-1}}
\newcommand{\dd}{\, \mathrm{d} \,}

\newcommand{\dx}{\, \mathrm{dx}\,}
\newcommand{\ds}{\, \mathrm{ds}\,}

\newcommand{\pd}{\partial}
\newcommand{\eps}{\varepsilon}
\newcommand{\der}{ \mathrm{D} }
\newcommand{\id}{\, \bm{\mathrm{I}}\,}
\newcommand{\abs}[1]{\left| \, #1 \, \right|}
\newcommand{\norm}[1]{\| #1 \|}

\newcommand{\inner}[2]{\langle #1 , #2 \rangle}
\newcommand{\biginner}[2]{\left\langle #1 , #2 \right\rangle}

\newcommand{\Laplace}{\Delta}
\renewcommand{\div}{\, \mathrm{div}\,}
\newcommand{\hp}{s}
\newcommand{\flow}{\bm{u}_{\infty}}
\newcommand{\linvel}{\bm{w}_{\eps}}
\newcommand{\linp}{r_{\eps}}

\newtheorem{thm}{Theorem}[section]
\newtheorem{lemma}[thm]{Lemma}

\newtheorem{remark}{Remark}[section]

\numberwithin{equation}{section}

\newenvironment{change}{\color{black}}{\color{black}}
\newcommand{\AAA}{\begin{change}}
\newcommand{\BBB}{\end{change}}

\begin{document}

\title{
A phase field approach to shape optimization in Navier--Stokes flow with integral state constraints
}

\author{Harald Garcke\footnotemark[1] 
\and Michael Hinze\footnotemark[2] 
\and Christian Kahle\footnotemark[3] 
\and Kei Fong Lam\footnotemark[4]}

\date{ }

\maketitle

\renewcommand{\thefootnote}{\fnsymbol{footnote}}

\footnotetext[1]{Fakult\"at f\"ur Mathematik, Universit\"at Regensburg, 93040 Regensburg, Germany
({\tt Harald.Garcke@mathematik.uni-regensburg.de}).}
\footnotetext[2]{Schwerpunkt Optimierung und Approximation, 
Fachbereich Mathematik, 
Universit\"at Hamburg, 
Bundesstrasse 55, 
20146 Hamburg, 
Germany ({\tt Michael.Hinze@uni-hamburg.de}).} 
\footnotetext[3]{Lehrstuhl f\"ur Optimalsteuerung, 
Zentrum Mathematik,
Technische Universit\"at M\"unchen, 
Garching bei M\"unchen, Germany 
({\tt Christian.Kahle@ma.tum.de}).}
\footnotetext[4]{Department of Mathematics, The Chinese University of Hong Kong, Shatin, N.T., Hong Kong ({\tt kflam@math.cuhk.edu.hk}).}

\renewcommand{\thefootnote}{\arabic{footnote}}

\begin{abstract}
We consider the shape optimization of an object in Navier--Stokes flow by employing a combined phase field and porous medium approach, along with additional perimeter regularization.  By considering integral control and state constraints, we extend the results of earlier works concerning the existence of optimal shapes and the derivation of first order optimality conditions.  The control variable is a phase field function that prescribes the shape and topology of the object, while the state variables are the velocity and the pressure of the fluid.  In our analysis, we cover a multitude of constraints which include constraints on the center of mass, the volume of the fluid region, and the drag of the object.  Finally, we present numerical results of the optimization problem that is solved using the variable metric projection type (VMPT) method proposed by Blank and Rupprecht, where we consider one example of topology optimization without constraints and one example of maximizing the lift of the object with a state constraint, as well as a comparison  with earlier results for the drag minimization.
\end{abstract}



\section{Introduction}\label{sec:Intro}
Fundamental to the design of aircraft and cars, as well as any technologies that would involve an object traveling within a fluid, such as wind turbines and drug delivery in biomedical applications, is the consideration of hydrodynamic forces acting on the object, for example the drag and lift forces.  The desire to construct an object with minimal drag or with maximal lift-to-drag ratio naturally leads to the notion of shape optimization in fluids, in which the problem can often be formulated in terms of an optimal control problem with PDE constraints.

Let us assume that $\Omega \subset \R^{d}$, $d = 2,3$, is a bounded domain with Lipschitz boundary, and contains a non-permeable object $B$.  We will denote the boundary of $B$ by $\Gamma := \pd B \cap \Omega$ with the outer unit normal $\bm{\nu}$, and assume that $\Gamma \cap \pd \Omega = \emptyset$, i.e., the object $B$ never touches the external boundary.  A fluid is present in the complement region $E := \Omega \setminus B$, and we assume that the velocity $\bm{u}$ and the pressure $p$ of the fluid in the region $E$ obey the stationary Navier--Stokes equations with no-slip conditions on $\Gamma$, namely,
\begin{subequations}\label{SI:NS}
\begin{alignat}{3}
-\mu \Laplace \bm{u} + (\bm{u} \cdot \nabla) \bm{u} + \nabla p & = \bm{f} && \text{ in } E, \\
\div \bm{u} & = 0 && \text{ in } E, \\
\bm{u} & = \bm{0} &&\text{ on } \Gamma, \\
\bm{u} & = \bm{g} && \text{ on } \pd E \cap \pd \Omega.
\end{alignat}
\end{subequations}
Here $\bm{f}$ denotes the external body force, $\mu$ denotes the  \AAA (constant) \BBB viscosity, and $\bm{g}$ models the
inflow and outflow on the boundary $\pd \Omega$ such that $\int_{\pd \Omega} \bm{g} \cdot
\bm{\nu}_{\pd \Omega} \dHaus = 0$, where $\bm{\nu}_{\pd \Omega}$ denotes the outer unit normal on
$\pd \Omega$.

Our present contribution is motivated from a previous numerical study \cite{GHHKL} for the shape optimization problem of maximizing the lift-to-drag ratio subject to the PDE constraint \eqref{SI:NS}.  In two spatial dimensions, \AAA the classical formulation of \BBB the lift-to-drag ratio is defined as
\begin{align}\label{LtoD}
\frac{\int_{\Gamma} \flow^{\perp} \cdot \left (\mu \left ( \nabla \bm{u} + \left ( \nabla \bm{u} \right )^{\top} \right ) - p \id \right ) \bm{\nu} \dHaus}{\int_{\Gamma} \flow \cdot \left (\mu \left ( \nabla \bm{u} + \left ( \nabla \bm{u} \right )^{\top} \right ) - p \id \right ) \bm{\nu} \dHaus},
\end{align}
where $\flow$ is the flow direction, $\flow^{\perp}$ is the perpendicular vector and $\Haus$ is the Hausdorff measure on the set $\Gamma$.  In \cite{GHHKL}, using a phase field approximation which we will detail below, the authors obtain an optimal shape similar to a \AAA non-symmetric \BBB airfoil \AAA with a small angle of attack. \BBB  However, a \AAA chief \BBB obstacle to a rigorous mathematical treatment of the problem is that it is unknown if the lift-to-drag ratio \eqref{LtoD} is bounded from above (as we want to maximize the ratio).  \AAA Furthermore, due to the fractional form of the lift-to-drag ratio \eqref{LtoD}, we also observe fractions entering in the associated adjoint system and optimality conditions computed by the formal Lagrangian method, leading to severe complications in the numerical implementation. \BBB

\AAA One idea is to study a related problem involving maximizing the lift while the drag is constrained to be below a certain threshold, i.e.,
\begin{align*}
& \max \int_{\Gamma} \flow^{\perp} \cdot \left (\mu \left ( \nabla \bm{u} + \left ( \nabla \bm{u} \right )^{\top} \right ) - p \id \right ) \bm{\nu} \dHaus \\
& \text{ subject to \eqref{SI:NS} and } \; \int_{\Gamma} \flow \cdot \left (\mu \left ( \nabla \bm{u} + \left ( \nabla \bm{u} \right )^{\top} \right ) - p \id \right ) \bm{\nu} \dHaus \leq D,
\end{align*}
where $D > 0$ is a threshold for the drag.  In this case, the problematic fractional form is replaced and analysis can be performed on the optimization problem.   In exchange we now have to deal with (integral) state constraints, and the difficulty lies in establishing the existence of the associated Lagrange multipliers.  
\BBB

To fix the setting for this paper, we now introduce a design function $\varphi : \Omega \to \{ \pm 1 \}$, where $\{\varphi = 1\} = E$ 
describes the fluid region and $\{\varphi = -1 \} = B$ is its complement.  The natural function space for the design functions is the space of bounded variations that take values $\pm 1$, i.e., $\varphi \in BV(\Omega, \{\pm 1\})$, which implies that the fluid region $E$ has finite perimeter $P_{\Omega}(E)$.  If $\varphi$ is a function of bounded variation, its distributional derivative $\der \varphi$ is a finite Radon measure which can be decomposed into a positive measure $\abs{\der \varphi}$ and a $S^{d-1}$-valued function $\bm{\nu}_{\varphi} \in L^{1}(\Omega, \abs{\der \varphi})^{d}$, \AAA where $S^{d-1}$ denotes the $(d-1)$-dimensional sphere. \BBB  The total variation for $\varphi \in BV(\Omega, \{ \pm 1 \})$, denoted by $\abs{\der \varphi}(\Omega)$, satisfies
\begin{align*}
\abs{\der \varphi}(\Omega) = 2 P_{\Omega} (\{\varphi = 1 \}),
\end{align*}
and thus we can express the Hausdorff measure \AAA $\Haus$ \BBB on the set $\Gamma$ as $ \frac{1}{2} \abs{\der \varphi}(\Omega)$.  Furthermore, the $S^{d-1}$-valued function $\bm{\nu}_{\varphi}$ can be considered as a generalized normal on the set $\pd \{\varphi = 1\}$ \AAA (see \cite{Ambrosio,EvansGariepy,Giusti} for a more detailed introduction to the theory of sets of finite perimeter and functions of bounded variation). \BBB  

For functions $b : \Omega \times \R^{d} \times \R^{d \times d} \times \R \times \R \to \R$ and $h : \Omega \times \R^{d \times d} \times \R \times \R^{d} \to \R$, we consider the following \AAA general \BBB shape optimization problem with perimeter regularization:
\begin{equation}\label{SI:Objective}
\begin{aligned}
\min_{(\varphi, \bm{u}, p)} \JJ_{0}(\varphi, \bm{u}, p) & := \int_{\Omega} b(x, \bm{u}, \nabla \bm{u}, p, \varphi) \dx \\
& + \int_{\Omega} \frac{1}{2} h(x, \nabla \bm{u}, p, \bm{\nu}_{\varphi}) \dd \abs{\der \varphi} + \frac{\gamma}{2} \abs{\der \varphi}(\Omega),
\end{aligned}
\end{equation}
subject to $\varphi \in BV(\Omega, \{ \pm 1 \})$ and $(\bm{u}, p) \in \bm{H}^{1}(E) \times L^{2}(E)$ fulfilling
\begin{subequations}\label{SI:state}
\begin{alignat}{3}
-\mu \Laplace \bm{u} + (\bm{u} \cdot \nabla) \bm{u} + \nabla p & = \bm{f} && \text{ in } E = \{\varphi = 1\}, \label{SI:state:1} \\
\div \bm{u} & = 0 && \text{ in } E, \\
\bm{u} & = \bm{g} && \text{ on } \pd \Omega \cap \pd E, \\
\bm{u} & = \bm{0} && \text{ on } \Gamma = \Omega \cap \pd E. \label{SI:noslip}
\end{alignat}
\end{subequations}
In addition, for fixed $m_{1}, m_{2} \in \N \cup \{0\}$, we impose the $m_{1}$ integral equality constraints and $m_{2}$ integral inequality constraints:
\begin{align}\label{SI:Const}
G_{i}(\varphi, \bm{u}, p) = 0 \text{ for } 1 \leq i \leq m_{1}, \quad G_{i}(\varphi, \bm{u}, p) \geq 0 \text{ for } m_{1} + 1 \leq i \leq m_{1} + m_{2},
\end{align}
where \AAA for each $1 \leq i \leq m_{1}+m_{2}$, \BBB
\begin{equation}\label{SI:IntegralConstraints}
\begin{aligned}
G_{i}(\varphi, \bm{u}, p) & := \int_{\Omega} K_{i}(x, \bm{u}, \nabla \bm{u}, p, \varphi) \dx \\
& \quad + \int_{\Omega} \frac{1}{2} \LL_{i}(x, \nabla \bm{u}, p) \cdot \bm{\nu}_{\varphi} \dd \abs{\der \varphi},
\end{aligned}
\end{equation}
for functions $K_{i} : \Omega \times \R^{d} \times \R^{d \times d} \times \R \times \R \to \R$ and $\LL_{i} : \Omega \times \R^{d \times d} \times \R \to \R^{d}$.  The parameter $\gamma > 0$ in \eqref{SI:Objective} is the weighting factor for the perimeter regularization, \AAA which is given by the term $\frac{1}{2} \abs{\der \varphi}$, and in light of the above discussion regarding the measure $\frac{1}{2} \abs{\der \varphi}$ representing the Hausdorff measure on $\Gamma$, we see that the functions $b$ and $\{K_{i}\}_{i=1}^{m_{1}+m_{2}}$ model objectives and constraints in the bulk phases $E$ and $B$, while $h$ and $\{\LL_{i} \cdot \bm{\nu}_{\varphi}\}_{i=1}^{m_{1}+m_{2}}$ model constraints on the interface $\Gamma$.   Examples of $b$, $h$, $K_{i}$ and $\LL_{i}$ are given below, and it is noteworthy to point out that there is no dependence on $\bm{u}$ in $\LL_{i}$ as the no-slip condition \eqref{SI:noslip} ensures that $\bm{u} = \bm{0}$ on $\Gamma$.  However, the gradient $\nabla \bm{u}$ may not vanish on $\Gamma$, which leads to its appearance in the surface constraints. \BBB

\AAA
The appearance of the perimeter regularization $\frac{\gamma}{2} \abs{ \der \varphi}(\Omega)$ in \eqref{SI:Objective} is motivated from the well-known difficulties regarding the mathematical treatment of shape optimization - in particular the existence of minimizers/optimal shapes are not guaranteed \cite{KPTZ,Murat,Tartar}.  However, if the shape optimization problem is additionally supplemented with a perimeter regularization, then positive results concerning existence of optimal shapes have been obtained (see for instance \cite{SHH}).
\BBB
 
Let us now give some examples of \AAA functions $b$, $h$, $K$ and $\LL$ (where we neglect the index $i$ for convenience) that are of relevance.  For a subset $A \subset \Omega$, we use the notation $\chi_{A}(x)$ to denote the characteristic function of $A$, i.e., $\chi_{A}(x) = 1$ if $ x \in A$ and $\chi_{A}(x) = 0$ if $x \in \Omega \setminus A$.  In particular, one can think of the design function $\varphi$ as $\varphi(x) = -1 + 2 \chi_{E}(x)$ which satisfies $\varphi(x) = 1$ for $x \in E$ and $\varphi(x) = -1$ for $x \in \Omega \setminus E = B$.  Hence, in the following examples for the function $b$, we can use $\frac{1}{2}(1+\varphi)$ as a restriction to the region $E$ and similarly, $\frac{1}{2}(1-\varphi)$ as a restriction to the region $B$:\BBB
\begin{itemize}
\item the total potential power \AAA of the fluid \BBB
\begin{align}\label{PotPower}
\frac{1+\varphi}{2} \left ( \frac{\mu}{2} \abs{\nabla
\bm{u}}^{2} - \bm{f} \cdot \bm{u} \right ),
\end{align}
\item the construction cost of the object $\frac{1-\varphi}{2} w(x)$, where $w$ denotes a cost function per unit volume,
\item the least square approximation 
\begin{align*}
\frac{1+\varphi}{2} \chi_{\mathcal{Q}}(x) (\delta_{1} \abs{p - p_{\mathrm{tar}}}^{2} + \delta_{2} \abs{\bm{u} - \bm{u}_{\mathrm{tar}}}^{2})
\end{align*} 
to a target velocity profile $\bm{u}_{\mathrm{tar}}$ and a target pressure profile $p_{\mathrm{tar}}$ in an
observation region $\mathcal{Q} \subset E$. Here $\delta_{1}$ and $\delta_{2}$ denote nonnegative
constants.
\end{itemize}
An example for the surface cost $h$ 
which \AAA has practical applications is \BBB the hydrodynamic force component in the direction of the unit vector $\bm{a}$, which is given as 
\begin{align}\label{hydrodynamic}
\bm{a} \cdot \left (\mu  \left ( \nabla \bm{u} + \left (\nabla \bm{u} \right )^{\top} \right ) - p \id \right ) \bm{\nu}_{\varphi},
\end{align}
where $\id$ denotes the identity tensor.  The drag of the object is given when $\bm{a}$ is parallel to the flow direction $\flow$, while the lift of the object is given when $\bm{a} = \flow^{\perp}$, the unit vector perpendicular to the flow direction.

\bigskip

Examples of integral constraints \AAA that are of interests include \BBB
\begin{itemize}
\item volume constraints on the amount of fluid -  \AAA setting $K_{1} = \varphi - \beta_{1}$, $\LL_{1} = \bm{0}$, $K_{2} = \beta_{2} - \varphi$ and $\LL_{2} = \bm{0}$ \BBB for fixed constants $-1 < \beta_{1} \leq \beta_{2} < 1$ leads to inequality constraints:
\begin{align*}
\AAA G_{1}(\varphi) = \int_{\Omega} \varphi - \beta_{1} \dx \geq 0, \quad G_{2}(\varphi) = \int_{\Omega} \beta_{2} - \varphi \dx \geq 0, \BBB
\end{align*}
\AAA or equivalently \BBB
\begin{equation}\label{SI:VolumeConstraint}
\begin{aligned}
& \frac{\beta_{1} + 1}{2} \abs{\Omega} \leq \int_{\Omega} \frac{1+\varphi}{2} \dx = \abs{E}  \leq \frac{\beta_{2}+1}{2} \abs{\Omega} \\
\Leftrightarrow & \; \beta_{1} \abs{\Omega} \leq \int_{\Omega} \varphi \dx \leq \beta_{2} \abs{\Omega},
\end{aligned}
\end{equation}
\item the prescribed mass of the object - \AAA setting \BBB $K_{1} = M \abs{\Omega}^{-1} - \frac{1-\varphi}{2}\rho(x)$, \AAA $ \LL_{1} = \bm{0}$, \BBB where $\rho(x)$ is a mass density and $M > 0$ is a target/maximal mass leads to the inequality constraint
\begin{align}\label{SI:MassConstraint}
\AAA G_{1}(\varphi) = M - \int_{\Omega} \tfrac{1}{2} \rho(x) (1-\varphi) \dx  \geq 0 \Leftrightarrow \BBB \int_{\Omega} \tfrac{1}{2} \rho(x) (1-\varphi) \dx \leq M,
\end{align}
\item the prescribed center of mass of the object (with uniform mass density) at a point $y$ in the interior of $\Omega$, i.e., $y \notin \pd \Omega$ - \AAA setting \BBB $K_{i} = \frac{1-\varphi}{2}(x_{i} - y_{i})$ \AAA and $\LL_{i} = \bm{0}$ \BBB for $1 \leq i \leq d$ leads to the equality constraints
\begin{align}\label{SI:CoM}
G_{i}(\varphi) = \int_{\Omega} \tfrac{1}{2} (1-\varphi)( x_{i} - y_{i}) \dx = 0 \text{ for } i = 1, 2, \dots, d,
\end{align}
\item the prescribed drag of the object - \AAA setting \BBB $\LL_{1} = - \mu ( \nabla \bm{u} + (\nabla \bm{u})^{\top})\bm{a} + p \bm{a}$, $K_{1} = D \abs{\Omega}^{-1}$, where $\bm{a}$ is the unit vector parallel to the flow direction $\flow$ and $D > 0$ is a maximal drag value leads to the inequality constraint
\begin{equation}\label{SI:DragCon}
\begin{aligned}
& G_{1}(\varphi, \bm{u}, p) = D - \bm{a} \cdot \int_{\Omega} \left ( \mu ( \nabla \bm{u} + (\nabla \bm{u})^{\top}) \bm{\nu}_{\varphi} - p \bm{\nu}_{\varphi} \right ) \frac{1}{2} \dd \abs{\der \varphi} \geq 0 \\
& \Leftrightarrow \; \bm{a} \cdot \int_{\Omega} \left ( \mu ( \nabla \bm{u} + (\nabla \bm{u})^{\top}) \bm{\nu}_{\varphi} - p \bm{\nu}_{\varphi} \right ) \frac{1}{2} \dd \abs{\der \varphi} \leq D.
\end{aligned}
\end{equation}
\end{itemize}
In the examples of the cost functional described above, the problem involving minimizing the drag of the object has received much attention and is well-studied in the literature, see \cite{Boisgerault,Bello97,Pironneau,PlotSoko,Simon91} and the references therein.  For the formal derivation of shape derivatives with general volume and boundary objective functionals in Navier--Stokes flow, we refer the reader to \cite{SS10}, but to the authors' best knowledge, the shape optimization problem with integral state constraint has not received much attention.  

\AAA
In this work, we study a phase field approximation of the problem \eqref{SI:Objective}-\eqref{SI:IntegralConstraints}, which is derived in Sec.~\ref{sec:PF}.  Under appropriate assumptions we prove the existence of minimizers to the phase field optimization problem and derive the first order optimality conditions.  The main difficulty we encounter is establishing the existence of Lagrange multipliers, which is achieved via constraint qualifications such as the Zowe--Kurcyusz constraint qualification \cite{ZoweKurcyusz}, for some of the integral state constraints mentioned above.  We give two examples: one involves constraints that depend only on the design function $\varphi$ which are the volume \eqref{SI:VolumeConstraint}, the prescribed mass \eqref{SI:MassConstraint} and the center of mass \eqref{SI:CoM}, while the second example involves the total potential power \eqref{PotPower}.  We encounter some technical difficulties regarding the drag constraint \eqref{SI:DragCon}, and can only show the existence of Lagrange multipliers if the threshold $D$ is sufficiently large.  Via numerical simulations we give a proof of concept showing that with the help of the phase field approach shape and topology optimization for fluid flow taking state constraints can be solved.  For large Reynolds number problems more efficient numerical solution methods have to be devised in the future.
\BBB

The rest of the paper is organized as follows:  In Sec.~\ref{sec:PF}, we present the phase field approximation of \eqref{SI:Objective}-\eqref{SI:IntegralConstraints} that utilizes the porous-medium approach of Borrvall and Petersson \cite{BP03}, and state several preliminary results on the state equations.  Then, in Sec.~\ref{sec:min}, we state the assumptions on $b$, $h$, $K_{i}$ and $\LL_{i}$ that allows us to establish the existence of minimizers to the phase field shape optimization problem.  Sufficient conditions on the differentiability of $b$, $h$, $K_{i}$ and $\LL_{i}$ are outlined in Sec.~\ref{sec:opt} which lead to the existence of Lagrange multipliers, the solvability of the adjoint system, and the derivation of the necessary optimality conditions.  We verify the aforementioned conditions in Sec.~\ref{sec:VerifyZK} for two specific examples of integral constraints; the first example involves constraints on the mass, center of mass and volume, while the second example involves a constraint on the total potential power.  Lastly, in Sec.~\ref{sec:numerics} we briefly outline our numerical approach to solving the optimality conditions, and present several numerical simulations.

\section{Phase field formulation}\label{sec:PF}
\AAA One approach to tackle shape optimization problems that can yield rigorous mathematical results is to employ a phase field approximation, similar in spirit to Bourdin and Chambolle \cite{BC03} that was applied to topology optimization (see also \cite{Blank,PRW,Takezawa,WangZhou} and the reference cited therein), and has been recently used for drag minimization in stationary Stokes flow \cite{GarckeHechtStokes} and in stationary Navier--Stokes flow \cite{GarckeHechtNS,GHHKNumerics,GHHKL,Kondoh}. \BBB

\AAA The approach we take in this paper is similar to the previous works \cite{GarckeHechtNS,GarckeHechtStokes,GHHKL}, in which \BBB we relax the condition that the design function $\varphi$ takes only values in $\{ \pm 1 \}$ (i.e., $\varphi \in BV(\Omega, \{ \pm 1 \})$) and now allow $\varphi$ to be a function with values in $\R$ and inherits $H^{1}(\Omega)$ regularity.  \AAA In particular, we change the admissible space of design functions from subsets of $BV(\Omega, \{ \pm 1 \})$ to subsets of $H^{1}(\Omega)$. \BBB This leads to the development of interfacial layers $\{-1 < \varphi < 1\}$ in between the fluid region $E = \{\varphi = 1\}$ and the object region $B = \{\varphi = -1 \}$.  This interfacial layer replaces the boundary $\Gamma$ of $B$ and a parameter $\eps > 0$ is associated to the thickness of the interfacial layer.  The idea \AAA is to reformulate the original shape optimization problem \eqref{SI:Objective}-\eqref{SI:IntegralConstraints} by taking into account the above modification of the design functions. For the perimeter regularization, we can use \BBB the \AAA scaled \BBB Ginzburg--Landau energy functional
\begin{align*}
\frac{1}{2 c_{0}}\e{\mathcal{E}}(\varphi) = \frac{1}{2 c_{0}}\int_{\Omega} \frac{\eps}{2} \abs{\nabla \varphi}^{2} + \frac{1}{\eps} \Psi(\varphi) \dx,
\end{align*}
where $\Psi$ is a potential with equal minima at $\varphi = \pm 1$, to approximate the perimeter functional $P_{\Omega}$.  The positive constant $c_{0}$ is dependent only on the potential $\Psi$ via the relation
\begin{align}\label{defn:c0}
c_{0} := \frac{1}{2} \int_{-1}^{1} \sqrt{2 \Psi(s)} \ds,
\end{align}
and it is well-known that $\frac{1}{2 c_{0}}\e{\mathcal{E}}$ approximates $\varphi \mapsto \frac{1}{2} \abs{\der \varphi}(\Omega) = P_{\Omega}(\{\varphi = 1\})$ in the sense of $\Gamma$-convergence \cite{Modica}.

By introducing an interfacial region between the fluid and the object, we have relaxed the non-permeability assumption of the object in the vicinity of its boundary.  Therefore, we use the so-called porous medium approach of Borrvall and Petersson \cite{BP03} and replace the object $B$ with a porous medium of small permeability $(\overline{\alpha_{\eps}})^{-1} \ll 1$.  A function $\alpha_{\eps}(\varphi)$ is introduced to interpolate between the inverse permeabilities of the fluid region $\alpha_{\eps}(1) = 0$ and the porous medium $\alpha_{\eps}(-1) = \overline{\alpha_{\eps}}$, which satisfies
\begin{align*}
\overline{\alpha_{\eps}} \to \infty \text{ as } \eps \to 0.
\end{align*} 
With this, we extend the state equations from $E$ to the whole domain $\Omega$ by the addition of the \emph{porous-medium} term $\alpha_{\eps}(\varphi) \bm{u}$:
\begin{subequations}\label{PF:state}
\begin{alignat}{3}
\alpha_{\eps}(\varphi) \bm{u} -\mu \Laplace \bm{u} + (\bm{u} \cdot \nabla) \bm{u} + \nabla p & = \bm{f} && \text{ in } \Omega, \label{PF:NS} \\
\div \bm{u} & = 0 && \text{ in } \Omega, \\
\bm{u} & = \bm{g} && \text{ on } \pd \Omega.
\end{alignat}
\end{subequations}
We note that this additional term vanishes in the fluid region, and in the limit $\eps \to 0$, one expects the velocity $\bm{u}$ in the object region to vanish.  \AAA We point out that in the modified state equations \eqref{PF:state} we solve for a velocity field $\bm{u}$ and pressure field $p$ that are defined on the fixed domain $\Omega$.  Furthermore, in the objective functional \eqref{SI:Objective} and in the integral state constraints \eqref{SI:IntegralConstraints}, the terms
\begin{align*}
\int_{\Omega} b(x, \bm{u}, \nabla \bm{u}, p, \varphi) \dx \text{ and } \int_{\Omega} K_{i}(x, \bm{u}, \nabla \bm{u}, p, \varphi) \dx
\end{align*}
require no modification when we consider the phase field setting.  For the surface terms such as
\begin{align*}
\int_{\Omega} \frac{1}{2} h(x, \nabla \bm{u}, p, \bm{\nu}_{\varphi}) \dd \abs{\der \varphi}  \text{ and } \int_{\Omega} \frac{1}{2} \LL_{i}(x, \nabla \bm{u}, p) \cdot \bm{\nu}_{\varphi} \dd \abs{\der \varphi}
\end{align*}
arising from the objective functional \eqref{SI:Objective} and the integral state constraints \eqref{SI:IntegralConstraints}, we employ the idea in \cite{GHHKL} to reformulate them.  Assuming the function $h$ is one-homogeneous with respect to its last variable, which is true for the case of the hydrodynamic force \eqref{hydrodynamic}, we \BBB use the vector-valued measure with density $\frac{1}{2} \nabla \varphi$ as an approximation to $\frac{1}{2} \bm{\nu}_{\varphi} \dd \abs{\der \varphi}$, see \cite[\S 3.2]{GHHKL} for more details.  This in turn gives us the phase field approximation
\begin{align*}
\int_{\Omega} \frac{1}{2} h(x, \nabla \bm{u}, p, \nabla \varphi) \dx 
\end{align*}
to the surface objective involving the function $h$ \AAA and from this point onward we will assume that $h$ is one-homogeneous with respect to its last variable.  By a similar argument, we see that
\begin{align*}
\int_{\Omega} \frac{1}{2} \LL_{i}(x, \nabla \bm{u}, p) \cdot \nabla \varphi \dx
\end{align*}
is a phase field approximation of the surface integral constraint involving $\LL_{i}$.
\BBB 

\AAA Recall that in the modified state equations \eqref{PF:state} the porous-medium term $\alpha_{\eps}(\varphi) \bm{u}$ serves to enforce the condition that the velocity $\bm{u}$ in the object region should vanish in the limit $\eps \to 0$.  In earlier works motivated by the paper of Borrvall and Petersson \cite{BP03} the authors of \cite{GarckeHechtStokes,HechtThesis} also added to the phase field objective functional
\begin{align*}
\int_{\Omega} b(x, \bm{u}, \nabla \bm{u}, p, \varphi) + \frac{1}{2} h(x, \nabla \bm{u}, p, \nabla \varphi) + \frac{\gamma}{2c_{0}} \left ( \frac{\eps}{2} \abs{\nabla \varphi}^{2} + \frac{1}{\eps} \Psi(\varphi) \right ) \dx
\end{align*}
a penalization term
\begin{align}\label{alpha:obj}
\int_{\Omega} \frac{1}{2} \hat{\alpha}_{\eps}(\varphi) \abs{\bm{u}}^{2} \dx,
\end{align}
where $\hat{\alpha}_{\eps}$ is a function with similar properties as $\alpha_{\eps}$, i.e., $\hat{\alpha}_{\eps}(1) = 0$ and $\hat{\alpha}_{\eps}(-1) \to \infty$ as $\eps \to 0$.  In fact, in the rigorous analysis of the phase field approximation in Stokes flow, the addition of penalization term \eqref{alpha:obj} to the objective functional does indeed lead to the velocity field vanishing in the object region as $\eps \to 0$ (see \cite[\S 3]{GarckeHechtStokes} and \cite[\S 6.3]{HechtThesis} for more details).  In this paper we consider including both elements in the analytical treatment of the optimization problem.  It is also possible to consider $\hat{\alpha}_{\eps} = 0$, however in this case no rigorous results on the sharp interface limit $\eps \to 0$ are known. \BBB

\AAA
Before we state the phase field optimization problem, let us mention that for the analysis we assume that the function $\alpha_{\eps}$ in the porous-medium term satisfies the properties:
\begin{enumerate}[label=$(\mathrm{A0})$, ref = $\mathrm{A0}$, leftmargin=*]
\item  \label{assump:alpha} $\alpha_{\eps} \in C^{1,1}(\R)$ is non-negative, and there exist constants $s_{a}, s_{b} \in \R$ with $s_{a} \leq -1$ and $s_{b} \geq 1$ such that
\begin{equation}\label{alphaeps:ass}
\begin{aligned}
\alpha_{\eps}(s) = \alpha_{\eps}(s_{a}) \quad \forall s \leq s_{a}, \\
\alpha_{\eps}(s) = \alpha_{\eps}(s_{b}) \quad \forall s \geq s_{b}.
\end{aligned}
\end{equation}
\end{enumerate}
In particular, for arbitrary $\phi$ and its truncation $\tilde{\phi} := \max(s_{a}, \min(s_{b},\phi))$ we see that $\alpha_{\eps}(\phi) = \alpha_{\eps}(\tilde{\phi})$, and so the state equations \eqref{PF:state} for $\phi$ and $\tilde{\phi}$ are equivalent.  Hence, without loss of generality, we now search for optimal design functions $\varphi$ exhibiting $H^{1}(\Omega)$-regularity and satisfies the pointwise bounds $s_{a} \leq \varphi \leq s_{b}$ a.e. in $\Omega$.
\BBB

\AAA Taking into account the above discussions, we arrive at the following phase field approximation to the optimal control problem \eqref{SI:Objective}-\eqref{SI:IntegralConstraints}: \BBB
\begin{equation}\label{OpProblem}
\begin{aligned}
\min_{(\varphi, \bm{u}, p)} \JJ_{\eps}(\varphi, \bm{u}, p) & := \int_{\Omega} \frac{1}{2} \hat{\alpha}_{\eps}(\varphi)  \abs{\bm{u}}^{2}  +  b(x, \bm{u}, \nabla \bm{u}, p, \varphi) \dx \\
& + \int_{\Omega} \frac{1}{2} h(x, \nabla \bm{u}, p, \nabla \varphi) + \frac{\gamma}{2c_{0}} \left ( \frac{1}{\eps} \Psi(\varphi) + \frac{\eps}{2} \abs{\nabla \varphi}^{2} \right ) \dx,
\end{aligned}
\end{equation}
subject to 
\begin{align*}
\varphi & \in \AAA \Phi := \{ f \in H^{1}(\Omega) \, | \, s_{a} \leq f \leq s_{b} \text{ a.e. in } \Omega \} \subset H^{1}(\Omega) \cap L^{\infty}(\Omega), \BBB \\ 
\bm{u} & \in \bH^{1}_{\bm{g},\sigma}(\Omega) := \left \{ \bm{h} \in \bH^{1}(\Omega) \, | \, \div \bm{h} = 0 \text{ in } \Omega \text{ and } \bm{h} = \bm{g} \text{ on } \pd \Omega \right \}, \\
p & \in L^{2}_{0}(\Omega) := \left \{ h \in L^{2}(\Omega) \, \big{|} \, \int_{\Omega} h \dx = 0 \right \}
\end{align*}
satisfying the weak formulation of \eqref{PF:state}:
\begin{align}\label{NSweak}
\int_{\Omega} \alpha_{\eps}(\varphi) \bm{u} \cdot \bm{v} + \mu \nabla \bm{u} \cdot \nabla \bm{v} + (\bm{u} \cdot \nabla) \bm{u} \cdot \bm{v} - p \div \bm{v} - \bm{f} \cdot \bm{v} \dx 
\end{align}
for all $\bm{v} \in \bm{H}^{1}_{0}(\Omega) := \{ \bm{h} \in \bm{H}^{1}(\Omega) \, | \, \bm{h} = \bm{0} \text{ on } \pd \Omega \}$, along with $m_{1}$ equality and $m_{2}$ inequality integral constraints
\begin{align*}
G_{j}(\varphi, \bm{u}, p)&  = 0 \text{ for } 1 \leq j \leq m_{1}, \quad G_{m_{1}+k}(\varphi, \bm{u}, p) \geq 0 \text{ for } 1 \leq k \leq m_{2}, 
\end{align*}
of the form
\begin{align}\label{PF:constraint}
G_{i}(\varphi, \bm{u}, p) & = \int_{\Omega} K_{i}(x, \bm{u}, \nabla \bm{u}, p, \varphi) + \frac{1}{2} \nabla \varphi \cdot \LL_{i}(x, \nabla \bm{u}, p) \dx
\end{align}
for $1 \leq i \leq m_{1}+m_{2}$.

\subsection{Preliminaries on the state equations}
\AAA
Since the porous medium Navier--Stokes equation \eqref{PF:state} have been analyzed in detail in previous works \cite{GarckeHechtNS,HechtThesis}, we recall some useful results in this section.
\BBB

\begin{lemma}[{{\cite[Lem.~4.3]{GHHKL}}}]\label{lem:ExistenceState}
Suppose $\alpha_{\eps} \in C^{1,1}(\R)$ is non-negative and satisfies \eqref{alphaeps:ass}, for every $\varphi \in L^{1}(\Omega)$ there exists at least one pair $(\bm{u}, p) \in \bH^{1}_{\bm{g},\sigma}(\Omega) \times L^{2}_{0}(\Omega)$ such that \eqref{NSweak} is satisfied.  Furthermore, there exists a positive constant $C = C(\mu, \alpha_{\eps}, \bm{f}, \bm{g}, \Omega)$ independent  of $\varphi$ such that
\begin{align}\label{NSest}
\norm{\bm{u}}_{\bH^{1}(\Omega)} + \norm{p}_{L^{2}(\Omega)} \leq C.
\end{align}
\end{lemma}

\AAA
The estimate \eqref{NSest} can be established by testing the weak form \eqref{NSweak} with $\bm{u} - \bm{G}$, where $\bm{G} \in \bm{H}^{1}_{\bm{g}, \sigma}(\Omega)$ is a vector field depending on the inflow/outflow $\bm{g}$ and the domain $\Omega$, and  satisfying certain properties.  Furthermore, this computation also shows that the constant $C$ depends on $\eps$ only through the function $\alpha_{\eps}$.
\BBB

By the above existence result, we can define a set-valued solution operator
\begin{align}\label{solnoper}
\bm{S}_{\eps}(\varphi) := \left \{ (\bm{u},p) \in \bH^{1}_{\bm{g},\sigma}(\Omega) \times L^{2}_{0}(\Omega) \, | \, (\bm{u},p) \text{ satisfies } \eqref{NSweak} \right \}
\end{align}
for any $\varphi \in L^{1}(\Omega)$.  Next, we state a continuity property of the solution operator.
\begin{lemma}[{{\cite[Lem.~4.4 and 4.5]{GHHKL}}}]\label{lem:solnopcts}
Let $(\varphi_{k})_{k \in \N} \subset L^{1}(\Omega)$ be a sequence with corresponding solution $(\bm{u}_{k}, p_{k}) \in \bm{S}_{\eps}(\varphi_{k}) \subset \bH^{1}(\Omega) \times L^{2}(\Omega)$ for each $k \in \N$.  Suppose there exists $\varphi \in L^{1}(\Omega)$ such that
\begin{align*}
\norm{\varphi_{k} - \varphi}_{L^{1}(\Omega)} \to 0 \text{ as } k \to \infty.
\end{align*}
Then, there exists a subsequence, denoted by the same index, and functions $\bm{u} \in \bH^{1}(\Omega)$, $p \in L^{2}(\Omega)$ such that
\begin{align*}
\norm{\bm{u}_{k} - \bm{u}}_{\bH^{1}(\Omega)} \to 0, \quad \norm{p_{k} - p}_{L^{2}(\Omega)} \to 0 \text{ as } k \to \infty,
\end{align*}
with the property that $(\bm{u}, p) \in \bm{S}_{\eps}(\varphi)$.  Furthermore, it holds that
\begin{align*}
\lim_{k \to \infty} \int_{\Omega} \alpha_{\eps}(\varphi_{k}) \abs{\bm{u}_{k}}^{2} \dx = \int_{\Omega} \alpha_{\eps}(\varphi) \abs{\bm{u}}^{2} \dx.
\end{align*}
\end{lemma}

 In general, we do not have uniqueness of solutions to \eqref{NSweak}, but there is a conditional uniqueness result.
\begin{lemma}[{{\cite[Lem. 5]{GarckeHechtNS}}}, {{\cite[Lem. 12.2]{HechtThesis}}}] \label{lem:condUniq}
If there exists $\bm{u} \in \bm{S}_{\eps}(\varphi)$ with
\begin{align}\label{uniqcond}
\norm{\nabla \bm{u}}_{\bm{L}^{2}(\Omega)} < \frac{\mu}{K_{\Omega}}, 
\end{align}
where
\begin{align}\label{defn:KOmega}
K_{\Omega} := \begin{cases} \frac{1}{2} \abs{\Omega}^{\frac{1}{2}} & \text{ for } d = 2, \\
\frac{2 \sqrt{2}}{3} \abs{\Omega}^{\frac{1}{6}} & \text{ for } d = 3.
\end{cases}
\end{align}
Then, $\bm{S}_{\eps}(\varphi) = \{(\bm{u},p)\}$.  That is, there is exactly one solution of \eqref{NSweak} corresponding to $\varphi \in L^{1}(\Omega)$.
\end{lemma}

\AAA
The additional assumption \eqref{uniqcond} on the solution $\bm{u} \in \bm{S}_{\eps}(\varphi)$ to ensure uniqueness of the state equations can be achieved for small data $\bm{f}$ and $\bm{g}$ or with high viscosity $\mu$.  However, there are also settings in which \eqref{uniqcond} can be justified a posteriorly \cite{GarckeHechtNS}.  For the subsequent analysis, more precisely in showing the differentiability of the solution operator $\bm{S}_{\eps}$ and the derivation of the optimality conditions, we require that $\bm{S}_{\eps}$ is a one-to-one mapping.  Hence, throughout the rest of the paper we assume that \eqref{uniqcond} holds.  Alternatively, instead of assuming \eqref{uniqcond}, we can work with an isolated local solution to \eqref{NSweak}, for which the subsequent analysis is valid in a neighbourhood of this isolated local solution.
\BBB

We now state the Fr\'{e}chet differentiability of the solution operator $\bm{S}_{\eps}$.

\begin{lemma}[{\cite[Lem.~4.8]{GHHKL}}]\label{lem:SolnOpDiff}
Let $\e{\varphi} \in H^{1}(\Omega)\, \cap L^{\infty}(\Omega)$ be given such that $\bm{S}_{\eps}(\e{\varphi}) = \{(\e{\bm{u}}, \e{p})\}$.  Then, there exists a neighborhood $N$ of $\e{\varphi}$ in $H^{1}(\Omega) \cap L^{\infty}(\Omega)$ such that for every $\delta \in N$, the solution operator consists of exactly one pair, and we may write $\bm{S}_{\eps} : N \to \bH^{1}(\Omega) \times L^{2}(\Omega)$.  This mapping is differentiable at $\e{\varphi}$ with
\begin{align*}
\der \bm{S}_{\eps}(\e{\varphi})(\delta) =: (\linvel, \linp ) \in \bH^{1}_{0}(\Omega) \times L^{2}_{0}(\Omega),
\end{align*}
where $(\linvel,\linp)$ is the unique solution to 
\AAA
\begin{equation}\label{LinearizedStateSys}
\begin{aligned}
& \int_{\Omega} \alpha_{\eps}'(\varphi_{\eps}) \delta \bm{u}_{\eps} \cdot \bm{v} + \alpha_{\eps}(\varphi_{\eps}) \linvel \cdot \bm{v} + \mu \nabla \linvel \cdot \nabla \bm{v} \dx \\
& \quad + \int_{\Omega} (\linvel \cdot \nabla) \bm{u}_{\eps} \cdot \bm{v} + (\bm{u}_{\eps} \cdot \nabla) \linvel \cdot \bm{v} - \linp \div \bm{v} \dx = 0 \quad \forall \bm{v} \in \bm{H}^{1}_{0}(\Omega),
\end{aligned}
\end{equation}
which is the weak formulation of the following \BBB
the linearized state system
\begin{equation*}
\begin{alignedat}{3}
\alpha_{\eps}'(\e{\varphi}) \delta \e{\bm{u}} + \alpha_{\eps}(\e{\varphi}) \linvel - \mu \Laplace \linvel + (\linvel \cdot \nabla) \e{\bm{u}} + (\e{\bm{u}} \cdot \nabla) \linvel + \nabla \linp & = \bm{0} && \text{ in } \Omega, \\
\div \linvel & = 0 && \text{ in } \Omega, \\
\linvel & = \bm{0} && \text{ on } \pd \Omega.
\end{alignedat}
\end{equation*}
\end{lemma}

\section{Existence of a minimizer}\label{sec:min}
We make the following assumptions for the potential $\Psi$ and the functions $\hat{\alpha}_{\eps}$, $b$, $h$, $K_{i}$, and $\LL_{i}$.
\begin{enumerate}[label=$(\mathrm{A \arabic*})$, ref = $\mathrm{A \arabic*}$, leftmargin=*]
\item \label{assump:Psi} Let $\Psi \in C^{1,1}(\R)$ be a non-negative function such that $\Psi(s) = 0$
if and only if $s = \pm 1$, and that there exist positive constants $c_{1}$, $c_{2}$, $t_{0}$ such
that
\begin{align*}
c_{1} t^{k} \leq \Psi(t) \leq c_{2} t^{k} \quad \forall \abs{t} \geq t_{0}, k \geq 2.
\end{align*}
\item \label{assump:hatalpha} The function $\hat{\alpha} \in C^{1,1}(\R)$ satisfies the same assumptions as $\alpha_{\eps}$, i.e., $\hat{\alpha}_{\eps}$ is non-negative, with $\hat{\alpha}_{\eps}(1) = 0$, $\hat{\alpha}_{\eps}(-1) \to \infty $ as $\eps \to 0$, and fulfills \eqref{alphaeps:ass}.
\item \label{assump:b} The function $b: \Omega \times \R^{d} \times \R^{d \times d} \times \R \times \R  \to \R$ is a Carath\'{e}odory function of the form
\begin{align}\label{b:specificform}
b(x, \bm{w}, \bm{A}, s, t) := B(x, \bm{w}, \bm{A}, s) \; z(x, t),
\end{align}
for some Carath\'{e}odory functions $B: \Omega \times \R^{d} \times \R^{d \times d} \times \R \to \R$, $z : \Omega \times \R \to \R$ and there exist non-negative functions $f_{b} \in L^{1}(\Omega)$, $\{f_{b,i}\}_{i=1}^{4} \subset L^{\infty}(\Omega)$ such that for a.e. $x \in \Omega$ it holds for any $r \geq 0$, $p \geq 2$ in two-dimensions and $2 \leq p \leq 6$ in three-dimensions,
\begin{align*}
\abs{B(x, \bm{w}, \bm{A}, s)} & \leq f_{b}(x) + f_{b,1}(x) \abs{\bm{w}}^{p} + f_{b,2}(x) \abs{\bm{A}}^{2} + f_{b,3}(x) \abs{s}^{2}, \\
\abs{z(x, t)} & \leq  f_{b,4}(x) \abs{t}^{r}, 
\end{align*}
for all $s,t \in \R$, $\bm{w} \in \R^{d}$ and $\bm{A} \in \R^{d \times d}$.
\item \label{assump:h} The function $h: \Omega \times \R^{d \times d} \times \R \times \R^{d}  \to \R$ is a Carath\'{e}odory function that is one-homogeneous with respect to its last variable and there exist non-negative functions $f_{h} \in L^{1}(\Omega)$, $\{f_{h,k}\}_{k=1}^{3} \subset L^{\infty}(\Omega)$ such that for a.e. $x \in \Omega$ it holds,
\begin{align*}
\abs{h(x, \bm{A}, s, \bm{w})} & \leq f_{h}(x) + f_{h,1}(x) \abs{\bm{A}}^{2} + f_{h,2}(x) \abs{s}^{2} + f_{h,3}(x) \abs{\bm{w}}^{2},
\end{align*}
for all $s,t \in \R$, $\bm{w} \in \R^{d}$ and $\bm{A} \in \R^{d \times d}$.
\item \label{assump:Ki} For each $1 \leq i \leq m_{1}$, the function $K_{i}: \Omega \times \R^{d} \times \R^{d \times d} \times \R \times \R \to \R$ is a Carath\'{e}odory function of the form
\begin{align}\label{Ki:specialform}
K_{i}(x, \bm{w}, \bm{A}, s, t) := \mathcal{K}_{i}(x, \bm{w}, \bm{A}, s) \; y_{i}(x,t) + k_{i}(x),
\end{align}
for functions $k_{i} \in L^{1}(\Omega)$ and Carath\'{e}odory functions $\mathcal{K}_{i}: \Omega \times \R^{d} \times \R^{d \times d} \times \R \to \R$, $y_{i}: \Omega \times \R \to \R$ and there exist non-negative functions $z_{1} \in L^{1}(\Omega)$, $z_{2}, z_{3}, z_{4}, z_{5} \in L^{\infty}(\Omega)$ such that for a.e. $x \in \Omega$ it holds for any $r \geq 0$, $p \in [2,\infty)$ in two-dimensions and $p \in [2,6)$ in three-dimensions,
\begin{align*}
\abs{\mathcal{K}_{i}(x, \bm{w}, \bm{A}, s)} & \leq z_{1}(x) + z_{2}(x) \abs{\bm{w}}^{p} + z_{3}(x) \abs{\bm{A}}^{2} + z_{4}(x) \abs{s}^{2}, \\
\abs{y_{i}(x, t)} & \leq z_{5}(x) \abs{t}^{r},
\end{align*}
for all $s,t \in \R$, $\bm{w} \in \R^{d}$ and $\bm{A} \in \R^{d \times d}$.
\item \label{assump:LLi} For each $1 \leq j \leq m_{2}$, the function $\LL_{j} : \Omega \times \R^{d \times d} \times \R \to \R^{d}$ is a Carath\'{e}odory function and there exist non-negative functions $z_{6} \in L^{1}(\Omega)$, $z_{7}, z_{8} \in L^{\infty}(\Omega)$ such that for a.e. $x \in \Omega$ it holds,
\begin{align*}
\abs{\LL_{j}(x, \bm{A}, s)} & \leq z_{6}(x) + z_{7}(x) \abs{\bm{A}} + z_{8}(x) \abs{s},
\end{align*}
for all $s \in \R$ and $\bm{A} \in \R^{d \times d}$.
\item \label{Kad} We assume that the set
\begin{align*}
\K_{ad}  := \Big{\{} \varphi \in \Phi & \text{ with } (\bm{u}, p) = \bm{S}_{\eps}(\varphi) \text{ s.t. }  G_{i}(\varphi, \bm{u}, p) = 0 \text{ for } 1 \leq i \leq m_{1}, \\
& \text{ and }  G_{m_{1}+j}(\varphi, \bm{u}, p) \geq 0 \text{ for } 1 \leq j \leq m_{2} \Big{\}},
\end{align*}
is non-empty, where the phase field integral state constraints $G_{k}$, for $1 \leq k \leq m_{1}+m_{2}$, are of the form \eqref{PF:constraint} and we recall the set $\Phi$ is defined as $\{ f \in H^{1}(\Omega) \, | \, s_{a} \leq f \leq s_{b} \text{ a.e. in } \Omega \}$.
\item \label{assump:functional} The functionals $\BB : H^{1}(\Omega) \times \bH^{1}(\Omega) \times L^{2}(\Omega) \to \R$ and $\HH: H^{1}(\Omega) \times \bH^{1}(\Omega) \times L^{2}(\Omega) \to \R$ defined as
\begin{align*}
\BB(\varphi, \bm{u}, p) & := \int_{\Omega} B(x, \bm{u}, \nabla \bm{u}, p) \; z(x, \varphi) \dx, \\ \HH(\varphi, \bm{u}, p) & := \int_{\Omega} \frac{1}{2} h(x, \nabla \bm{u}, p, \nabla \varphi) \dx
\end{align*}
satisfy $\BB \vert_{\K_{ad} \times \bH^{1}_{\bm{g},\sigma}(\Omega) \times L^{2}_{0}(\Omega)}$ and $\HH \vert_{\K_{ad} \times \bH^{1}_{\bm{g},\sigma}(\Omega) \times L^{2}_{0}(\Omega)}$ are bounded from below, $\BB$ is weakly lower semicontinuous, and for all $\varphi_{n} \rightharpoonup \varphi$ in $H^{1}(\Omega)$, $\bm{u}_{n} \to \bm{u}$ in $\bH^{1}(\Omega)$, $p_{n} \to p$ in $L^{2}(\Omega)$, it holds that
\begin{align*}
\HH(\varphi, \bm{u}, p) \leq \liminf_{n \to \infty} \HH(\varphi_{n}, \bm{u}_{n}, p_{n}).
\end{align*}
\end{enumerate}

The particular forms of $b$ and $K_{i}$ are motivated from the discussions in Sec.~\ref{sec:Intro}, where $z$ and $y$ would typically be functions of the form $\frac{1+\varphi}{2}$, and the function $k_{i}$ would be of the form $D \abs{\Omega}^{-1}$.  \AAA Furthermore, the set $\K_{ad}$ is the set of admissible design functions whose elements satisfy the $m_{1}$ equality integral constraints and $m_{2}$ inequality integral constraints.  While we assume the non-emptiness of $\K_{ad}$ for the general setting here, later in Sec.~\ref{sec:VerifyZK} we show for two examples that the corresponding set $\K_{ad}$ is indeed non-empty.
\BBB 

\AAA The following weakly closed property is useful for showing the existence of minimizers to the optimal control problem  \eqref{OpProblem}-\eqref{PF:constraint}. \BBB

\begin{lemma}\label{lem:Kadweaklyclosed}
Under \eqref{assump:Ki} and \eqref{assump:LLi}, let $\{\varphi_{n}\}_{n \in \N}$ be a sequence in $\K_{ad}$ such that $\varphi_{n} \rightharpoonup \varphi \in H^{1}(\Omega)$ for some $\varphi \in H^{1}(\Omega)$, then $\varphi \in \K_{ad}$.
\end{lemma}
\begin{proof}
\AAA Let $\{\varphi_{n}\}_{n \in \N}$ be a sequence in $\K_{ad}$ with weak limit $\varphi \in H^{1}(\Omega)$.  It suffices to show that if $(\bm{u}, p) \in \bm{S}_{\eps}(\varphi)$, then $G_{i}(\varphi, \bm{u}, p) = 0$ for $1 \leq i \leq m_{1}$ and $G_{m_{1}+j}(\varphi, \bm{u}, p) \geq 0$ for $1 \leq j \leq m_{2}$, which then implies that $\varphi \in \K_{ad}$.
\BBB

Let $\{(\bm{u}_{n}, p_{n})\}_{n \in \N} \subset \bm{H}^{1}_{\bm{g}, \sigma}(\Omega) \times L^{2}_{0}(\Omega)$ be the corresponding solutions to \eqref{PF:state} for $\varphi_{n}$, i.e., for each $n \in \N$, $(\bm{u}_{n}, p_{n}) \in \bm{S}_{\eps}(\varphi_{n})$.  Since $\varphi_{n} \rightharpoonup \varphi \in H^{1}(\Omega)$, by compactness we have strong convergence along subsequences $\varphi_{n_{j}} \to \varphi$ in $L^{p}(\Omega)$ for $p \in [1,\infty)$ in two dimensions and $p \in [1,6)$ in three dimensions.  Consequently, we also have $\varphi_{n_{j}} \to \varphi$ a.e. in $\Omega$ and hence $s_{a} \leq \varphi \leq s_{b}$ a.e. in $\Omega$.  Furthermore, by the assertions of Lem.~\ref{lem:solnopcts}, the corresponding solutions $\{(\bm{u}_{n_{j}}, p_{n_{j}})\}_{j \in \N}$ satisfy $\bm{u}_{n_{j}} \to \bm{u}$ in $\bm{H}^{1}(\Omega)$ and $p_{n_{j}} \to p$ in $L^{2}(\Omega)$ where $(\bm{u},p) \in \bm{S}_{\eps}(\varphi)$.

For each $1 \leq i \leq m_{1}+m_{2}$, by the continuity of $\LL_{i}$ with respect to its second and third variables, it holds that $\LL_{i}(x, \nabla \bm{u}_{n_{j}}, p_{n_{j}}) \to \LL_{i}(x, \nabla \bm{u}, p)$ a.e. in $\Omega$.  Using the growth conditions in \eqref{assump:LLi}, the strong convergences for $\{\bm{u}_{n_{j}}, p_{n_{j}}\}_{j \in \N}$ and the generalized Lebesgue dominated convergence theorem leads to
\begin{align}\label{LLstrongConv}
\LL_{i}(x, \nabla \bm{u}_{n_{j}}, p_{n_{j}}) \rightarrow \LL_{i}(x, \nabla \bm{u}, p) \text{ strongly in } \bm{L}^{2}(\Omega) \text{ as } j \to \infty.
\end{align}
Together with the weak convergence $\nabla \varphi_{n_{j}}$ to $\nabla \varphi$ in $\bm{L}^{2}(\Omega)$, we have 
\begin{align*}
\lim_{j \to \infty} \int_{\Omega} \frac{1}{2} \nabla \varphi_{n_{j}} \cdot \LL_{i}(x, \nabla \bm{u}_{n_{j}}, p_{n_{j}}) \dx = \int_{\Omega} \frac{1}{2} \nabla \varphi \cdot \LL_{i}(x, \nabla \bm{u}, p) \dx.
\end{align*}
Note that $s_{a} \leq \varphi_{n_{j}}, \varphi \leq s_{b}$ a.e. in $\Omega$ for all $j \in \N$, and thus there exists a constant $M > 0$ such that $\sup_{x \in \Omega} \left ( \abs{y_{i}(x, \varphi_{n_{j}})}, \abs{y_{i}(x, \varphi)} \right ) \leq M$ for all $n \in \N$.  Using the splitting
\begin{align*}
& \abs{\int_{\Omega} \mathcal{K}_{i}(x, \bm{u}_{n_{j}}, \nabla \bm{u}_{n_{j}}, p_{n_{j}}) \; y_{i}(x, \varphi_{n_{j}}) - \mathcal{K}_{i}(x, \bm{u}, \nabla \bm{u}, p) \; y_{i}(x, \varphi) \dx} \\
& \quad \leq \abs{\int_{\Omega} \left (\mathcal{K}_{i}(x, \bm{u}_{n_{j}}, \nabla \bm{u}_{n_{j}}, p_{n_{j}}) - \mathcal{K}_{i}(x, \bm{u}, \nabla \bm{u}, p) \right )\; y_{i}(x, \varphi_{n_{j}}) \dx } \\
& \quad + \abs{ \int_{\Omega} \mathcal{K}_{i}(x, \bm{u}, \nabla \bm{u}, p) \; (y_{i}(x, \varphi_{n_{j}}) - y_{i}(x, \varphi)) \dx} =: I_{1} + I_{2},
\end{align*}
we can show that $\lim_{n \to \infty} G_{i}(\varphi_{n_{j}}, \bm{u}_{n_{j}}, p_{n_{j}}) = G_{i}(\varphi, \bm{u}, p)$ once we demonstrate that $I_{1}, I_{2} \to 0$ as $n \to \infty$.  This would then imply that $\varphi \in \K_{ad}$.  Using the growth conditions in \eqref{assump:Ki} for $\mathcal{K}_{i}$, the strong convergences for $\{(\bm{u}_{n_{j}}, p_{n_{j}})\}_{j \in \N}$ and the generalized Lebesgue dominated convergence theorem yields that 
\begin{align*}
\mathcal{K}_{i}(x, \bm{u}_{n_{j}}, \nabla \bm{u}_{n_{j}}, p_{n_{j}}) \to \mathcal{K}_{i}(x, \bm{u}, \nabla \bm{u}, p) \text{ strongly in } L^{1}(\Omega) \text{ as } j \to \infty.
\end{align*}
Then, the assertion that $I_{1} \to 0$ as $j \to \infty$ follows from the above strong convergence in $L^{1}(\Omega)$ and the boundedness of $y_{i}(x, \varphi_{n_{j}})$ in $L^{\infty}(\Omega)$.  Meanwhile, dominating the sequence $\{\mathcal{K}_{i}(x, \bm{u}, \nabla \bm{u}, p) \; y_{i}(x, \varphi_{n_{j}})\}_{j \in \N}$ by the function $\norm{z_{5}}_{L^{\infty}(\Omega)} M \abs{\mathcal{K}_{i}(x, \bm{u}, \nabla \bm{u}, p)} \in L^{1}(\Omega)$, and the application of the usual Lebesgue dominating convergence theorem yields
\begin{align*}
\lim_{j \to \infty} \int_{\Omega} \mathcal{K}_{i}(x, \bm{u}, \nabla \bm{u}, p) \; y_{i}(x, \varphi_{n_{j}}) \dx = \int_{\Omega} \mathcal{K}_{i}(x, \bm{u}, \nabla \bm{u}, p) \; y_{i}(x, \varphi) \dx,
\end{align*}
and hence $I_{2} \to 0$ as $n \to \infty$.
\qed
\end{proof}

We state the existence result for a minimizer of the problem \eqref{OpProblem}-\eqref{PF:constraint}.
\begin{thm}\label{thm:existence}
Under Assumptions \eqref{assump:Psi}-\eqref{assump:functional}, 
there exists at least one minimizer to the problem \eqref{OpProblem}-\eqref{PF:constraint}.
\end{thm}
\begin{proof}
By \eqref{assump:functional}, $(\BB + \HH) \vert_{\K_{ad} \times \bH^{1}_{\bm{g}, \sigma}(\Omega) \times L^{2}_{0}(\Omega)}$
is bounded from below by a constant $C_{0} \in \R$.  Then, by the non-negativity of $\hat{\alpha}_{\eps}$ and $\Psi$, we find that there exists a constant $C_{1} \in \R$ such that $\JJ_{\eps} : \K_{ad} \times \bH^{1}_{\bm{g}, \sigma}(\Omega) \times L^{2}_{0}(\Omega) \to \R$ is bounded from below by $C_{1}$.  Thus, we can choose a minimizing sequence $(\varphi_{n}, \bm{u}_{n}, p_{n})_{n \in \N} \subset \K_{ad} \times \bH^{1}_{\bm{g},\sigma}(\Omega) \times L^{2}_{0}(\Omega)$ such that $(\bm{u}_{n}, p_{n}) \in \bm{S}_{\eps}(\varphi_{n})$ for all $n \in \N$ and
\begin{align*}
\lim_{n \to \infty} \JJ_{\eps}(\varphi_{n}, \bm{u}_{n}, p_{n}) = \inf_{\varphi \in \K_{ad}, (\bm{u}, p) \in \bm{S}_{\eps}(\varphi)} \JJ_{\eps}(\varphi, \bm{u}, p) \geq C_{1} > -\infty.
\end{align*}
Then, for arbitrary $\eta > 0$, there exists $N \in \N$ such that for $n > N$, 
\begin{align*}
C_{0} + \frac{\gamma \eps}{4 c_{0}} \norm{\nabla \varphi_{n}}_{\bm{L}^{2}(\Omega)}^{2} \leq \JJ_{\eps}(\varphi_{n}, \bm{u}_{n}, p_{n}) \leq \inf_{\varphi \in \K_{ad}, (\bm{u}, p) \in \bm{S}_{\eps}(\varphi)} \JJ_{\eps}(\varphi, \bm{u}, p) + \eta.
\end{align*}
The above estimate implies that $\{\varphi_{n}\}_{n \in \N} \subset \K_{ad}$ 
is bounded uniformly in $H^{1}(\Omega) \cap L^{\infty}(\Omega)$.  
Thus, we may choose a subsequence
$(\varphi_{n_{k}})_{k \in \N}$ such that $\varphi_{n_{k}} \to \varphi$
 strongly in $L^{p}(\Omega)$ and almost everywhere in $\Omega$ for $2 \leq p < \infty$ in
 two-dimensions and $2 \leq p < 6$ in three-dimensions.  Furthermore, by Lem.~\ref{lem:Kadweaklyclosed} we also have that $\varphi \in \K_{ad}$, and by Lem.~\ref{lem:solnopcts}, there is a subsequence $(\bm{u}_{n_{k}}, p_{n_{k}})_{n \in \N} \subset \bH^{1}_{\bm{g},\sigma}(\Omega) \times L^{2}_{0}(\Omega)$ such that
\begin{align*}
\lim_{k \to \infty} \norm{\bm{u}_{n_{k}} - \bm{u}}_{\bH^{1}(\Omega)} = 0, \quad \lim_{k \to \infty} \norm{p_{n_{k}} - p}_{L^{2}(\Omega)} = 0,
\end{align*}
for some $(\bm{u}, p) \in \bm{S}_{\eps}(\varphi)$, and
\begin{align*}
\lim_{k \to \infty} \int_{\Omega} \alpha_{\eps}(\varphi_{n_{k}}) \abs{\bm{u}_{n_{k}}}^{2} \dx = \int_{\Omega} \alpha_{\eps}(\varphi) \abs{\bm{u}}^{2} \dx.
\end{align*}
The continuity of $\Psi$ together with the fact that $(\varphi_{n_{k}})_{k \in \N} \subset L^{\infty}(\Omega)$ implies $(\Psi(\varphi_{n_{k}}))_{k \in \N}$ is a bounded sequence in $L^{\infty}(\Omega)$.  The application of the dominated convergence theorem yields that $\Psi(\varphi_{n_{k}})$ converges strongly to $\Psi(\varphi)$ in $L^{1}(\Omega)$ as $k \to \infty$.  Furthermore, by the weak lower semicontinuity assumptions of $\BB$  and $\HH$, and the weak lower semicontinuity of the mapping $\varphi \mapsto \norm{\nabla \varphi}_{\bm{L}^{2}(\Omega)}^{2}$, we find that
\begin{align*}
\JJ_{\eps}(\varphi, \bm{u}, p) \leq \liminf_{k \to \infty} \JJ_{\eps}(\varphi_{n_{k}}, \bm{u}_{n_{k}}, p_{n_{k}}) = \inf_{\phi \in \K_{ad}, (\bm{v}, q) \in \bm{S}_{\eps}(\phi)} \JJ_{\eps}(\phi, \bm{v}, q),
\end{align*}
and so $(\varphi, \bm{u}, p) \in \K_{ad} \times \bH^{1}_{\bm{g}, \sigma}(\Omega) \times L^{2}_{0}(\Omega)$ is a minimizer of \eqref{OpProblem}-\eqref{PF:constraint}.
\qed
\end{proof}

\AAA From this point onwards, for fixed $\eps > 0$, we denote a minimizer to the optimal control problem \eqref{OpProblem}-\eqref{PF:constraint} as $\varphi_{\eps}$ with corresponding unique solution $(\bm{u}_{\eps}, p_{\eps})$ to the state equation \eqref{NSweak}.
\BBB

\section{Optimality conditions}\label{sec:opt}

We use the notation $\der_{j}f$ to denote the partial derivative of $f$ with respect to its $j$th variable.  Furthermore, the notation $\abs{\der_{(i,j)} f} \leq P$ means that the partial derivatives $\der_{i} f$ and $\der_{j} f$ satisfy $\abs{\der_{i}f} \leq P$ and $\abs{\der_{j}f} \leq P$.  To obtain optimality conditions, we make the following assumptions on the differentiability of $B$, $z$, $h$, $\mathcal{K}_{i}$, $y_{i}$, and $\LL_{i}$.

\begin{enumerate}[label=$(\mathrm{B \arabic*})$, ref = $\mathrm{B \arabic*}$, leftmargin=*]
\item \label{assump:bhdiff} In addition to \eqref{assump:b} assume further that $x \mapsto B(x, \bm{w}, \bm{A}, s)$, $x \mapsto z(x, t)$ and $x \mapsto h(x, \bm{A}, s, \bm{w})$ belong to $W^{1,1}(\Omega)$ for all $\bm{w} \in \R^{d}$, $\bm{A} \in \R^{d \times d}$, $s,t \in \R$, and the partial derivatives 
\begin{align*}
& \der_{2}B(x, \cdot, \bm{A}, s), \; \der_{3}B(x, \bm{w}, \cdot, s), \; \der_{4}B(x, \bm{w}, \bm{A}, \cdot), \; \der_{2} z(x, \cdot), \\
& \der_{2}h(x, \cdot, s, \bm{w}), \; \der_{3}h(x, \bm{A}, \cdot, \bm{w}), \; \der_{4}h(x, \bm{A}, s, \cdot )
\end{align*}
exist for all $\bm{w} \in \R^{d}$, $s \in \R$, $\bm{A} \in \R^{d \times d}$, and a.e. $x \in \Omega$ as Carath\'{e}odory functions with
\begin{align}
\notag \abs{\der_{2} B(x, \bm{w}, \bm{A}, s)} & \leq \tilde{c}(x) + \tilde{b}_{1}(x) \abs{\bm{w}}^{p-1} + \tilde{b}_{2}(x) \abs{\bm{A}} + \tilde{b}_{3}(x) \abs{s}, \\
\notag \abs{\der_{(3,4)} B(x, \bm{w}, \bm{A}, s)} & \leq \tilde{a}(x) + \tilde{b}_{1}(x) \abs{\bm{w}}^{p/2} + \tilde{b}_{2}(x) \abs{\bm{A}} + \tilde{b}_{3}(x) \abs{s},  \\
\notag \abs{\der_{2} z(x, t)} & \leq \tilde{b}_{1}(x), \\
\abs{\der_{(2,3,4,5)} h(x, \bm{A}, s, \bm{w})} & \leq \tilde{a}(x) + \tilde{b}_{1}(x) \abs{\bm{A}} + \tilde{b}_{2}(x) \abs{s} + \tilde{b}_{3}(x) \abs{\bm{w}},\label{equ:partialderivativesh}
\end{align}
for some non-negative functions $\tilde{a} \in L^{2}(\Omega)$, $\tilde{c} \in L^{\frac{p}{p-1}}(\Omega)$, $\tilde{b}_{1}, \tilde{b}_{2}, \tilde{b}_{3} \in L^{\infty}(\Omega)$, where $p \geq 2$ in two dimensions and $p \in [2,6]$ in three dimensions.
\item \label{assump:KiLLidiff} For each $1 \leq i \leq m_{1}+m_{2}$, in addition to \eqref{assump:Ki} assume further that  $x \mapsto \mathcal{K}_{i}(x, \bm{w}, \bm{A}, s)$, $x \mapsto k_{i}(x)$, $x \mapsto y_{i}(x, t)$, and $x \mapsto \LL_{i}(x, \bm{A}, s)$ belong to $W^{1,1}(\Omega)$ for all $\bm{w} \in \R^{d}$, $\bm{A} \in \R^{d \times d}$, $s, t \in \R$ and the partial derivatives
\begin{align*}
& \der_{2} \mathcal{K}_{i}(x, \cdot, \bm{A}, s), \; \der_{3}\mathcal{K}_{i}(x, \bm{w}, \cdot, s), \; \der_{4}\mathcal{K}_{i}(x, \bm{w}, \bm{A}, \cdot), \\
& \der_{2} y(x, \cdot), \; \der_{2} \LL_{1}(x, \cdot, s), \; \der_{3} \LL_{1}(x, \bm{A}, \cdot)
\end{align*}
exist for all $\bm{w} \in \R^{d}$, $s \in \R$, $\bm{A} \in \R^{d \times d}$, and a.e. $x \in \Omega$ as Carath\'{e}odory functions.  Moreover, we assume that 
\begin{align*}
\abs{\der_{2} \mathcal{K}_{i}(x, \bm{w}, \bm{A}, s)} & \leq \tilde{c}(x) + \tilde{b}_{1}(x) \abs{\bm{w}}^{p-1} + \tilde{b}_{2}(x) \abs{\bm{A}} + \tilde{b}_{3}(x) \abs{s}, \\
\abs{\der_{(3,4)} \mathcal{K}_{i}(x, \bm{w}, \bm{A}, s)} & \leq \tilde{a}(x) + \tilde{b}_{1}(x) \abs{\bm{w}}^{p/2} + \tilde{b}_{2}(x) \abs{\bm{A}} + \tilde{b}_{3}(x) \abs{s}, \\
\abs{\der_{(2,3)} \LL_{i}(x, \bm{A}, s)} & \leq \tilde{b}_{1}(x), \\
\abs{\der_{2} y_{i}(x, t)} & \leq \tilde{b}_{1}(x), 
\end{align*}
for some non-negative functions $\tilde{a} \in L^{2}(\Omega)$, $\tilde{c} \in L^{\frac{p}{p-1}}(\Omega)$ and $\tilde{b}_{1}, \tilde{b}_{2}, \tilde{b}_{3} \in L^{\infty}(\Omega)$, where $p \geq 2$ in two dimensions and $p \in [2,6]$ in three dimensions.
\end{enumerate}

Under \eqref{assump:bhdiff} and using \cite[\S 4.3.3]{Troltzsch} or \cite[Thm.~1 and 3]{Goldberg92}, the Nemytskii operators
\begin{align*}
(L^{2}(\Omega))^{d \times d} \ni \bm{A} & \mapsto \der_{2}h( \cdot, \bm{A}, s, \bm{w}) \in L^{2}(\Omega) \quad \forall s \in L^{2}(\Omega), \bm{w} \in (L^{2}(\Omega))^{d}, \\
L^{2}(\Omega) \ni s & \mapsto \der_{3}h( \cdot, \bm{A}, s, \bm{w}) \in L^{2}(\Omega) \quad \forall \bm{A} \in (L^{2}(\Omega))^{d \times d}, \bm{w} \in (L^{2}(\Omega))^{d}, \\
(L^{2}(\Omega))^{d} \ni \bm{w} & \mapsto \der_{4}h( \cdot, \bm{A}, s, \bm{w}) \in L^{2}(\Omega) \quad \forall \bm{A} \in (L^{2}(\Omega))^{d \times d}, s \in L^{2}(\Omega),
\end{align*}
are well-defined and the operator
\begin{align*}
(L^{2}(\Omega))^{d \times d} \times L^{2}(\Omega) \times (L^{2}(\Omega))^{d} \ni (\bm{A}, s, \bm{w}) \mapsto h( \cdot, \bm{A}, s, \bm{w}) \in L^{1}(\Omega)
\end{align*}
is continuously Fr\'{e}chet differentiable (see \cite[Thm. 7]{Goldberg92} or \cite[\S 4.3.3]{Troltzsch} with $p = r = 2$ and $ q = 1$).  Hence, we find that
\begin{align*}
\HH : \left ( H^{1}(\Omega) \cap L^{\infty}(\Omega) \right ) \times \bm{H}^{1}(\Omega) \times L^{2}(\Omega) & \to \R \\
(\varphi, \bm{u}, p) & \mapsto \int_{\Omega} \frac{1}{2} h(x, \nabla \bm{u}, p, \nabla \varphi) \dx
\end{align*}
is continuously Fr\'{e}chet differentiable with derivative at $(\varphi_{\eps}, \bm{u}_{\eps}, p_{\eps})$ in the direction $(\eta, \bm{v}, s)$ given as
\begin{equation}\label{FrechDerivativemathcalH}
\begin{aligned}
& \der \HH (\varphi_{\eps}, \bm{u}_{\eps}, p_{\eps})(\eta, \bm{v}, s) \\
& \quad = \int_{\Omega} \frac{1}{2} (\der_{2}h, \der_{3}h, \der_{4}h)\mid_{(x, \nabla \bm{u}_{\eps}, p_{\eps}, \nabla \varphi_{\eps})} \cdot \, (\nabla \bm{v}, s, \nabla \eta) \dx.
\end{aligned}
\end{equation}
Here we use the notation
\begin{align*}
& (\der_{2}h, \der_{3}h, \der_{4}h)\mid_{(x, \nabla \bm{u}_{\eps}, p_{\eps}, \nabla \varphi_{\eps})} \cdot \, (\nabla \bm{v}, s, \nabla \eta) \\
& \quad := (\der_{2}h) : \nabla \bm{v} + (\der_{3}h) s + (\der_{4}h) \cdot \nabla \eta,
\end{align*}
where the partial derivatives are evaluated at $(x, \nabla \bm{u}_{\eps}, p_{\eps}, \nabla \varphi_{\eps})$.  With a similar argument, the mappings
\begin{align*}
\BB : (H^{1}(\Omega) \cap L^{\infty}(\Omega)) \times \bH^{1}(\Omega) \times L^{2}(\Omega) & \to \R \\
 (\varphi, \bm{u}, p) & \mapsto \int_{\Omega} B(x, \bm{u}, \nabla \bm{u}, p) \; z (x,\varphi) \dx, \\
G_{i} : (H^{1}(\Omega) \cap L^{\infty}(\Omega)) \times \bH^{1}(\Omega) \times L^{2}(\Omega) & \to \R \\
(\varphi, \bm{u}, p) & \mapsto \int_{\Omega} \mathcal{K}_{i}(x, \bm{u}, \nabla \bm{u}, p) \; y_{i}(x, \varphi) \dx \\
& + \int_{\Omega} k_{i}(x) + \frac{1}{2} \nabla \varphi \cdot \LL_{i}(x, \nabla \bm{u}, p) \dx,
\end{align*}
for $1 \leq i \leq m_{1}+m_{2}$, are continuously Fr\'{e}chet differentiable, with derivatives at $(\varphi_{\eps}, \bm{u}_{\eps}, p_{\eps})$ in the direction $(\eta, \bm{v}, s)$ given as
\begin{align}
\notag & \der \BB(\varphi_{\eps}, \bm{u}_{\eps}, p_{\eps}) (\eta, \bm{v}, s) \\
& \quad = \int_{\Omega} z(x, \varphi_{\eps}) \; (\der_{2}B, \der_{3}B, \der_{4}B) \vert_{(x, \bm{u}_{\eps}, \nabla \bm{u}_{\eps}, p_{\eps})} \cdot (\bm{v}, \nabla \bm{v}, s) \dx \label{FDmathcalB}  \\
\notag &  \quad \quad + \int_{\Omega} B(x, \bm{u}_{\eps}, \nabla \bm{u}_{\eps}, p) \; \der_{2}z(x, \varphi_{\eps}) \eta \dx,  \\
\notag & \der G_{i} (\varphi_{\eps}, \bm{u}_{\eps}, p_{\eps})(\eta, \bm{v}, s) \\
\notag & \quad = \int_{\Omega} y_{i}(x, \varphi_{\eps}) \; (\der_{2} \mathcal{K}_{i}, \der_{3} \mathcal{K}_{i}, \der_{4} \mathcal{K}_{i} ) \vert_{(x, \bm{u}_{\eps}, \nabla \bm{u}_{\eps}, p_{\eps})} \cdot (\bm{v}, \nabla \bm{v}, s) \dx \\
 & \quad + \int_{\Omega} \mathcal{K}_{i}(x, \bm{u}_{\eps}, \nabla \bm{u}_{\eps}, p_{\eps}) \; \der_{2} y_{i}(x, \varphi_{\eps}) \eta  \dx \label{FD:Gi} \\
\notag & \quad + \frac{1}{2} \int_{\Omega} \nabla \eta \cdot \LL_{i}(x, \nabla \bm{u}_{\eps}, p_{\eps}) + \nabla \varphi_{\eps} \cdot \left( (\der_{2} \LL_{i}, \der_{3} \LL_{i}) \vert_{(x, \nabla \bm{u}_{\eps}, p_{\eps})} \cdot (\nabla \bm{v}, s) \right ) \dx.
\end{align}

\subsection{Fr\'{e}chet differentiability of the objective functional}\label{sec:FrechetDiffjeps}
Due to the well-posedness of the state equations, we may now write the problem \eqref{OpProblem}-\eqref{PF:constraint} as a minimizing problem for a reduced objective functional defined on an open set in $H^{1}(\Omega) \cap L^{\infty}(\Omega)$ with the help of Lem.~\ref{lem:SolnOpDiff}.  Let $(\e{\varphi}, \e{\bm{u}}, \e{p}) \in \K_{ad} \times \bH^{1}_{\bm{g}, \sigma}(\Omega) \times L^{2}_{0}(\Omega)$ denote a minimizer of \eqref{OpProblem}-\eqref{PF:constraint}, obtained from Thm.~\ref{thm:existence}.  By Lem.~\ref{lem:SolnOpDiff}, there exists a neighborhood $N \subset H^{1}(\Omega) \cap L^{\infty}(\Omega)$ of $\e{\varphi}$ such that for every $\psi \in N$, \eqref{NSweak} is uniquely solvable.  We define the reduced functional $\e{j}: N \to \R$ by
\begin{align*}
\e{j}(\psi) := \e{\JJ}(\psi, \bm{S}_{\eps}(\psi))  \text{ for all } \psi \in N.
\end{align*}
We now show that, as a mapping from $\AAA N \subset \BBB H^{1}(\Omega) \cap L^{\infty}(\Omega) \to \R$, $\e{j}$ is Fr\'{e}chet differentiable at $\e{\varphi}$.  As Lem.~\ref{lem:SolnOpDiff} guarantees the Fr\'{e}chet differentiability of the solution operator $\bm{S}_{\eps}(\e{\varphi})$ as a mapping from $\AAA N \BBB $ to $\bH^{1}(\Omega) \times L^{2}(\Omega)$, we focus on the dependence of $\e{\JJ}$ on the first variable.

\AAA Fix $\varphi \in H^{1}(\Omega)$, \BBB then by \eqref{assump:hatalpha}, $\hat{\alpha}_{\eps}$ and $\hat{\alpha}'_{\eps}$ are uniformly bounded and so
\begin{align*}
L^{6}(\Omega) \ni q \mapsto \hat{\alpha}'_{\eps}(\varphi) q \in L^{6}(\Omega)
\end{align*} 
is a well-defined mapping from $H^{1}(\Omega) \subset L^{6}(\Omega)$ to $L^{6}(\Omega)$.  By \cite[\S 4.3.3]{Troltzsch}, we see that $\hat{\alpha}_{\eps}$ defines a Fr\'{e}chet differentiable Nemytskii operator as a mapping from $L^{6}(\Omega)$ to $L^{3}(\Omega)$.  Meanwhile, \AAA the assumption $\Psi \in C^{1,1}(\R)$ \BBB and \cite[Lem.~4.12]{Troltzsch} imply that $\Psi(\varphi)$ is continuously Fr\'{e}chet differentiable Nemytskii operator as a mapping from $L^{\infty}(\Omega)$ to $L^{\infty}(\Omega)$.  Combined with the Fr\'{e}chet differentiability of the mapping $H^{1}(\Omega) \ni \varphi \mapsto \int_{\Omega} \abs{\nabla \varphi}^{2} \dx$, $\BB$ and $\HH$, we obtain that $\e{j} : N \to \R$ is Fr\'{e}chet differentiable.

\subsection{Existence of Lagrange multipliers}\label{sec:LMexist}

To show the existence of Lagrange multipliers for the integral constraints, we make use of the Zowe--Kurcyusz constraint qualification (ZKCQ), see \cite{ZoweKurcyusz} and \cite[\S 6.1.2]{Troltzsch} for more details.  For this purpose, we introduce the notation
\begin{align*}
\Y & := \R^{m_{1}+m_{2}}, \\
\K & := \left \{ \bm{y} \in Y \, | \, y_{i} = 0, y_{j} \geq 0 \text{ for } 1 \leq i \leq m_{1}, m_{1}+1 \leq j \leq m_{1} + m_{2} \right \} \subset \Y , \\
\AAA \GG_{i}(\varphi) \BBB & \AAA := G_{i}(\varphi, \bm{S}_{\eps}(\varphi)) \text{ for } 1 \leq i \leq m_{1}+m_{2}, \BBB \\
\bm{g}(\varphi) & := \left ( \GG_{1}(\varphi), \dots, \GG_{m_{1}}(\varphi), \GG_{m_{1}+1}(\varphi), \dots,  \GG_{m_{1}+m_{2}}(\varphi) \right ),
\end{align*}
and recall the set
\begin{align*}
\Phi = \{ f \in H^{1}(\Omega) \, | \, s_{a} \leq f \leq s_{b} \text{ a.e. in } \Omega \}.
\end{align*}
Then, $\Phi$ is a closed convex subset of $H^{1}(\Omega)$ and $\K$ is a closed convex cone in $\Y$ with vertex at the origin, i.e., $\delta_{1} \K + \delta_{2} \K \subset \K$ for $\delta_{1}, \delta_{2} > 0$.  In the notation of \cite{ZoweKurcyusz}, we introduce the sets
\begin{align*}
\Phi(\e{\varphi}) & = \left \{ \beta (\varphi - \e{\varphi}) \, | \, \varphi \in \Phi, \beta \geq 0 \right \}, \\
\K(\bm{g}(\e{\varphi})) & = \left \{ \bm{\eta} - \beta \bm{g}(\e{\varphi}) \, | \, \bm{\eta} \in \K, \beta \geq 0\right \}.
\end{align*}
\AAA Fix $1 \leq i \leq m_{1}+m_{2}$ and an arbitrary function $\zeta \in \Phi$.  Convexity of $\Phi$ implies that $\e{\varphi} + t(\zeta - \e{\varphi}) \in \Phi$ for sufficiently small values of $t$.  Then, denoting the linearized state variables associated to $\delta = \zeta - \e{\varphi}$ as $(\linvel, \linp)$ (see Lem.~\ref{lem:SolnOpDiff}), the mapping $\phi \mapsto \GG_{i}(\phi) = G_{i}(\phi, \bm{S}_{\eps}(\phi))$ is continuously Fr\'{e}chet differentiable at $\varphi_{\eps}$ with derivative in the direction $\zeta - \e{\varphi}$ given as (see also \eqref{FD:Gi})
\begin{equation}\label{GGFdiff}
\begin{aligned}
& \der \GG_{i}(\e{\varphi})(\zeta - \e{\varphi}) \\
& \quad = \int_{\Omega} (\der_{2} K_{i}, \der_{3} K_{i}, \der_{4} K_{i}, \der_{5} K_{i}) \cdot (\linvel, \nabla \linvel, \linp, \zeta - \e{\varphi}) \dx \\
& \quad \quad + \frac{1}{2} \int_{\Omega} \nabla (\zeta - \e{\varphi}) \cdot \LL_{i}(x, \nabla \e{\bm{u}}, \e{p}) \dx \\
& \quad \quad + \frac{1}{2} \int_{\Omega} \nabla \e{\varphi} \cdot \left [ (\der_{2} \LL_{i}, \der_{3} \LL_{i}) \vert_{(x, \nabla \e{\bm{u}}, \e{p})} \cdot (\nabla \linvel, \linp) \right ] \dx,
\end{aligned}
\end{equation}
where $\der_{2}K_{i}$, $\der_{3} K_{i}$, $\der_{4} K_{i}$ and $\der_{5} K_{i}$ are evaluated at $(x, \e{\bm{u}}, \nabla \e{\bm{u}}, \e{p}, \e{\varphi})$.  Then, it holds that
\begin{align*}
\bm{g}'(\varphi_{\eps})(\zeta - \e{\varphi}) = (\der \GG_{1}(\varphi_{\eps})(\zeta - \e{\varphi}), \dots, \der \GG_{m_{1}+m_{2}}(\varphi_{\eps})(\zeta - \e{\varphi})).
\end{align*}
\BBB
The existence of bounded Lagrange multipliers $\bm{\lambda} := (\lambda_{1}, \dots, \lambda_{m_{1}+m_{2}}) \in \K^{+} := \{ \bm{y} \in \Y \, | \, \bm{y} \cdot \bm{\eta} = 0 \; \forall \bm{\eta} \in \K \}$ satisfying
\begin{align*}
\bm{\lambda} \cdot  \bm{g}(\e{\varphi}) = 0, \text{ and } \inner{\der \e{j}(\e{\varphi}) \AAA + \BBB \bm{\lambda} \cdot \bm{g}'(\e{\varphi})}{\zeta - \e{\varphi}} \geq 0 \quad \forall \zeta \in \Phi, 
\end{align*}
where $\inner{\cdot}{\cdot}$ denotes the duality pairing between $H^{1}(\Omega)$ and its dual, follows if $\varphi_{\eps}$ is a regular point in the sense of \cite{ZoweKurcyusz}, or equivalently the so-called Zowe--Kurcyusz constraint qualification 
\begin{align}\label{ZoweCondition}
\Y = g'(\e{\varphi}) \Phi(\e{\varphi}) - \K(g(\e{\varphi}))
\end{align}
has to hold.  We now make the following assumption:  
\begin{enumerate}[label=$(\mathrm{C \arabic*})$, ref = $\mathrm{C \arabic*}$, leftmargin=*]
\item \label{assump:ZK} For any $\bm{z} \in \Y = \R^{m_{1}+m_{2}}$, there exists a function $\psi_{*} \in \Phi$, vectors $\bm{\tau} \in \Y$, $\bm{\xi}, \bm{\eta} \in \R^{m_{2}}$ such that $\tau_{i} \geq 0$, $\xi_{j} \geq 0$, $\eta_{j} \geq 0$ for $1 \leq i \leq m_{1}+m_{2}$ and $1 \leq j \leq m_{2}$, and
\begin{equation*}
\begin{alignedat}{3}
z_{i} & = \tau_{i} \der \GG_{i}(\e{\varphi})(\psi_{*} - \e{\varphi}), && \text{ for } 1 \leq i \leq m_{1} \\
z_{m_{1}+j} & = \tau_{m_{1}+j} \der \GG_{m_{1} + j}(\e{\varphi})(\psi_{*} - \e{\varphi}) - \eta_{j} + \xi_{j} \GG_{m_{1}+j}(\e{\varphi}), && \text{ for } 1 \leq j \leq m_{2}.
\end{alignedat}
\end{equation*}
\end{enumerate}
Then, under \eqref{assump:ZK} and using \cite[Thm.~3.1 and 4.1]{ZoweKurcyusz} there exist $\lambda_{1}, \dots, \lambda_{m_{1}} \in \R$ and $\lambda_{m_{1}+1}, \dots, \lambda_{m_{1}+m_{2}} \in \R_{\geq 0}$ such that
\begin{equation}\label{GradientIneq}
\begin{aligned}
\der \e{j}(\e{\varphi})(\zeta - \e{\varphi}) & + \sum_{i=1}^{m_{1}} \lambda_{i} \der \GG_{i}(\e{\varphi})(\zeta - \e{\varphi}) \\
& + \sum_{j=1}^{m_{2}} \lambda_{m_{1}+j} \der \GG_{m_{1}+j}(\e{\varphi})(\zeta - \e{\varphi}) \geq 0 \quad \forall \zeta \in \Phi
\end{aligned}
\end{equation}
holds with the following complementary slackness conditions for the inequality constraints
\begin{align}\label{Slackness}
\lambda_{m_{1}+j} \GG_{m_{1}+j}(\e{\varphi}) = 0 \text{ for } 1 \leq j \leq m_{2}.
\end{align}

We mention that \eqref{ZoweCondition} is equivalent (see \cite[\S 3]{ZoweKurcyusz} and \cite[Thm. 1.56]{HPUU}) to the following interior point/linearized Slater condition (which is also commonly known as the Robinson regularity condition \cite{Robinson}):
\begin{align*}
\bm{0} \in \mathrm{int} \left ( \bm{g}(\e{\varphi}) + \bm{g}'(\e{\varphi})(\Phi - \e{\varphi}) - \K \right ).
\end{align*}

\subsection{Adjoint system}
We now introduce the Lagrangian $\mathbb{L} : (H^{1}(\Omega) \cap L^{\infty}(\Omega)) \times \bH^{1}(\Omega) \times L^{2}(\Omega) \times \bH^{1}_{0}(\Omega) \times L^{2}(\Omega) \to \R$ as
\begin{align*}
& \mathbb{L}(\varphi, \bm{u}, p, \bm{q}, \pi) \\
& \quad := \int_{\Omega} \frac{1}{2} \AAA \hat{\alpha}_{\eps}(\varphi) \BBB \abs{\bm{u}}^{2} + \frac{\gamma}{2c_{0}} \left ( \frac{1}{\eps} \Psi(\varphi) + \frac{\eps}{2} \abs{\nabla \varphi}^{2} \right ) \dx \\
& \quad \quad + \int_{\Omega} b(x, \bm{u}, \nabla \bm{u}, p, \varphi) + \frac{1}{2} h(x, \nabla \bm{u}, p, \nabla \varphi)\dx \\
& \quad \quad - \int_{\Omega} \alpha_{\eps}(\varphi) \bm{u} \cdot \bm{q} + \mu \nabla \bm{u} \cdot \nabla \bm{q} + (\bm{u} \cdot \nabla) \bm{u} \cdot \bm{q} - p \div \bm{q} - \bm{f} \cdot \bm{q} - \pi \div \bm{u}\dx \\
&  \quad \quad + \int_{\Omega} \sum_{i=1}^{m_{1}+m_{2}}\lambda_{i} \left ( K_{i}(x, \bm{u}, \nabla \bm{u}, p, \varphi) + \frac{1}{2} \nabla \varphi \cdot \LL_{i}(x, \nabla \bm{u}, p) \right ) + \theta p \dx
\end{align*}
where $\lambda_{i}$ is the Lagrange multiplier for the integral constraint $\GG_{i}(\varphi)$ and $\theta$ is a Lagrange multiplier for the constraint $\int_{\Omega} p \dx = 0$ for the pressure.  A formal computation of $\der_{\bm{u}} \mathbb{L}$ and $\der_{p} \mathbb{L}$ yields the following adjoint system \AAA for the minimizer $\e{\varphi}$: \BBB
\begin{subequations}\label{AdjointSystem}
\begin{alignat}{3}
\notag \alpha_{\eps} (\e{\varphi}) &\e{\bm{q}}  - \mu \div (\nabla \e{\bm{q}} + (\nabla \e{\bm{q}})^{\top}) + (\nabla \e{\bm{u}})^{\top} \e{\bm{q}} - (\e{\bm{u}} \cdot \nabla) \e{\bm{q}} + \nabla \e{\pi}  \\
\notag & = \AAA \hat{\alpha}_{\eps} \BBB (\e{\varphi}) \e{\bm{u}} + \der_{2}b - \div (\der_{3}b + \tfrac{1}{2} \der_{2}h) \\
 & + \sum_{i=1}^{m_{1}+m_{2}} \left ( \lambda_{i} \der_{2} K_{i} -   \div \left (\lambda_{i} \left ( \der_{3} K_{i} + \tfrac{1}{2} \nabla \e{\varphi} \cdot \der_{2} \LL_{i} \right )\right ) \right )  && \text{ in } \Omega, \\
\div \e{\bm{q}} & = -\der_{4}b - \tfrac{1}{2} \der_{3} h - \theta - \sum_{i=1}^{m_{1}+m_{2}} \left ( \lambda_{i} \der_{4} L_{i} +  \tfrac{1}{2} \lambda_{i} \nabla \e{\varphi} \cdot \der_{3} \LL_{1,i} \right )   && \text{ in } \Omega, \label{Adjoint:div} \\
\e{\bm{q}} & = \bm{0} && \text{ on } \pd \Omega,
\end{alignat}
\end{subequations}
where $\der_{(2,3,4)}b$ are evaluated at $(x, \e{\bm{u}}, \nabla \e{\bm{u}}, \e{p}, \e{\varphi})$, $\der_{(2,3)}h$ are evaluated at $(x, \nabla \e{\bm{u}}, \e{p}, \nabla \e{\varphi})$, $\der_{(2,3,4)}K_{i}$ are evaluated at $(x, \e{\bm{u}}, \nabla \e{\bm{u}}, \e{p}, \e{\varphi})$, and $\der_{(2,3)} \LL_{i}$ are evaluated at $(x, \nabla \e{\bm{u}}, \e{p})$, and upon integrating the divergence equation for $\e{\bm{q}}$, we obtain 
\begin{equation}
\label{LM:theta}
\begin{aligned}
\theta & = \frac{1}{\abs{\Omega}} \int_{\Omega} -\der_{4}b - \tfrac{1}{2} \der_{3}h - \sum_{i=1}^{m_{1}+m_{2}} \left ( \lambda_{i} \der_{4} K_{i} + \tfrac{1}{2} \lambda_{i}  \nabla \e{\varphi} \cdot \der_{3} \LL_{i} \right ) \dx.
\end{aligned}
\end{equation}
Let us also recall from \eqref{b:specificform} and \eqref{Ki:specialform} that
\begin{align*}
\der_{(2,3,4)} b(x, \bm{u}, \nabla \bm{u}, p, \varphi) & = z(x, \varphi) \; \der_{(2,3,4)} B(x, \bm{u}, \nabla \bm{u}, p), \\
\der_{(2,3,4)} K_{i}(x, \bm{u}, \nabla \bm{u}, p, \varphi) & = y_{i}(x, \varphi) \; \der_{(2,3,4)} \mathcal{K}_{i}(x, \bm{u}, \nabla \bm{u}, p).
\end{align*}
We now show that the adjoint system is well-posed.
\begin{lemma}
Let $\e{\varphi} \in H^{1}(\Omega) \cap L^{\infty}(\Omega)$ \AAA be the minimizer obtained from Thm.~\ref{thm:existence} and $(\e{\bm{u}}, \e{p}) = \bm{S}_{\eps}(\e{\varphi})$.  Furthermore, let $\{\lambda_{i}\}_{i=1}^{m_{1}+m_{2}}$ be the Lagrange multipliers associated to the integral state constraints $\{\GG_{i}(\e{\varphi})\}_{i=1}^{m_{1}+m_{2}}$. \BBB  Then, under \eqref{assump:bhdiff}, \eqref{assump:KiLLidiff} and \eqref{assump:ZK}, there exists a unique weak solution pair $(\e{\bm{q}}, \e{\pi}) \in \bH^{1}_{0}(\Omega) \times L^{2}(\Omega)$ to the adjoint system \eqref{AdjointSystem} \AAA in the following sense 
\begin{equation}\label{weak:adjoint}
\begin{aligned}
& \int_{\Omega} \alpha_{\eps}(\varphi_{\eps}) \e{\bm{q}} \cdot \bm{v} + \mu ( \nabla \e{\bm{q}} + (\nabla \e{\bm{q}})^{\top}) \cdot \nabla \bm{v} + (\nabla \e{\bm{u}})^{\top} \e{\bm{q}} \cdot \bm{v} - (\e{\bm{u}} \cdot \nabla) \e{\bm{q}} \cdot \bm{v} \dx \\
& \quad = \int_{\Omega} \hat{\alpha}_{\eps}(\e{\varphi}) \e{\bm{u}} \cdot \bm{v} + (\der_{3} b + \tfrac{1}{2} \der_{2}h) \cdot \nabla \bm{v} + \der_{2} b \cdot \bm{v} \dx \\
& \quad \quad + \int_{\Omega} \sum_{i=1}^{m_{1}+m_{2}} \lambda_{i} (\der_{2} K_{i} \cdot \bm{v} + (\tfrac{1}{2} \nabla \e{\varphi} \cdot \der_{2} \LL_{i} + \der_{3} K_{i}) \cdot \nabla \bm{v}) \dx 
\end{aligned}
\end{equation}
for all $\bm{v} \in \bH^{1}_{0,\sigma}(\Omega)$.
\BBB
\end{lemma}

\begin{proof}
\AAA For convenience, we use the notation
\begin{equation}\label{divqeps}
\begin{aligned}
g := -\der_{4}b - \tfrac{1}{2} \der_{3}h - \theta - \sum_{i=1}^{m_{1}+m_{2}} \lambda_{i} \left ( \der_{4} K_{i} + \tfrac{1}{2} \nabla \e{\varphi} \cdot \der_{3} \LL_{i} \right ),
\end{aligned}
\end{equation}
so that \eqref{Adjoint:div} reads as $\div \e{\bm{q}} = g$ in $\Omega$.
\BBB
Then, by \eqref{assump:bhdiff}, \eqref{assump:KiLLidiff} and \eqref{LM:theta}, we see that \AAA $g$ belongs to the function space $L^{2}_{0}(\Omega)$. \BBB

Applying \cite[Lem.~II.2.1.1]{Sohr}, we find a vector field $\bm{G} \in \bH^{1}_{0}(\Omega)$ such that
\begin{align*}
\div \bm{G} = g \text{ in } \Omega, \text{ and } \norm{\nabla \bm{G}}_{\bm{L}^{2}(\Omega)} \leq C \norm{g}_{L^{2}(\Omega)}
\end{align*}
for some constant $C > 0$ depending only on $\Omega$.  We define the bilinear form $a: \bH^{1}_{0,\sigma}(\Omega) \times \bH^{1}_{0,\sigma}(\Omega) \to (\bH^{1}_{0,\sigma}(\Omega))'$ by
\begin{equation}
\begin{aligned}
a(\bm{z}, \bm{v}) & := \int_{\Omega} \alpha_{\eps}(\e{\varphi}) \bm{z} \cdot \bm{v} + \mu \AAA (\nabla \bm{z} + (\nabla \bm{z})^{\top}) \BBB \cdot \nabla \bm{v} \dx \\
& \quad \quad + \int_{\Omega} (\nabla \e{\bm{u}})^{\top} \bm{z} \cdot \bm{v} - (\e{\bm{u}} \cdot \nabla) \bm{z} \cdot \bm{v} \dx.
\end{aligned}
\end{equation}
Using \eqref{uniqcond}, Poincar\'{e}'s inequality, H\"{o}lder's inequality, the boundedness of $\alpha_{\eps}$ and properties of the trilinear form $b(\bm{u}, \bm{v}, \bm{w}) := \int_{\Omega} (\bm{u} \cdot \nabla) \bm{v} \cdot \bm{w} \dx$ (see \cite[Lem.~4.1]{GHHKL}), it \AAA can be shown similar to \BBB \cite[Proof of Lem.~4.9]{GHHKL} that $a(\cdot, \cdot)$ is a bounded and coercive bilinear form.  Furthermore, defining
\begin{equation}\label{AdjointRHS}
\begin{aligned}
\bm{F}(\bm{v}) & := \int_{\Omega} \AAA \hat{\alpha}_{\eps} \BBB(\e{\varphi}) \e{\bm{u}} \cdot \bm{v} + (\der_{3}b + \tfrac{1}{2} \der_{2}h) \cdot \nabla \bm{v} + \der_{2} b \cdot \bm{v} \dx  \\
& \quad + \int_{\Omega} \sum_{i=1}^{m_{1}+m_{2}}  \lambda_{i} \left ( \der_{2}K_{i} \cdot \bm{v} + \left ( \tfrac{1}{2} \nabla \e{\varphi} \cdot \der_{2} \LL_{i} +  \der_{3} K_{i} \right ) \cdot \nabla \bm{v} \right ) \dx \\
& \AAA  \quad - \int_{\Omega} \alpha_{\eps}(\varphi_{\eps}) \bm{G} \cdot \bm{v} + \mu ( \nabla \bm{G} + (\nabla \bm{G})^{\top}) \cdot \bm{v} \dx \BBB \\
& \AAA \quad - \int_{\Omega} (\nabla \bm{u}_{\eps})^{\top} \bm{G} \cdot \bm{v} - (\bm{u}_{\eps} \cdot \nabla) \bm{G} \cdot \bm{v} \dx \BBB,
\end{aligned}
\end{equation}
\AAA and applying \eqref{assump:bhdiff}, \eqref{assump:KiLLidiff}, the fact that $\bm{G}, \bm{u}_{\eps} \in \bH^{1}(\Omega)$ and Sobolev embeddings leads to the deduction that \BBB $\bm{F}(\bm{v})$ is a bounded linear form on $\bH^{1}_{0,\sigma}(\Omega)$.  Thus, by the Lax--Milgram theorem, we obtain a unique $\hat{\bm{q}} \in \bH^{1}_{0,\sigma}(\Omega)$ such that
\begin{align*}
a(\hat{\bm{q}}, \bm{v}) =  \bm{F}(\bm{v}).  
\end{align*}
This implies that the solution $\e{\bm{q}} := \hat{\bm{q}} + \bm{G} \in \bH^{1}_{0}(\Omega)$ satisfies the weak formulation \eqref{weak:adjoint} with 
\begin{align*}
\div \e{\bm{q}} = \div \bm{G} = g.
\end{align*}
The existence of a unique adjoint pressure $\e{\pi} \in L^{2}(\Omega)$ follows from standard results, see for instance \cite[Lem.~II.2.2.1]{Sohr}.  Thus $(\e{\bm{q}}, \e{\pi})$ is the unique weak solution to the adjoint system \eqref{AdjointSystem}.
\end{proof}

\subsection{Necessary optimality conditions}
Now we can formulate the first order necessary optimality conditions for our optimal control problem.
\begin{thm}\label{thm:Optimality}
Let $\e{\varphi} \in \K_{ad}$ be a minimizer of \eqref{OpProblem}-\eqref{PF:constraint} with corresponding (unique) state variables $(\e{\bm{u}}, \e{p}) = \bm{S}_{\eps}(\e{\varphi})$, $\e{\bm{u}} \in \bH^{1}_{\bm{\sigma},\sigma}(\Omega)$, $\e{p} \in L^{2}_{0}(\Omega)$.  \AAA Furthermore, let $\{\lambda_{i}\}_{i=1}^{m_{1}+m_{2}}$ be the Lagrange multipliers associated to the integral state constraints $\{\GG_{i}(\e{\varphi})\}_{i=1}^{m_{1}+m_{2}}$, and $(\e{q}, \e{\pi})$ be the unique solution to the adjoint system \eqref{AdjointSystem}. \BBB  Then, under \eqref{assump:bhdiff}, \eqref{assump:KiLLidiff} and \eqref{assump:ZK}, the following optimality system is fulfilled:
\begin{equation}\label{FONC}
\begin{aligned}
0 & \leq \biginner{ \frac{1}{2} \AAA \hat{\alpha}_{\eps}'(\e{\varphi}) \BBB \abs{\e{\bm{u}}}^{2} - \alpha_{\eps}'(\e{\varphi}) \e{\bm{u}} \cdot \e{\bm{q}} + \frac{\gamma}{2c_{0} \eps} \Psi'(\e{\varphi})}{ \zeta - \e{\varphi}} \\
& \quad  + \biginner{\der_{5}b  + \sum_{i=1}^{m_{1}+m_{2}} \lambda_{i} \der_{5} K_{i}}{\zeta - \e{\varphi}}_{L^{2}(\Omega)} \\
& \quad + \biginner{\frac{\gamma \eps}{2c_{0}} \nabla \e{\varphi} + \frac{1}{2} \der_{4} h + \frac{1}{2} \sum_{i=1}^{m_{1}+m_{2}} \lambda_{i} \LL_{i}}{\nabla (\zeta - \e{\varphi})}_{\bm{L}^{2}(\Omega)} \quad \forall \zeta \in \Phi,
\end{aligned}
\end{equation}
where $\der_{4}h$ is evaluated at $(x, \nabla \e{\bm{u}}, \e{p}, \nabla \e{\varphi})$, $\LL_{i}$ is evaluated at $(x, \nabla \e{\bm{u}}, \e{p})$, and $\der_{5}b = B(x, \e{\bm{u}}, \nabla \e{\bm{u}}, \e{p}) \; \der_{2} z(x, \e{\varphi})$, $\der_{5} K_{i} = \mathcal{K}_{i}(x, \e{\bm{u}}, \nabla \e{\bm{u}}, \e{p}) \; \der_{2} y_{i}(x, \e{\varphi})$.
\end{thm}
\begin{proof}
In Sec.~\ref{sec:FrechetDiffjeps} we have \AAA shown \BBB that the reduced functional 
\begin{align*}
\e{j}(\e{\varphi}) := \e{\JJ}(\e{\varphi}, \bm{S}_{\eps}(\e{\varphi}))
\end{align*}
is Fr\'{e}chet differentiable with respect to $\e{\varphi}$, and in Sec.~\ref{sec:LMexist} we derived the gradient equation \eqref{GradientIneq}.  We now want to rewrite \eqref{GradientIneq} into a more convenient form using the adjoint system.  \AAA For $\zeta \in \Phi$, let $(\linvel, \linp)$ denote the unique solution to the linearized state equations \eqref{LinearizedStateSys} corresponding to $\delta = \zeta - \e{\varphi}$.  Then, computing the derivative of $\e{j}$ at $\e{\varphi}$ in the direction $\delta$ leads to \BBB 
\begin{equation}\label{compute:Djeps}
\begin{aligned}
& \der \e{j}(\e{\varphi})(\zeta - \e{\varphi}) \\
& \quad = \int_{\Omega} \frac{1}{2} \AAA \hat{\alpha}_{\eps}' \BBB (\e{\varphi}) (\zeta - \e{\varphi}) \abs{\e{\bm{u}}}^{2} + \AAA \hat{\alpha}_{\eps} \BBB (\e{\varphi}) \e{\bm{u}} \cdot \linvel \dx \\
& \quad \quad + \int_{\Omega} \frac{\gamma}{2 c_{0}} \left ( \frac{1}{\eps} \Psi'(\e{\varphi}) (\zeta - \e{\varphi}) + \eps \nabla \e{\varphi} \cdot \nabla (\zeta - \e{\varphi}) \right ) \dx \\
& \quad \quad + \int_{\Omega} ( \der_{2}b, \der_{3}b, \der_{4}b, \der_{5}b) \cdot (\linvel, \nabla \linvel, \linp, \zeta - \e{\varphi}) \dx \\
& \quad \quad + \int_{\Omega} \frac{1}{2} (\der_{2}h, \der_{3}h, \der_{4}h)  \cdot (\nabla \linvel, \linp, \nabla (\zeta - \e{\varphi})) \dx,
\end{aligned}
\end{equation}
where in the above and for the rest of the proof $\{ \der_{i} b \}_{i=2}^{5}$ are evaluated at $(x, \e{\bm{u}}, \nabla \e{\bm{u}}, \e{p}, \e{\varphi})$ and $\{\der_{i}h\}_{i=2}^{4}$ are evaluated at $(x, \nabla \e{\bm{u}}, \e{p}, \nabla \e{\varphi})$.  Using the adjoint state $\e{\bm{q}}$ as a test function in \eqref{LinearizedStateSys} (with $\delta = \zeta - \e{\varphi}$) leads to
\begin{equation}\label{LinearizedState:test:qeps}
\begin{aligned}
0 & = \int_{\Omega} \alpha_{\eps}'(\e{\varphi}) (\zeta - \e{\varphi}) \e{\bm{u}} \cdot \e{\bm{q}} + \alpha_{\eps}(\e{\varphi}) \linvel \cdot \e{\bm{q}} + \mu \nabla \linvel \cdot \nabla \e{\bm{q}} \dx \\
& \quad + \int_{\Omega} (\linvel \cdot \nabla ) \e{\bm{u}} \cdot \e{\bm{q}} + (\e{\bm{u}} \cdot \nabla ) \linvel \cdot \e{\bm{q}} - \linp g \dx,
\end{aligned}
\end{equation}
where $g = \div \e{q}$ as in \eqref{divqeps}.  Using the linearized state $\linvel$ as a test function in the adjoint system \eqref{weak:adjoint} leads to
\begin{equation}\label{Adjoint:test:linearized}
\begin{aligned}
& \int_{\Omega} \alpha_{\eps}(\e{\varphi}) \e{\bm{q}} \cdot \linvel + \mu \nabla \e{\bm{q}} \cdot \nabla \linvel + (\nabla \e{\bm{u}})^{\top} \e{\bm{q}} \cdot \linvel - (\e{\bm{u}} \cdot \nabla) \e{\bm{q}} \cdot \linvel \dx \\
& \quad = \int_{\Omega} \AAA \hat{\alpha}(\e{\varphi}) \BBB \e{\bm{u}} \cdot \linvel + (\der_{3}b + \tfrac{1}{2} \der_{2}h) \cdot \nabla \linvel + \der_{2}b \cdot \linvel \dx \\
& \quad \quad + \int_{\Omega} \sum_{i=1}^{m_{1}+m_{2}} \lambda_{i} ( \der_{2} K_{i} \cdot \linvel + (\tfrac{1}{2} \nabla \e{\varphi} \cdot \der_{2} \LL_{i} + \der_{3} K_{i}) \cdot \nabla \linvel) \dx,
\end{aligned}
\end{equation}
where we have used $\linvel \in \bH^{1}_{0,\sigma}(\Omega)$ to deduce that $\int_{\Omega} (\nabla \e{q})^{\top} \cdot \nabla \linvel \dx = 0$ (see for instance \cite[(4.29)]{GHHKL}).  Upon comparing terms in \eqref{LinearizedState:test:qeps} and \eqref{Adjoint:test:linearized} we find that
\begin{equation}\label{OpSys:simplify:1}
\begin{aligned}
& \int_{\Omega} \hat{\alpha}(\e{\varphi}) \e{\bm{u}} \cdot \linvel + (\der_{3}b + \tfrac{1}{2} \der_{2}h) \cdot \nabla \linvel + \der_{2}b \cdot \linvel + (\e{\bm{u}} \cdot \nabla) \e{\bm{q}} \cdot \linvel \dx \\
& \quad \quad + \int_{\Omega} \sum_{i=1}^{m_{1}+m_{2}} \lambda_{i} ( \der_{2} K_{i} \cdot \linvel + (\tfrac{1}{2} \nabla \e{\varphi} \cdot \der_{2} \LL_{i} + \der_{3} K_{i}) \cdot \nabla \linvel) \dx \\
& \quad = \int_{\Omega} \alpha_{\eps}(\e{\varphi}) \e{\bm{q}} \cdot \linvel + \mu \nabla \e{q} \cdot \nabla \linvel + (\linvel \cdot \nabla ) \e{\bm{u}} \cdot \e{\bm{q}} \dx \\
&  \quad = \int_{\Omega} \linp g - \alpha_{\eps}'(\e{\varphi}) (\zeta - \e{\varphi}) \e{\bm{u}} \cdot \e{\bm{q}} - (\e{\bm{u}} \cdot \nabla ) \linvel \cdot \e{\bm{q}} 
\end{aligned}
\end{equation}
Using that $\linp \in L^{2}_{0}(\Omega)$, $\div \e{\bm{u}} = 0$ in $\Omega$, $\e{\bm{q}} = \linvel = \bm{0}$ on $\pd \Omega$, and thus
\begin{align*}
\int_{\Omega} \linp \theta \dx & = \theta \int_{\Omega} \linp \dx = 0, \\
 \int_{\Omega} (\e{\bm{u}} \cdot \nabla) \e{\bm{q}} \cdot \linvel + (\e{\bm{u}} \cdot \nabla) \linvel \cdot \e{\bm{q}} \dx & = \int_{\Omega} \e{\bm{u}} \cdot \nabla (\e{\bm{q}} \cdot \linvel) \dx = 0,
\end{align*}
we can simplify \eqref{OpSys:simplify:1} into
\begin{align*}
& \int_{\Omega} \linp \left ( -\der_{4}b - \tfrac{1}{2} \der_{3}h  - \sum_{i=1}^{m_{1}+m_{2}} \lambda_{i} \left ( \der_{4}K_{i} + \tfrac{1}{2} \nabla \e{\varphi} \cdot \der_{3} \LL_{i} \right ) \right ) \dx \\
& \quad \quad - \int_{\Omega} \alpha_{\eps}'(\e{\varphi}) (\zeta - \e{\varphi}) \e{\bm{u}} \cdot \e{\bm{q}} \dx \\
& \quad = \int_{\Omega} \AAA \hat{\alpha}_{\eps} \BBB (\e{\varphi}) \e{\bm{u}} \cdot \linvel + \left (\der_{3}b + \tfrac{1}{2} \der_{2}h \right ) \cdot \nabla \linvel + \der_{2}b \cdot \linvel \dx \\
& \quad \quad + \int_{\Omega} \sum_{i=1}^{m_{1}+m_{2}} \lambda_{i} \left ( \der_{2}K_{i} \cdot \bm{u} + \left ( \tfrac{1}{2} \nabla \e{\varphi} \cdot \der_{2} \LL_{i} + \der_{3} K_{i} \right ) \cdot \nabla \linvel \right ) \dx,
\end{align*}
and upon rearranging we obtain
\begin{equation}\label{OptSys:simplify:2}
\begin{aligned}
& \int_{\Omega} \AAA \hat{\alpha}_{\eps} \BBB (\e{\varphi}) \e{\bm{u}} \cdot \linvel + (\der_{2}b, \der_{3}b, \der_{4}b) \cdot (\linvel, \nabla \linvel, \linp) \dx \\
& \quad \quad + \int_{\Omega} \tfrac{1}{2}( \der_{2}h, \der_{3}h) \cdot (\nabla \linvel, \linp) \dx \\
& \quad = \int_{\Omega} - \alpha_{\eps}'(\e{\varphi}) (\zeta - \e{\varphi}) \e{\bm{u}} \cdot \e{\bm{q}} \dx \\
& \quad \quad  - \int_{\Omega} \sum_{i=1}^{m_{1}+m_{2}} \lambda_{i} (\der_{2}K_{i}, \der_{3}K_{i},  \der_{4}K_{i}) \cdot (\linvel, \nabla \linvel,  \linp) \dx \\
& \quad \quad - \int_{\Omega} \sum_{i=1}^{m_{1}+m_{2}} \lambda_{i} \tfrac{1}{2} \nabla \e{\varphi} \cdot (\der_{2} \LL_{i}, \der_{3} \LL_{i}) \cdot (\nabla \linvel, \linp) \dx.
\end{aligned}
\end{equation}
Substituting \eqref{OptSys:simplify:2} into \eqref{compute:Djeps}, we obtain
\begin{equation}
\begin{aligned}
& \der \e{j}(\e{\varphi})(\zeta - \e{\varphi}) \\
& \quad  = \int_{\Omega}  \left ( \frac{1}{2} \AAA \hat{\alpha}_{\eps}' \BBB (\e{\varphi}) \abs{\e{\bm{u}}}^{2} - \alpha_{\eps}'(\e{\varphi}) \e{\bm{u}} \cdot \e{\bm{q}} + \frac{\gamma}{2 c_{0} \eps} \Psi'(\e{\varphi})  \right ) (\zeta - \e{\varphi})  \dx \\
& \quad \quad + \int_{\Omega}  \der_{5} b \, (\zeta - \e{\varphi}) +  \left ( \frac{\gamma}{2 c_{0}}  \eps \nabla \e{\varphi} + \tfrac{1}{2} \der_{4}h \right ) \cdot \nabla (\zeta - \e{\varphi})\dx \\
& \quad \quad - \int_{\Omega} \sum_{i=1}^{m_{1}+m_{2}} \lambda_{i} (\der_{2}K_{i}, \der_{3}K_{i}, \der_{4}K_{i}) \cdot (\linvel, \nabla \linvel,  \linp) \dx \\
& \quad  \quad - \int_{\Omega} \sum_{i=1}^{m_{1}+m_{2}} \lambda_{i}  \tfrac{1}{2} \nabla \e{\varphi} \cdot (\der_{2} \LL_{i}, \der_{3} \LL_{i}) \cdot (\nabla \linvel, \linp) \dx.
\end{aligned}
\end{equation}
Together with the gradient equation \eqref{GradientIneq} and the distributional derivatives \eqref{GGFdiff}, we then obtain \eqref{FONC}.
\end{proof}

\begin{remark}
In the case where there is only a volume constraint, i.e., $m_{1} + m_{2} = 1$ with $\GG(\varphi) := \int_{\Omega} \varphi - \beta \dx$ for a fixed constant $\beta \in (-1,1)$, the existence of Lagrange multipliers using the Zowe--Kurcyusz constraint qualification has been shown in \cite[Proof of Thm.~7.1]{HechtThesis} (for the case of inequality constraint), see also \cite[Proof of Thm.~3]{GarckeHechtStokes} for another argument using geometric variations.  For the case of equality constraint, we refer to \cite[Proof of Thm.~4.10]{GHHKL} which is based on a different argument.
\end{remark}

\section{Verification of constraint qualification}\label{sec:VerifyZK}

In this section, we consider a model problem of minimizing the drag subject to constraints on the mass, center of mass and volume of the object.  More precisely, in a bounded domain $\Omega \subset \R^{2}$ with Lipschitz boundary, we study the following optimal control problem
\begin{align*}
\min_{(\varphi, \bm{u}, p)} \int_{\Omega} \frac{1}{2} \bm{a} \cdot \left ( \mu (\nabla \bm{u} + (\nabla \bm{u})^{\top}) - p \id \right ) \nabla \varphi + \frac{\gamma}{2 c_{0}} \left ( \frac{1}{\eps} \Psi(\varphi) + \frac{\eps}{2} \abs{\nabla \varphi}^{2} \right ) \dx 
\end{align*}
subject to $(\varphi, \bm{u}, p)$ solving the porous-medium Navier--Stokes equations \eqref{NSweak} and the following integral constraints:
\begin{align*}
\GG_{1}(\varphi) & = \int_{\Omega} \tfrac{1}{2}(1-\varphi) x_{1} \dx = 0, \\
\GG_{2}(\varphi) & = \int_{\Omega} \tfrac{1}{2}(1-\varphi) x_{2} \dx = 0, \\
\GG_{3}(\varphi) & = M -  \int_{\Omega} \tfrac{1}{2}\rho(x)(1-\varphi) \dx \geq 0, \\
\GG_{4}(\varphi) & = \int_{\Omega} \varphi - \beta \dx \geq 0,
\end{align*}
where $\bm{a}$ is a constant unit vector parallel to the flow direction $\flow$, $M > 0$ is a given positive constant representing an upper bound on the mass of the object, $\rho(x) \in L^{\infty}(\Omega)$ is a non-negative mass density, and $\beta \in (-1,1)$ so that the object is constraint to occupy a maximal volume of $\frac{1-\beta}{2} \abs{\Omega}$.  

The constraints $\GG_{1}(\varphi) = 0$ and $\GG_{2}(\varphi) = 0$ imply that the centre of mass for the object is located at the origin in $\R^{2}$ (which we can assume to hold without loss of generality by translating the domain $\Omega$).  We point out that one can also consider more general surface objective functionals $h$ that are one-homogeneous with respect to the last variable, as well as volume objective functionals $b$, however we consider this particular example of drag minimization as a practical application of our present approach.  Furthermore, in this example we have chosen to neglect the penalization term $\hat{\alpha}_{\eps} \abs{\bm{u}}^{2}$ in the objective functional.

It is straightforward to check that the function
\begin{align*}
h(x, \nabla \bm{u}, p, \nabla \varphi) := \nabla \varphi \cdot \left ( \mu (\nabla \bm{u} + (\nabla \bm{u})^{\top}) - p \id \right ) \bm{a}
\end{align*}
fulfils \eqref{assump:h} by the application of the Young's inequality.  Furthermore, it is shown in \cite[Proof of Thm.~4.1 and Rmk.~4.2]{GHHKL} that the functional 
\begin{align*}
\HH(\varphi, \bm{u}, p) := \int_{\Omega} \nabla \varphi \cdot \left ( \mu (\nabla \bm{u} + (\nabla \bm{u})^{\top}) - p \id \right ) \bm{a} \dx
\end{align*}
is bounded from below for $\varphi \in H^{1}(\Omega) \cap L^{\infty}(\Omega)$ with $s_{a} \leq \varphi \leq s_{b}$ a.e. in $\Omega$, $\bm{u} \in \bH^{1}_{\bm{g}, \sigma}(\Omega)$ and $p \in L^{2}_{0}(\Omega)$, and satisfies due to the product of weak-strong convergence:
\begin{align*}
\lim_{n \to \infty} \HH(\varphi_{n}, \bm{u}_{n}, p_{n}) = \HH(\varphi, \bm{u}, p)
\end{align*}
for sequences $\varphi_{n} \rightharpoonup \varphi$ in $H^{1}(\Omega)$, $\bm{u}_{n} \to \bm{u}$ in $\bH^{1}(\Omega)$ and $p_{n} \to p$ in $L^{2}(\Omega)$.  Hence, \eqref{assump:functional} is also fulfilled.  A short computation shows that
\begin{align*}
\der_{2}h = \mu (\nabla \varphi \otimes \bm{a} + \bm{a} \otimes \nabla \varphi), \; \der_{3} h = - a \cdot \nabla \varphi, \; \der_{4}h = (\mu (\nabla \bm{u} + (\nabla \bm{u})^{\top}) - p \id) \bm{a},
\end{align*}
and as $\bm{a}$ is a constant vector, one can infer that \eqref{assump:bhdiff} (specifically \eqref{equ:partialderivativesh}) is also fulfilled.  Then, it remains to verify \eqref{assump:Ki}, \eqref{assump:LLi}, \eqref{assump:KiLLidiff} and \eqref{assump:ZK} for the existence of Lagrange multipliers for the integral constraints $\GG_{1}, \dots, \GG_{4}$, and show that the admissible set $\K_{ad}$ is non-empty.

For the latter, \AAA note that we have the trivial example $\phi \equiv 1 \in \K_{ad}$ which corresponds to the case where there is no object in the domain $\Omega$.  In the following we will construct a non-trivial example in order to rule out the possibility where $\K_{ad} = \{ 1 \}$, which would imply the solution to the shape optimization problem is to have no object at all. \BBB We can always choose a function $\phi \in H^{1}(\Omega)$, $-1 \leq \phi \leq 1$ a.e. in $\Omega$ such that
\begin{align*}
\abs{\Omega} \beta < \int_{\Omega} \phi \dx, \quad \int_{\Omega} \tfrac{1}{2}(1-\phi) x_{i} \dx = 0 \text{ for } i = 1,2,
\end{align*}
which is equivalent to choosing an object $\{ \phi = -1 \}$ with its centre of mass at the origin with volume bounded above by $\tfrac{1-\beta}{2} \abs{\Omega}$.  Note that the mapping $\phi \mapsto \int_{\Omega} \frac{1}{2} \rho(x)(1-\phi) \dx$ is continuous, and thus we can always decrease the volume of the object region $\{\phi = -1 \}$ to ensure the mass is bounded from above by the constant $M$.  This ensures that $\phi \in \K_{ad}$ and hence \eqref{Kad} is satisfied.  

As $\Omega$ is a bounded domain, the functions $x_{1}$, $x_{2}$ are bounded.  Then, upon setting
\begin{equation*}
\begin{alignedat}{5}
& \mathcal{K}_{1} = x_{1}, \quad && y_{1} = \tfrac{1}{2}(1-\varphi),  \quad && k_{1} = 0,  \quad && \LL_{1}  = \bm{0}, \\
& \mathcal{K}_{2} = x_{2},  \quad && y_{2} = \tfrac{1}{2}(1-\varphi),  \quad && k_{2} = 0,  \quad && \LL_{2} = \bm{0}, \\
& \mathcal{K}_{3} = - \rho(x),  \quad && y_{3}  = \tfrac{1}{2}(1-\varphi), \quad && k_{3}  = M \abs{\Omega}^{-1}, \quad && \LL_{3}  = \bm{0}, \\
& \mathcal{K}_{4}  = 1, \quad && y_{4} = \varphi, \quad && k_{4} = - \beta, \quad && \LL_{4}  = \bm{0},
\end{alignedat}
\end{equation*}
we observe that \eqref{assump:Ki}, \eqref{assump:LLi} and \eqref{assump:KiLLidiff} are fulfilled by the above choices.  Then, by Thm.~\ref{thm:existence} we are guaranteed the existence of a minimizer $\e{\varphi}$ to the optimal control problem.  To verify the main assumption \eqref{assump:ZK} and derive the optimality conditions, we have to show that for an arbitrary $z = (z_{1}, z_{2}, z_{3}, z_{4})^{\top} \in \Y = \R^{4}$, there exists one function $\psi_{*} \in \Phi$, along with non-negative constants $\tau_{1}, \dots, \tau_{4}, \xi_{1}, \xi_{2}, \eta_{1}, \eta_{2}$ such that the following four conditions are fulfilled simultaneously:
\begin{subequations}
\begin{alignat}{3}
2z_{1} & = \tau_{1} \int_{\Omega} (\e{\varphi} - \psi_{*}) x_{1} \dx, \quad 2z_{2} = \tau_{2} \int_{\Omega} (\e{\varphi} - \psi_{*}) x_{2} \dx, \label{Constraint1:G1G2} \\
z_{3} & = \tau_{3} \int_{\Omega} \tfrac{1}{2} \rho(x) (\psi_{*} - \e{\varphi}) \dx - \eta_{1} + \xi_{1} \left ( M -  \int_{\Omega} \tfrac{1}{2} \rho(x) (1-\e{\varphi}) \dx \right ), \label{Constraint1:G3} \\
z_{4} & = \tau_{4} \int_{\Omega} \psi_{*} -  \e{\varphi} \dx - \eta_{2} + \xi_{2} \left ( \int_{\Omega} \e{\varphi} - \beta \dx \right ). \label{Constraint1:G4}
\end{alignat}
\end{subequations}
Due to their nature as equality constraints, we can use the fact that $\GG_{1}(\e{\varphi}) = \GG_{2}(\e{\varphi}) = 0$ to simplify \eqref{Constraint1:G1G2} into
\begin{align}
\label{Constraint1:G1G2:sim}
2 z_{i} = \tau_{i} \int_{\Omega} (1-\psi_{*}) x_{i} \dx \text{ for } i = 1,2.
\end{align}

We first argue for \eqref{Constraint1:G1G2:sim}.  As the origin $\bm{0} \notin \pd \Omega$, this implies that $\Omega$ has non-empty intersections with the four quadrants of $\R^{2}$, which we denote by $Q_{1} = \{x_{1},x_{2} > 0\}$, $Q_{2} = \{x_{1} < 0 ,x_{2} > 0 \}$, $Q_{3} = \{ x_{1}, x_{2} < 0 \}$ and $Q_{4} = \{ x_{1} > 0, x_{2} < 0\}$.  If $z_{1}$ (resp. $z_{2}$) is zero, we choose $\tau_{1}$ (resp. $\tau_{2}$) to be zero.  Thus, it is sufficient to focus on the case where $z_{1}$ and $z_{2}$ are non-zero, and in this case we consider a function $\psi_{*} \in \Phi$ not identically equal to $1$ with $\beta < \psi_{*} \leq 1$ a.e. in $\Omega$ such that the non-empty set $A := \supp{1-\psi_{*}}$ has Lebesgue measure
\begin{align}\label{A:Leb}
\abs{A} < \frac{2M}{(1-\beta) \norm{\rho}_{L^{\infty}(\Omega)}}
\end{align} 
and satisfies
\begin{align*}
A \subset \subset Q_{i} \cap \Omega \text{ if } (z_{1},z_{2}) \in Q_{i} \text{ for } i = 1,2,3,4.
\end{align*}
Then, we set
\begin{align*}
\tau_{i} = \frac{2 z_{i}}{\int_{\Omega} (1-\psi_{*}) x_{i} \dx},
\end{align*}
and thanks to the fact that $\psi_{*} \leq 1$ a.e. in $\Omega$, the function $1-\psi_{*}$ is non-negative in $\Omega$ and only positive in $A$.  The location of $A$ implies that the integrand $(1-\psi_{*}) x_{i}$ has the same sign as $z_{i}$ for $i = 1,2$, and so $\tau_{i}$ is positive for $i = 1,2$.  The condition on the Lebesgue measure of $A$ is used to satisfy the mass constraint.

For the inequality constraint, we have to show that the same function $\psi_{*}$ considered above simultaneously satisfies \eqref{Constraint1:G3} and \eqref{Constraint1:G4}.  We argue for the mass constraint, and the volume constraint follows along a similar argument.  There are two cases to consider: suppose the inequality constraint $\GG_{3}(\e{\varphi})$ is not active for the minimizer $\e{\varphi}$, i.e., $\e{\varphi}$ satisfies $\int_{\Omega} \frac{1}{2} \rho(x) (1-\e{\varphi}) \dx < M$.  Then, we can choose $\tau_{3} = 0$ and it holds that
\begin{align*}
\left \{ -\eta_{1} + \xi_{1} \left ( M - \int_{\Omega} \tfrac{1}{2} \rho(x) (1-\e{\varphi}) \dx \right ) \, | \, \eta_{1}, \xi_{1} \geq 0 \right \} = \R.
\end{align*}
Hence, we have fulfilled \eqref{Constraint1:G3} without making use of the function $\psi_{*}$.  On the other hand, if $\GG_{3}(\e{\varphi})$ is active, i.e., $\int_{\Omega} \frac{1}{2} \rho(x) (1-\e{\varphi}) \dx = M$, the condition \eqref{Constraint1:G3} simplifies to 
\begin{align*}
z_{3} = \tau_{3} \left ( M + \int_{\Omega} \tfrac{1}{2} \rho(x) (\psi_{*} - 1) \dx \right ) - k_{1}.
\end{align*}
A short calculation using \eqref{A:Leb} shows that the quantity in the bracket is positive, and so
\begin{align*}
\left \{ \tau_{3} \left ( M + \int_{\Omega} \tfrac{1}{2} \rho(x) (\psi_{*} - 1) \dx \right ) - \eta_{1}  \, | \, \eta_{1}, \tau_{1} \geq 0  \right \} = \R,
\end{align*}
which implies that \eqref{Constraint1:G3} is fulfilled.  Indeed, we see that 
\begin{align*}
M - \int_{\Omega} \tfrac{1}{2} \rho(x) (1 - \psi_{*}) \dx & = M - \int_{A} \tfrac{1}{2} \rho(x) (1-\psi_{*}) \dx \\
& \geq M - \tfrac{1}{2} \norm{\rho}_{L^{\infty}(\Omega)} (1-\beta) \abs{A} > 0.
\end{align*}
For the volume constraint \eqref{Constraint1:G4} we again divide the argument into two cases: if $\GG_{4}(\e{\varphi})$ is inactive, then \eqref{Constraint1:G4} holds automatically without the use of the function $\psi_{*}$, and if $\GG_{4}(\e{\varphi})$ is active, then using $\psi_{*} > \beta$ yields the desired result.

As a consequence, \eqref{assump:ZK} is fulfilled and we obtain the existence of Lagrange multipliers $\lambda_{1}, \lambda_{2} \in \R$, $\lambda_{3}, \lambda_{4} \in \R_{\geq 0}$.  By Thm.~\ref{thm:Optimality} the first order optimality condition is
\begin{align*}
0 & \leq  \biginner{\frac{\gamma \eps}{2c_{0}} \nabla \e{\varphi} + \frac{1}{2} \left ( \mu \left ( \nabla \bm{u}_{\eps} + \left ( \nabla \bm{u}_{\eps} \right )^{\top} \right ) - p_{\eps} \id \right ) \bm{a} }{\nabla (\zeta - \e{\varphi})}_{\bm{L}^{2}(\Omega)} \\
& + \biginner{ -\alpha_{\eps}'(\e{\varphi}) \e{\bm{u}} \cdot \e{\bm{q}} + \frac{\gamma}{2c_{0} \eps} \Psi'(\e{\varphi})}{\zeta - \e{\varphi}}_{L^{2}(\Omega)} \\
& + \biginner{-\tfrac{1}{2}\lambda_{1}x_{1} - \tfrac{1}{2} \lambda_{2} x_{2} -\tfrac{1}{2} \lambda_{3} \rho(x) + \lambda_{4}}{\zeta - \e{\varphi}}_{L^{2}(\Omega)} \quad \forall \zeta \in \Phi,
\end{align*}
together with the complementary slackness conditions
\begin{align*}
\lambda_{3} \left ( M - \int_{\Omega} \tfrac{1}{2} \rho(x) (1-\e{\varphi}) \dx \right ) = 0, \quad \lambda_{4} \left ( \int_{\Omega} \e{\varphi} - \beta \dx \right ) = 0.
\end{align*}

\AAA
\begin{remark}
We point out that the mass constraint $\GG_{3}(\varphi) = M - \int_{\Omega} \tfrac{1}{2}\rho(x) (1-\varphi) \dx \geq 0$ can also be thought of as a constraint on a construction cost, where the value $\rho(x) > 0$ represents the cost of building the object at the point $x \in \Omega$, and $M$ denotes a maximal cost.  
\end{remark}
\BBB

\AAA
Let us now consider a similar model problem but with the single integral constraint on the total potential power \eqref{PotPower}.  More precisely, in a bounded domain $\Omega \subset \R^{2}$ with Lipschitz boundary, we study the following optimal control problem
\begin{align*}
\min_{(\varphi, \bm{u}, p)} \int_{\Omega} \frac{1}{2} \bm{a} \cdot \left ( \mu (\nabla \bm{u} + (\nabla \bm{u})^{\top}) - p \id \right ) \nabla \varphi + \frac{\gamma}{2 c_{0}} \left ( \frac{1}{\eps} \Psi(\varphi) + \frac{\eps}{2} \abs{\nabla \varphi}^{2} \right ) \dx ,
\end{align*}
subject to $(\varphi, \bm{u}, p)$ solving the porous-medium Navier--Stokes equations \eqref{NSweak} (with zero body force $\bm{f} = \bm{0}$) and the following integral constraint:
\begin{align*}
G(\varphi, \bm{u}) & = \int_{\Omega} D \abs{\Omega}^{-1} - \tfrac{1}{2}(1+\varphi) \frac{\mu}{2} \abs{\nabla \bm{u}}^{2} \dx \geq 0,
\end{align*}
where $\bm{a} = - \flow^{\perp}$ is the negative unit vector perpendicular to the flow direction $\flow$ and $D > 0$ is a given positive constant representing an upper bound on the total potential power.  Then, upon setting
\begin{align*}
\mathcal{K} = \frac{\mu}{2} \abs{\nabla \bm{u}}^{2} , \quad y = -\tfrac{1}{2}(1+\varphi), \quad k = D \abs{\Omega}^{-1}, \quad \LL = \bm{0},
\end{align*}
we see that \eqref{assump:Ki}, \eqref{assump:LLi} and  \eqref{assump:KiLLidiff} are fulfilled.  In this setting we observe that the trivial example $\varphi \equiv -1$ belongs to the admissible set of design functions $\K_{ad}$.  A non-trivial example can be found if the domain $\Omega$ is sufficiently large or the viscosity $\mu$ is sufficiently small.  Indeed, let $\varphi$ be a function in $H^{1}(\Omega)$ with $-1 \leq \varphi \leq 1$ a.e. in $\Omega$ but not identically equal to $1$ or $-1$.  Denote by $\bm{u}$ the unique velocity field associated to the state equation \eqref{PF:state}, then by \eqref{uniqcond} it holds that
\begin{align}\label{Pot:Power:Kad}
\int_{\Omega} \frac{1+\varphi}{2} \frac{\mu}{2} \abs{\nabla \bm{u}}^{2} \dx \leq \frac{\mu}{2} \norm{\nabla \bm{u}}_{\bm{L}^{2}(\Omega)}^{2} < \frac{\mu^{3}}{2 K_{\Omega}^{2}},
\end{align} 
where from \eqref{defn:KOmega} the constant $K_{\Omega}$ is $K_{\Omega} = \frac{1}{2} \abs{\Omega}^{\frac{1}{2}}$ in two dimensions.  Note that the above upper bound is independent of $\varphi$.  For any fixed positive constant $D$, we can take a sufficiently large domain $\Omega$ or sufficiently small viscosity $\mu$, so that $\frac{\mu^{3}}{2 K_{\Omega}^{2}} \leq D$.  Then, this implies that $\K_{ad}$ is non-empty and thus \eqref{Kad} is fulfilled.  Furthermore, following the arguments above, we deduce by Thm.~\ref{thm:existence} that there exists at least one minimizer $\e{\varphi}$ to the optimal control problem.  Writing $\GG(\e{\varphi}) = G(\e{\varphi}, \bm{S}_{\eps}(\e{\varphi}))$, to verify the assumption \eqref{assump:ZK}, we have to show for an arbitrary $z \in \R$, there exists one function $\psi_{*} \in \Phi$ along with non-negative constants $\tau, \xi, \eta$ such that
\begin{equation}\label{ZK:PotPow}
\begin{aligned}
z = \tau \der \GG(\e{\varphi})(\psi_{*} - \e{\varphi}) - \eta + \xi \GG(\e{\varphi}).
\end{aligned}
\end{equation}
Observe that for this particular setting (with a large domain $\Omega$ or small viscosity $\mu$), the inequality $D \geq \frac{\mu^{3}}{2 K_{\Omega^{2}}}$ holds independently of the minimizer $\e{\varphi}$, and thus by \eqref{Pot:Power:Kad} $\GG(\e{\varphi}) > 0$ holds.  In particular, the constraint is always inactive, and we do not need to find the function $\psi_{*}$ as
\begin{align*}
\left \{ \xi \GG(\e{\varphi}) - \eta \, | \, \eta, \xi \geq 0 \right \} = \R.
\end{align*}
Furthermore, as the constraint is always inactive the complementary slackness condition \eqref{Slackness} implies that the associated Lagrange multiplier $\lambda$ is zero.  Hence,  by Thm.~\ref{thm:Optimality} the first order optimality condition is
\begin{align*}
0 & \leq  \biginner{\frac{\gamma \eps}{2c_{0}} \nabla \e{\varphi} + \frac{1}{2} \left ( \mu \left ( \nabla \bm{u}_{\eps} + \left ( \nabla \bm{u}_{\eps} \right )^{\top} \right ) - p_{\eps} \id \right ) \bm{a} }{\nabla (\zeta - \e{\varphi})}_{\bm{L}^{2}(\Omega)} \\
& + \biginner{ -\alpha_{\eps}'(\e{\varphi}) \e{\bm{u}} \cdot \e{\bm{q}} + \frac{\gamma}{2c_{0} \eps} \Psi'(\e{\varphi})}{\zeta - \e{\varphi}}_{L^{2}(\Omega)} \quad \forall \zeta \in \Phi.
\end{align*}
\BBB


\section{Numerical implementation and simulations}\label{sec:numerics}
Let us now describe how we can use the above results to compute optimal shapes and topologies in given flow
settings.  Since our optimization variable is a phase field, and thus has the natural regularity $\varphi \in H^{1}(\Omega) \cap L^{\infty}(\Omega)$, we use the variable metric projection type (VMPT) method proposed in \cite{Blank_Rupprecht} to solve the resulting minimization problems.  A standard projected gradient method can not be used for the constraint minimization problem due to the fact that $H^{1}(\Omega) \cap L^{\infty}(\Omega)$ is not a Hilbert space.  \AAA The VMPT method uses derivative information which can be represented with the help of the adjoint variables as specified in \eqref{AdjointSystem}.
\BBB

For the potential function $\Psi$ we use the double-obstacle free energy, namely
\begin{align}\label{defn:Obstacle}
\Psi(\e{\varphi}) = \begin{cases}
\frac{1}{2}(1-\e{\varphi}^2) & \text{ if } \abs{\e{\varphi}} \leq 1,\\
\infty & \text{ else}.
\end{cases}
\end{align}
From this we obtain the constraint $\abs{\e{\varphi}} \leq 1$,
and $c_{0} = \frac{\pi}{2}$, where $c_{0}$ is the constant defined in \eqref{defn:c0}. 
Although the double-obstacle potential \eqref{defn:Obstacle} does not satisfy
\eqref{assump:Psi}, the analysis is not affected once we choose $s_{a} = -1$ and $s_{b} = 1$, so that $\abs{\e{\varphi}} \leq 1$ and the potential becomes $\Psi(\e{\varphi}) = \frac{1}{2}(1-\e{\varphi}^{2})$.  
We refer the reader to \cite{GarckeHechtNS,GHHKNumerics} which also uses 
the double-obstacle potential \eqref{defn:Obstacle}.  
For the porous-medium term $\e{\alpha}(\e{\varphi})$ in the state equations \eqref{PF:NS} we choose
\begin{align}\label{num:alpha}
\e{\alpha}(\e{\varphi}) = \frac{\overline{\alpha}}{2 \eps}(1-\e{\varphi}),
\end{align}
with a fixed positive constant $\overline{\alpha}$, and \eqref{assump:alpha} is fulfilled with $s_{a} = -1$ and $s_{b} = 1$.  \AAA We choose
\begin{align*}
\hat{\alpha}_{\eps} \equiv \alpha_{\eps}
\end{align*}
so that \eqref{assump:hatalpha} is also satisfied.  For the remaining part of this section, we denote both variables by $\alpha_{\eps}$, set $\bm{f} = \bm{0}$ in \eqref{PF:state}, and define
\begin{align*}
\Phi = \{ f \in H^{1}(\Omega) \, | \, -1 \leq f \leq 1 \text{ a.e. in } \Omega \}.
\end{align*} \BBB

\subsection{Spatial discretization}
We use finite elements for the numerical discretization of the minimization problem.  
We use piecewise linear and globally continuous finite elements for the representation 
of $\e{\varphi}$, $\e{p}$ and $\e{\pi}$ and piecewise quadratic and globally continuous finite elements for $\e{\bm{u}}$ and $\e{\bm q}$ on a conforming triangulation of the domain $\Omega$.

\AAA 
It is well-known that in phase field applications the variable $\e{\varphi}$ changes rapidly across the interfacial layers, and an adaptive concept for its spatial resolution is indispensable.  Hence, for the mesh generation we use the Dual Weighted Residual (DWR) method \cite{Becker} where our implementation is guided by \cite{HHKK}.  This generates adaptive meshes which well resolve the interfacial regions, and also well reflect the underlying flow physics, compare also \cite{HHK}. \BBB  The DWR approach is only applicable if for a given triangulation an optimal solution is already found and uses this information to calculate error indicators.

\AAA
For fast calculations, it is desirable to use coarse meshes.  In the core of the VMPT method we solve projection-type problems using a
primal-dual-active-set strategy (PDAS).  Here the active set corresponds to degrees of freedom with $\abs{\e{\varphi}} = 1$.  Thus in every step of the PDAS we solve the problems on the inactive set $\abs{\e{\varphi}} < 1$ only.  Note that the integral constraints have to be fulfilled by changing the
phase field on the inactive set only.  If this set contains too few degrees of freedom, the PDAS is not
successful in solving the projection-type problem and thus the algorithm breaks down. 

To overcome this numerical issue on coarse meshes, we additionally require that a given amount, say 2\%, of the phase field's degrees of freedom are inactive.  If this is not the case, we use mesh adaptation that is based on $\e{\varphi}$ only, namely we use the jumps of the normal derivatives of $\e{\varphi}$ across edges as proposed in \cite{GHHKNumerics} to generate new degrees of freedom inside the interface to be
able to proceed with the PDAS.
\BBB

We stop the adaptation loop as soon as a given maximum number of degrees of freedom is reached.

\subsection{Topology optimization - a tube through heavy ground}
\AAA Although we have mainly focused on shape optimization with the phase field approach in this paper, we point out that \BBB using a phase field variable for the representation of the \AAA unknown shape also \BBB allows us to deal with situations
where no a priori \AAA toplogical \BBB information is available.  \AAA In particular, the phase field approach is capable of topology optimization, as done in \cite{Blank,PRW,Takezawa,WangZhou}. \BBB  Here we consider the situation where the domain $\Omega = (0,1)^{2}$ contains several impermeable rocks and we would like to search for a tube that connects the inflow at the bottom to the outflow at the top, see Fig.~\ref{fig:heavyGround}. 

\begin{figure}
  \centering
\includegraphics[width=0.4\textwidth]{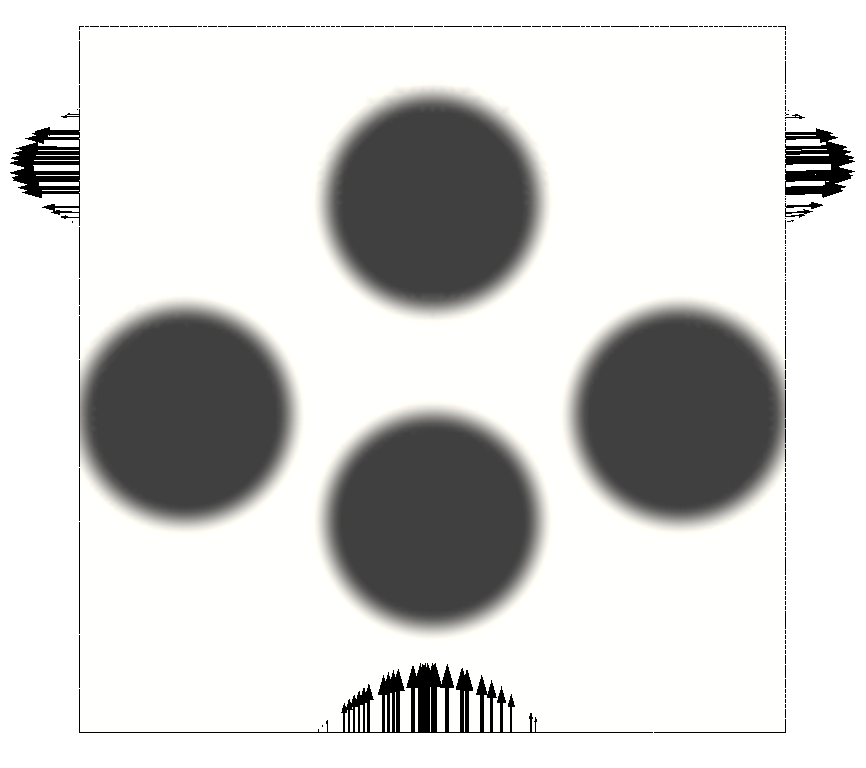}
\caption{The inflow and outflow conditions for the Navier--Stokes equations together
with the
  location of the rocks.}
\label{fig:heavyGround}
\end{figure}

Constructing a tube through the rocks is expensive and therefore a tube that avoids these regions is desired.  
 So this is a setting where we want to minimize the cost of an object.  
 The inflow and the outflow regions as well as the location of the rocks are a
priori known.
  We define the inflow and outflow conditions as
\begin{align*}
g_{in}(x)& =
\begin{pmatrix}
0\\
\max \left(2\left( 1-\left(\frac{x_{1}-0.5}{1/6}\right)^2  \right),0.0\right)
\end{pmatrix}, \\
g_{out}^i(x)&
=
\begin{pmatrix}
\max\left((-1)^{i}\left( 1-\left(\frac{x_{2}-0.8}{1/12}\right)^2  \right),0.0\right)\\
0
\end{pmatrix} \text{ for } i = 1,2.
\end{align*}
For the objective functional we define a `rock' centered at $\bm{m}$ with radius $\sigma$ and associated cost $c$ as
\begin{align*}
& R[\bm{m},\sigma,c](x) :=
  (c-1) \left (
\frac{\phi_{0}(-\frac{1}{\eps}(\norm{\frac{x-\bm{m}}{\sigma}}-1))+1}{2} \right
)+1, \\
\text{ where } & \quad  \phi_{0}(z) = 
  \begin{cases}
  \sin(z) & \text{ if } \abs{z} \leq \frac{\pi}{2},\\
  \text{sign}(z) & \text{else.}
  \end{cases} 
\end{align*}
We consider the functions
\begin{align*}
& b(x, \bm{u}, \nabla \bm{u}, p, \varphi):= \left (\frac{1+\varphi}{2} \right )
\prod_{i=1}^{4} R[\bm{m}_{i}, \sigma, c](x),
\quad h(x, \nabla \bm{u}, p, \nabla \varphi) = 0,
\end{align*} 
where 
\begin{equation*}
\begin{alignedat}{5}
\bm{m}_{1} & = (0.5, 0.3)^{\top}, \quad && \bm{m}_{2} && = (0.15, 0.45)^{\top}, \\
\bm{m}_{3} & = (0.85, 0.45)^{\top}, \quad && \bm{m}_{4} && = (0.5, 0.75)^{\top}.
\end{alignedat}
\end{equation*}
The optimization problem \eqref{OpProblem} then becomes
\begin{align*}
\min_{(\varphi, \bm{u}, p)} \mathcal{J}_{\eps}(\varphi, \bm{u}, p)&  = \int_{\Omega}
\frac{1}{2} \alpha_{\eps}(\varphi) \abs{\bm{u}}^{2} + \frac{1 + \varphi}{2}
\prod_{i=1}^{4} R[\bm{m}_{i}, \sigma, c] \dx \\
& \quad  + \int_{\Omega} \frac{\gamma}{\pi} \left ( \frac{1}{\eps}
\Psi(\varphi) + \frac{\eps}{2} \abs{\nabla \varphi}^{2} \right ) \dx,
\end{align*}
subject to $\varphi \in \Phi$, $\bm{u} \in \bm{H}^{1}_{\bm{g},
\sigma}(\Omega)$, $p \in L^{2}_{0}(\Omega)$ satisfying \eqref{NSweak}, and $\alpha_{\eps}$ was defined earlier in \eqref{num:alpha}.  For this example we do not apply any integral constraints, \AAA as it serves to demonstrate the strength of the phase field approach in being able to deal with situations where no a priori topological information is known.  Having the solution to the unconstrained problem at hand, one might reduce for example the size of the tube by imposing additional volume constraints.  However, specifying such constraints beforehand might lead to inadmissible situations.
\BBB 
 
We start the optimization procedure with no \AAA prior \BBB information, i.e., $\e{\varphi}^{0}
\equiv 0$, on a homogeneous coarse grid with mesh size $h=1/20$ yielding 685 degrees of unknowns for $\e{\varphi}$.  We stop the solution and adaptation procedures as soon as an optimal solution with
more than 100000 degrees of freedom is found.  The numerical parameters are: $\sigma = 0.15$, $c = 50$, $\eps = 0.01$, $\overline{\alpha} = 5$, $\mu = 0.02$ and $\gamma = 0.001$.  To stress the advantages of our approach in Fig.~\ref{fig:heavyGround:iterates} we show
$\e{\varphi}$ at various stages of the optimization procedure.

\begin{figure}
  \centering
  \includegraphics[width=0.25\textwidth]{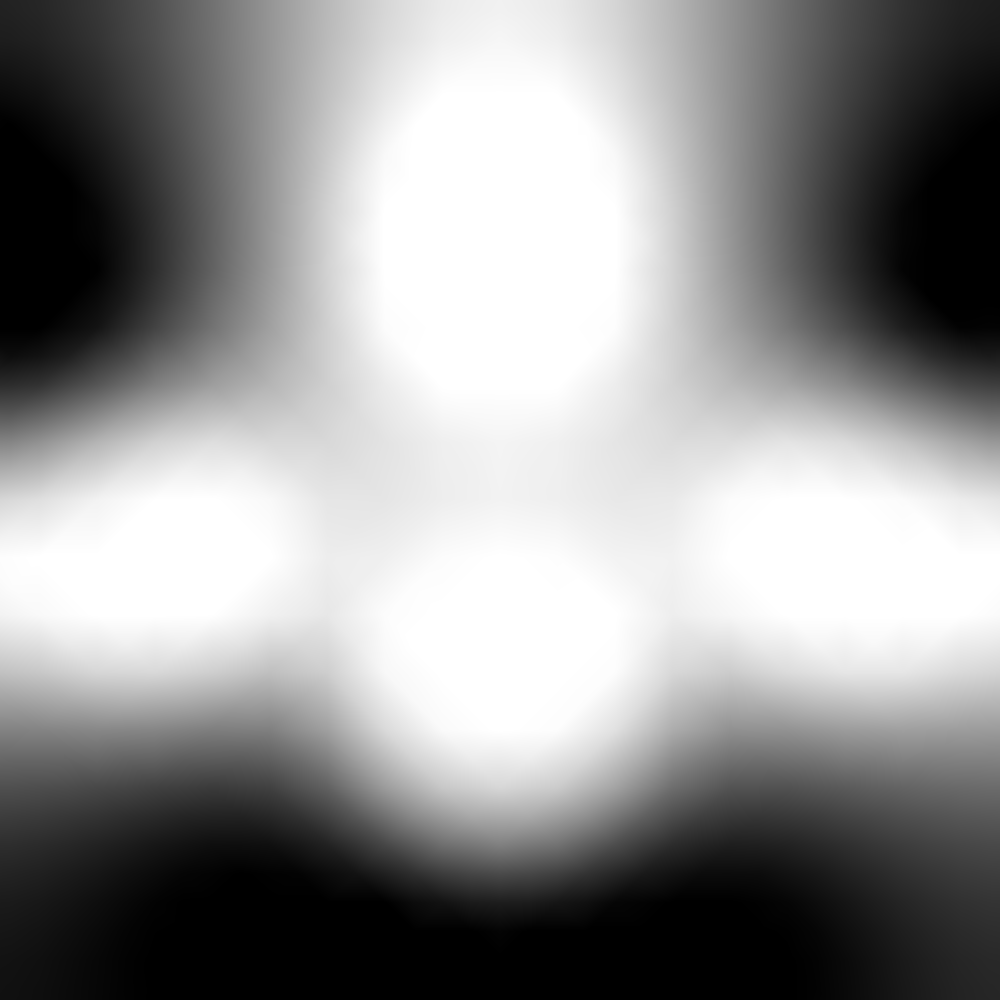}
  \hspace{1em}
  \includegraphics[width=0.25\textwidth]{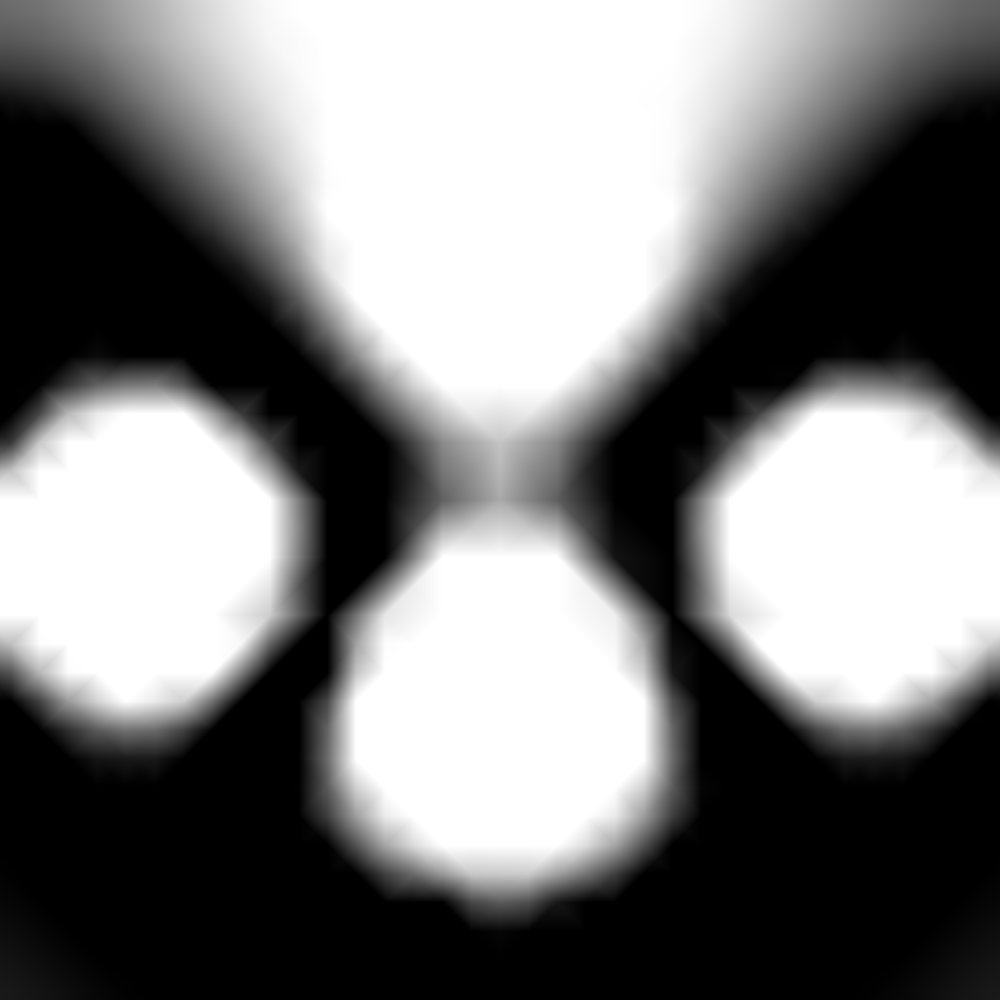}
  \hspace{1em}
  \includegraphics[width=0.25\textwidth]{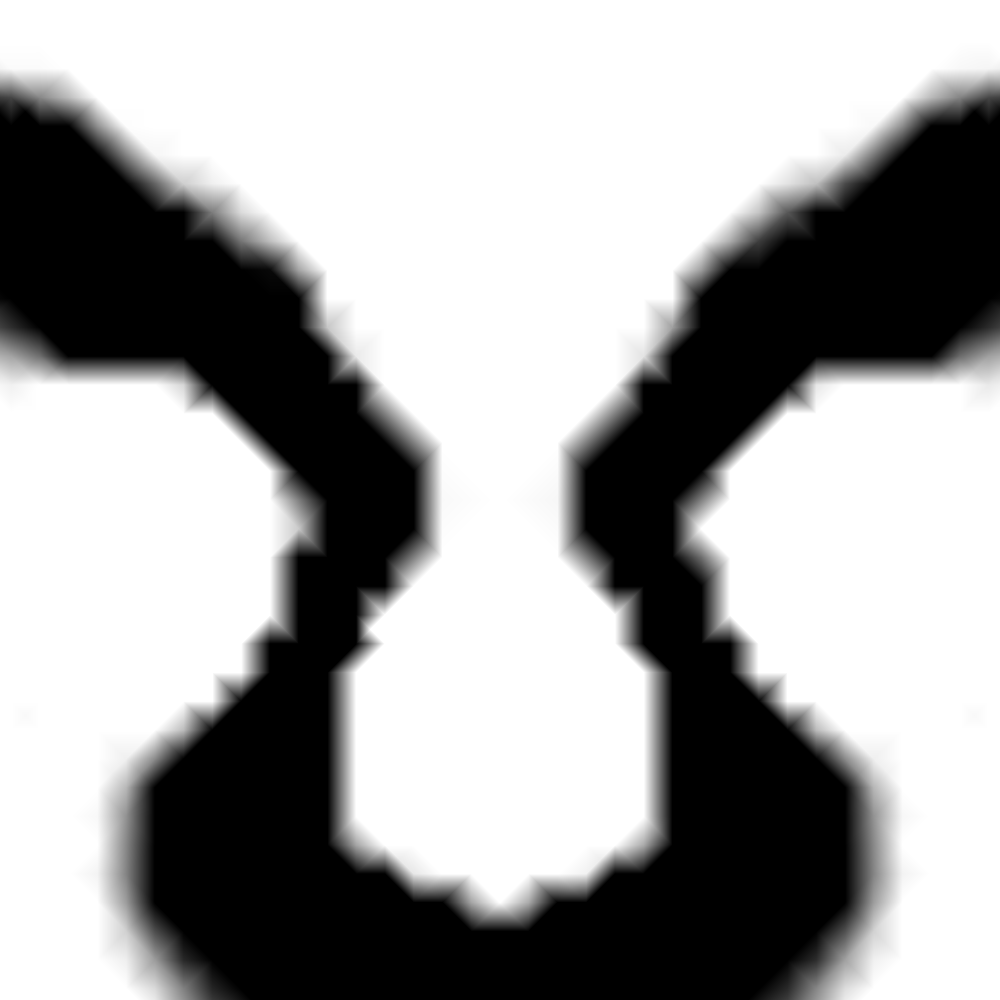}\\
  \includegraphics[width=0.25\textwidth]{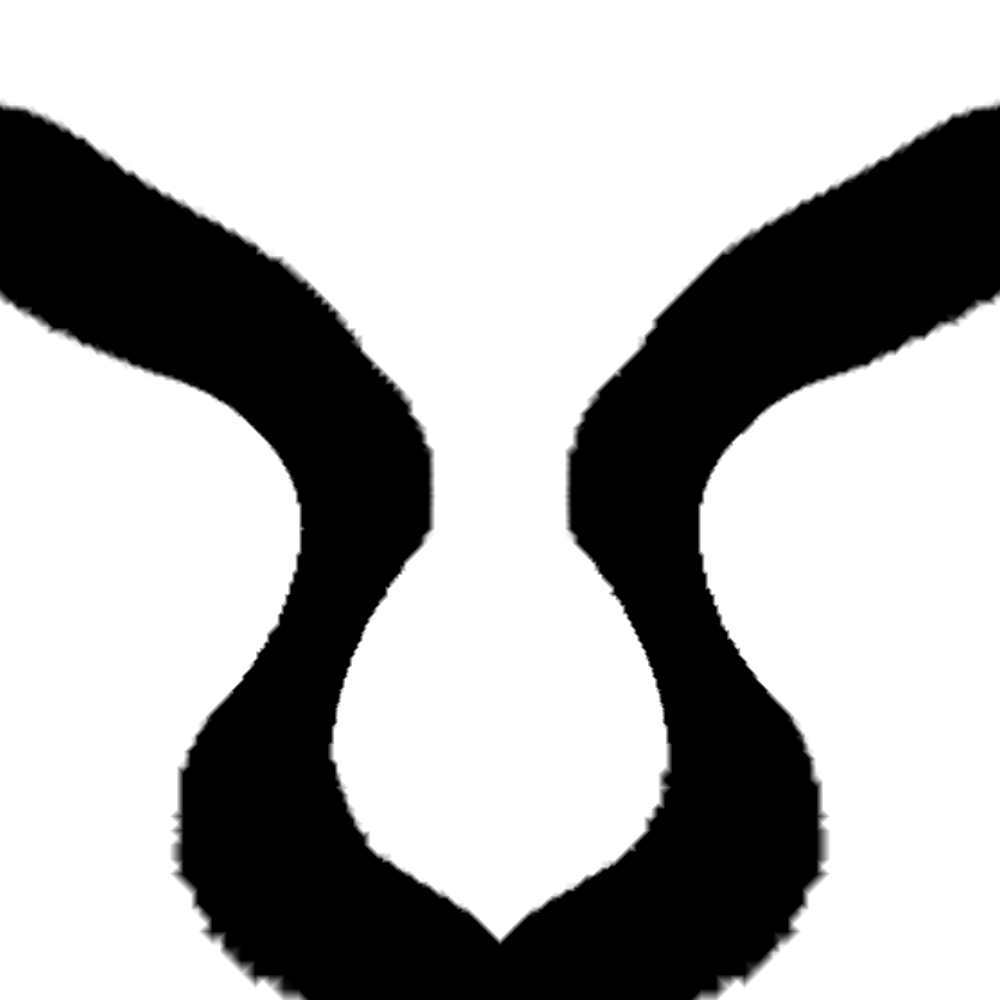}
  \hspace{1em}
  \includegraphics[width=0.25\textwidth]{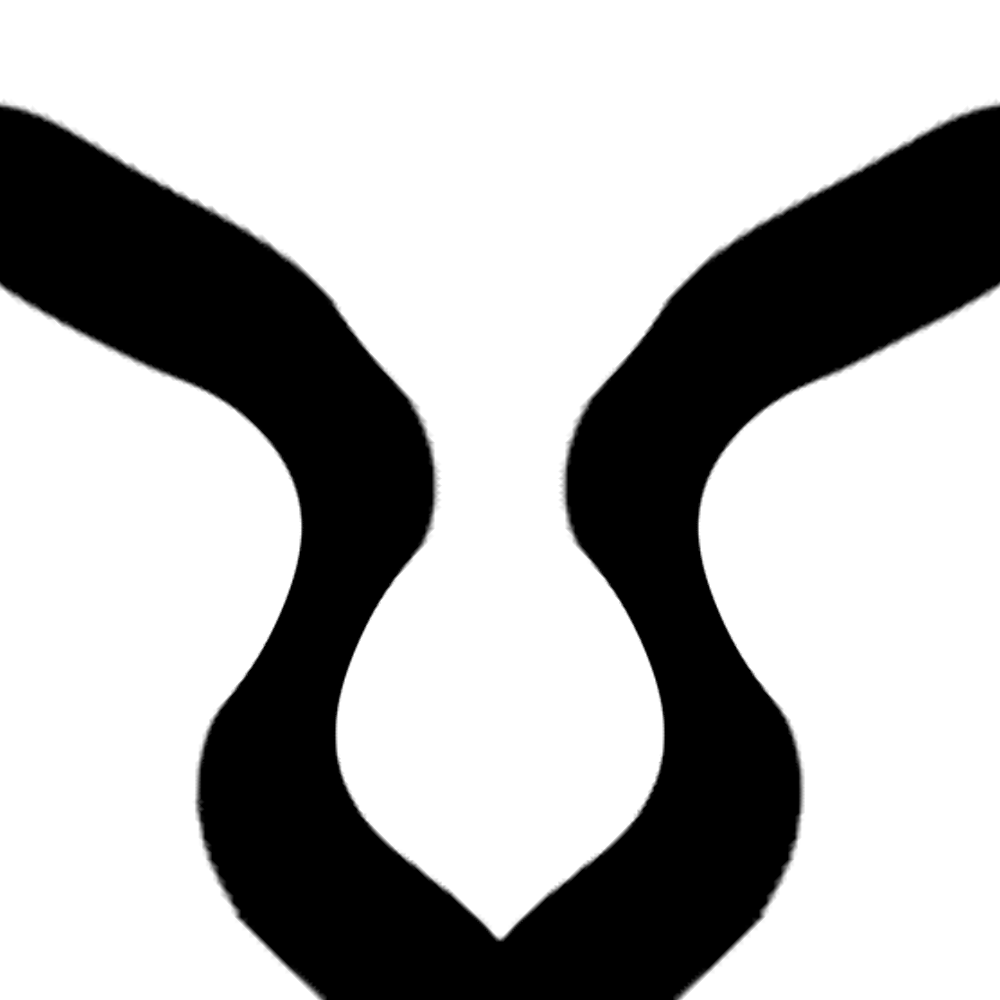}
  \hspace{1em}
  \includegraphics[width=0.25\textwidth]{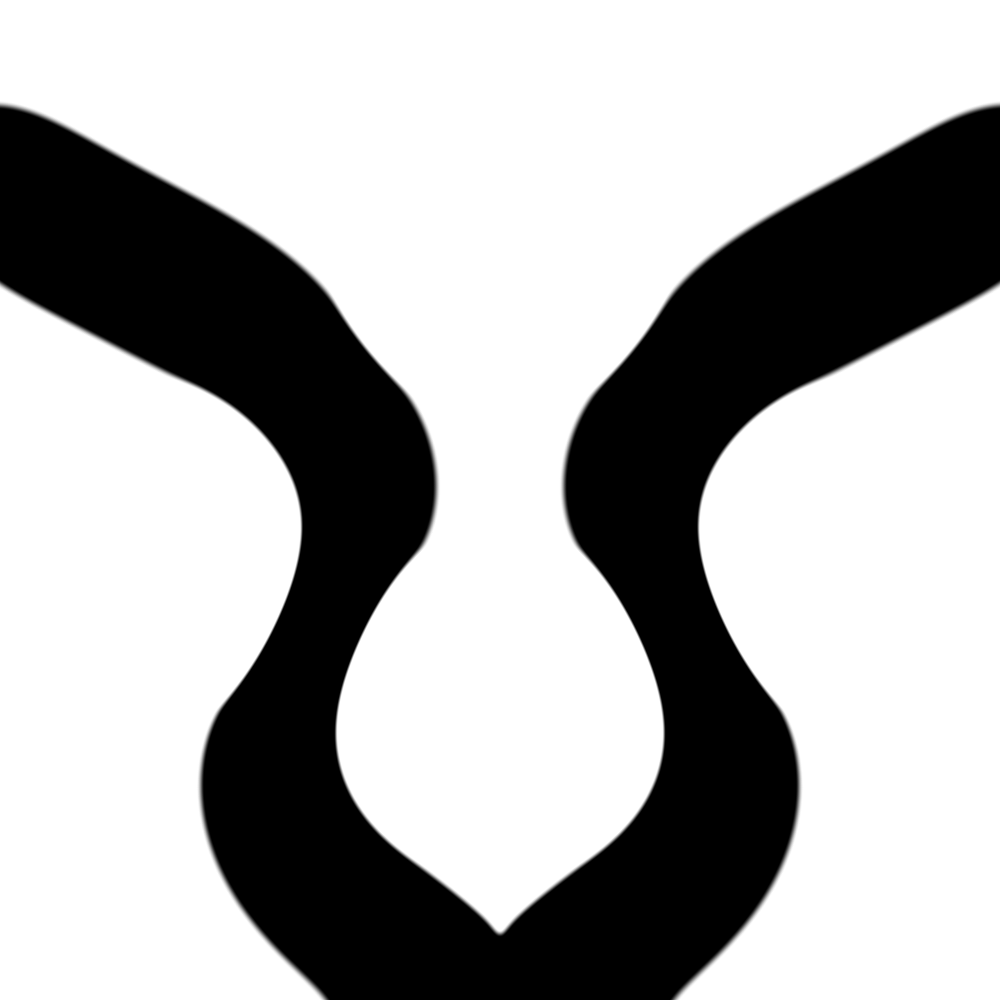}
  \caption{The iterations 20,30,40,80,120,190 of the VMPT to minimize the cost of a
tube through heavy ground. We see that after 40 steps already an optimal
structure is found and that in subsequent steps mostly the resolution of the
structure is improved.  Let us also note that at iteration 20 we have only 923
degrees of freedom for $\e{\varphi}$ and in iteration 40 still only 1398 degrees
of freedom.  The final iteration has 125069 degrees of freedom.}
\label{fig:heavyGround:iterates}
\end{figure}

\subsection{Reproduction of results on drag minimization from earlier works}
\label{ssec:num:reproDragSurf}
We now reproduce the numerical results for the surface formulation of drag minimization presented by the authors in \cite{GHHKL}.  \AAA The key distinction is that in \cite{GHHKL} a gradient flow approach is used to solve the optimality conditions, leading to a non-linear time-dependent equation of Cahn--Hilliard type for $\e{\varphi}$.  However, here we employ the VMPT method to solve the optimization problem, which reads \BBB
\begin{align*}
\min_{(\varphi, \bm{u}, p)} \mathcal{J}_{\eps}(\varphi, \bm{u}, p) & = \int_{\Omega}
\tfrac{1}{2} \alpha_{\eps}(\varphi) \abs{\bm{u}}^{2} +  \frac{\gamma}{\pi} \left ( \frac{1}{\eps} \Psi(\varphi) +
\frac{\eps}{2} \abs{\nabla \varphi}^{2} \right ) \dx \\
& + \int_{\Omega} \tfrac{1}{2} \bm{a} \cdot \left
( \mu \left ( \nabla \bm{u} + \left ( \nabla \bm{u} \right )^{\top} \right ) - p \id
\right) \nabla \varphi \dx
\end{align*}
subject to $\varphi \in \AAA \Phi \BBB$, $\bm{u} \in \bm{H}^{1}_{\bm{g},
\sigma}(\Omega)$, $p \in L^{2}_{0}(\Omega)$ 
satisfying \eqref{NSweak} and the volume constraint (see \eqref{SI:VolumeConstraint})
\begin{align}\label{Drag:volCon}
\int_{\Omega} \varphi \dx \leq \beta_{2} \abs{\Omega} \text{ for } \beta_{2} \in
(-1,1).
\end{align}
We use the parameters from \cite{GHHKL}, 
namely $\Omega = (0,1.7) \times (0,0.4)$, $\eps = 0.00025$, $\overline{\alpha} =
0.03$, $\mu = 0.001$ and $\gamma = 0.01$.  
The boundary velocity is set to $\bm{g} = (1,0)^{\top}$ to stay close to the
analysis and we initialize the optimization with $\e{\varphi}^{0}(x) :=
-R[(0.5,0.2)^{\top},0.25,-1](x)$, i.e., a ball around $m = (0.5,0.2)^{\top}$ with
radius $r = 0.25$.  For the volume constraint, we choose $\beta_{2} = \beta =
0.975$, i.e., $\int_{\Omega} \varphi \dx \leq 0.663$.

To be able to use only a small number of unknown as long as possible, we start the
optimization with
$\eps = 0.008$ and a maximum number of allowed degrees of freedom of 10000.  
We halve the value of $\eps$ as soon as an optimal solution is found with current
maximum allowed number of degrees 
of freedom and increase this value by 20\%, resulting in 45000 unknowns for the
final result.  In Fig.~\ref{fig:num:repro_drag_surf} we show the optimal shape for different
values of $\eps$, 
namely $\eps \in \{0.008,0.004,0.002,0.001,0.0005,0.00025\}$.  
In Table~\ref{tab:num:repro_drag_surf} we show the diffuse interface drag
\begin{align}\label{PF:HydrdodynamicForce} 
\e{F}^{D} = \int_{\Omega} \tfrac{1}{2} \bm{a} \cdot \left ( \mu \left ( \nabla \bm{u} + \left ( \nabla \bm{u} \right)^{\top} \right) - p \id \right ) \nabla \varphi \dx
\end{align}
and the sharp interface drag 
\begin{align}\label{SI:HydrodynamicForce}
F^{D} = \int_{\Gamma} \bm{a} \cdot \left ( \mu \left ( \nabla \bm{u} + \left ( \nabla \bm{u} \right)^{\top} \right) - p \id \right )\bm{\nu} \dHaus
\end{align} by evaluation with $\bm{a} = (1.0,0.0)^{\top}$ over $\Gamma =
\{\e{\varphi} = 0\}$.

We reproduce the results found in \cite{GHHKL} where a gradient flow approach is applied that is based on an artificial time evolution.  We stress that, in using a gradient flow approach, the interface has to be resolved in each time step of the temporal evolution, which leads to a large numerical effort.  To be precise, while the results shown here are found in a few hours using the VMPT method, the results in \cite{GHHKL} required several days of calculation using the gradient flow.

\begin{figure}
  \centering
        \includegraphics[width=0.25\textwidth]{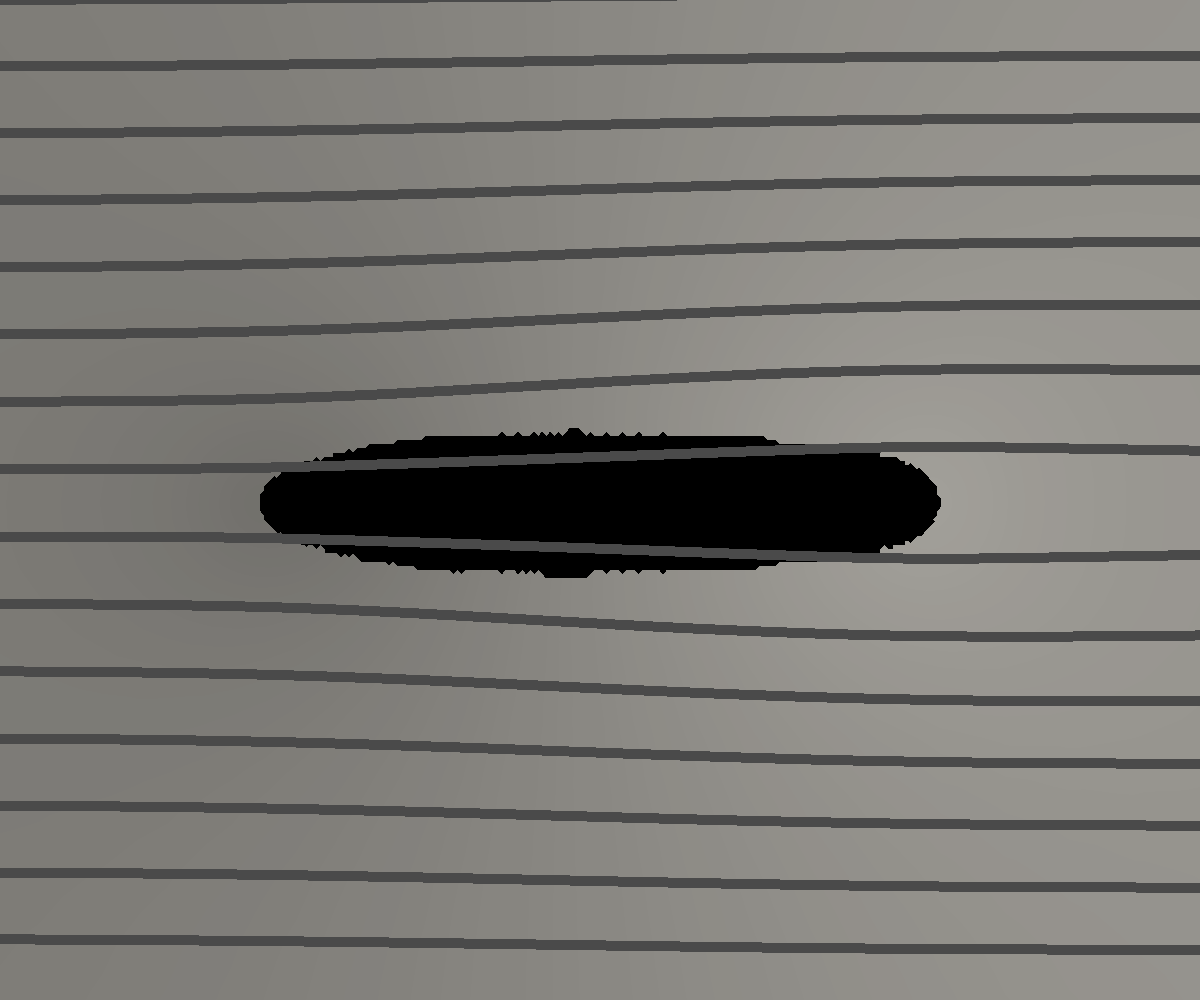}
        \includegraphics[width=0.25\textwidth]{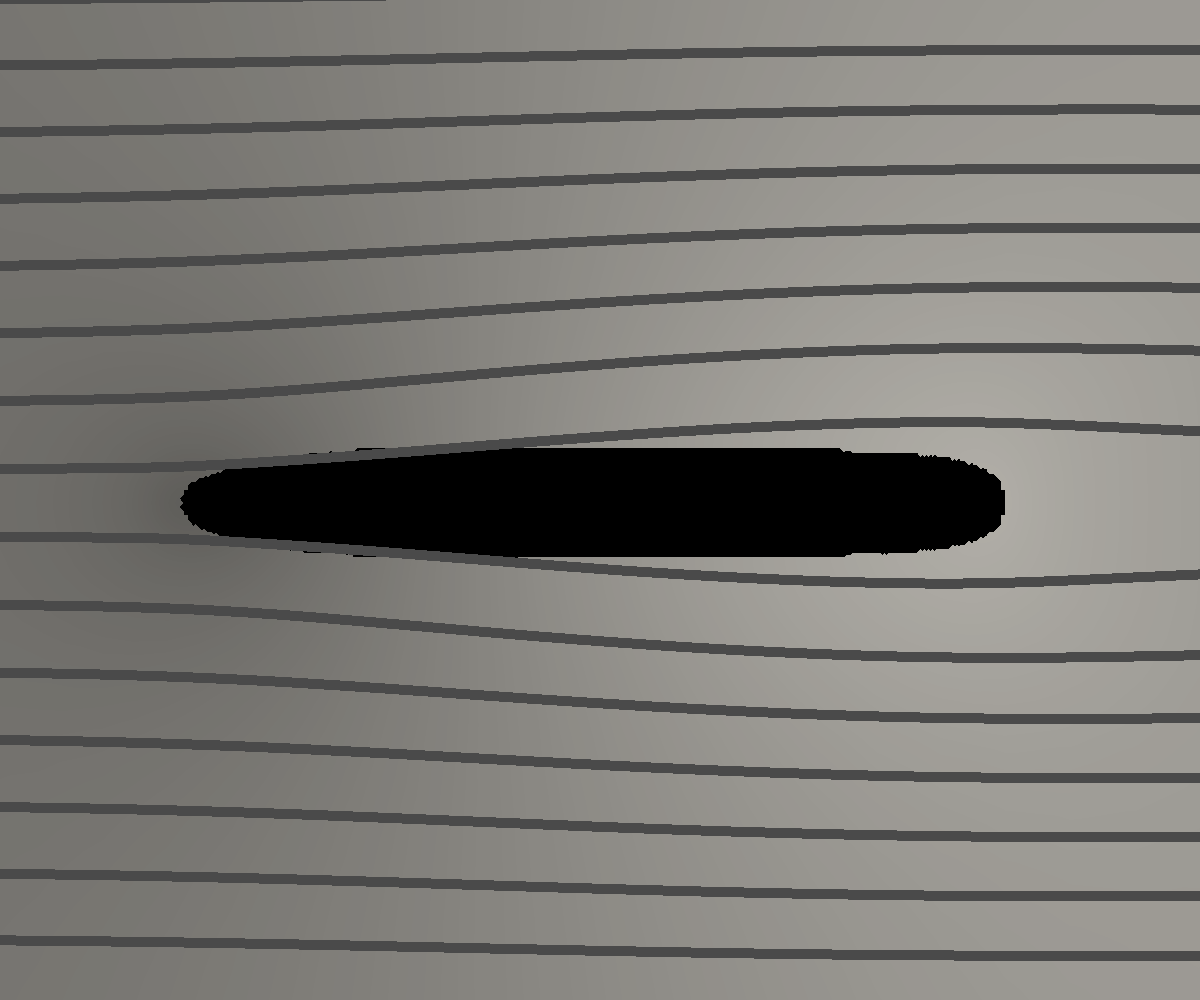}
        \includegraphics[width=0.25\textwidth]{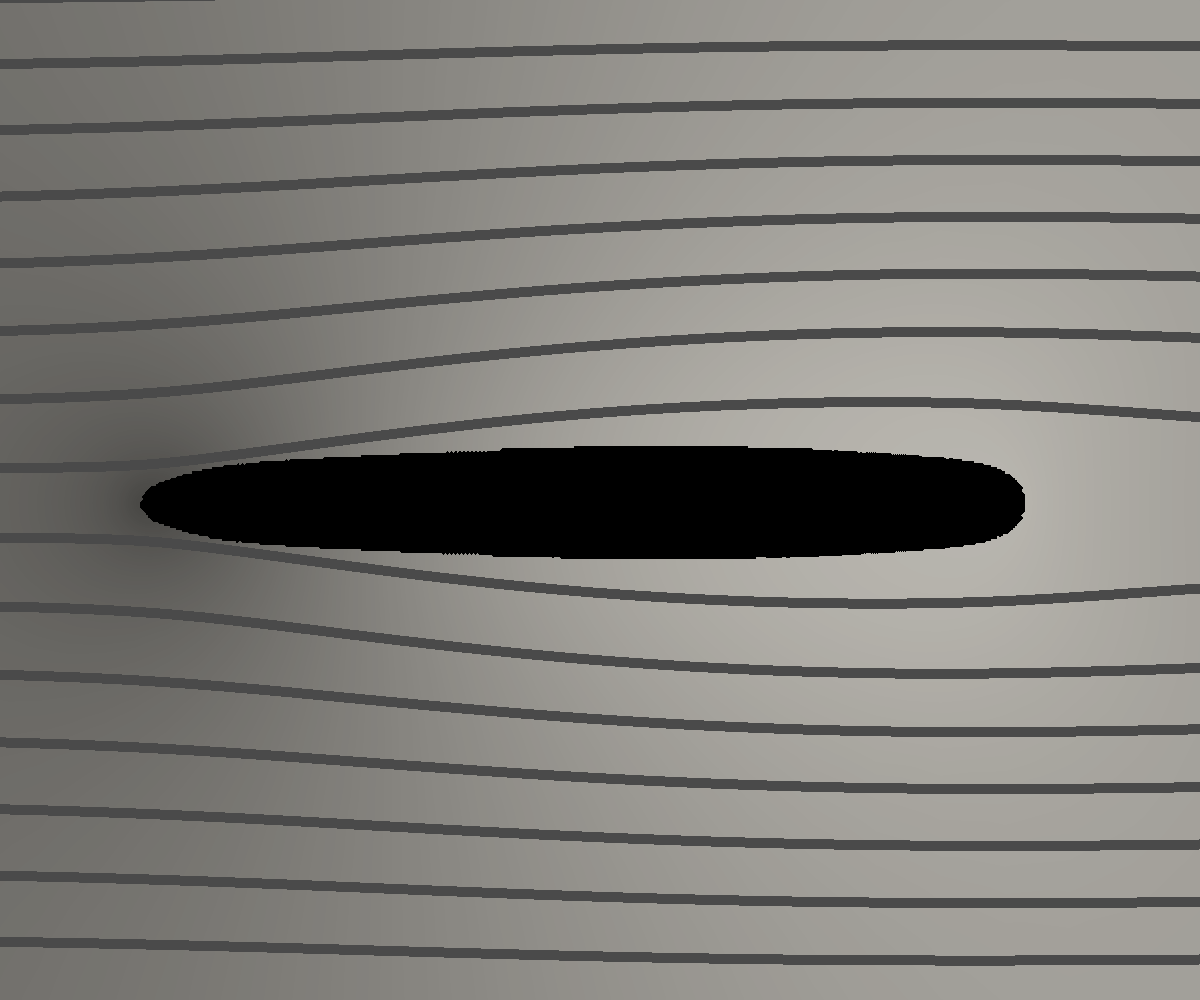}\\[1em]
        \includegraphics[width=0.25\textwidth]{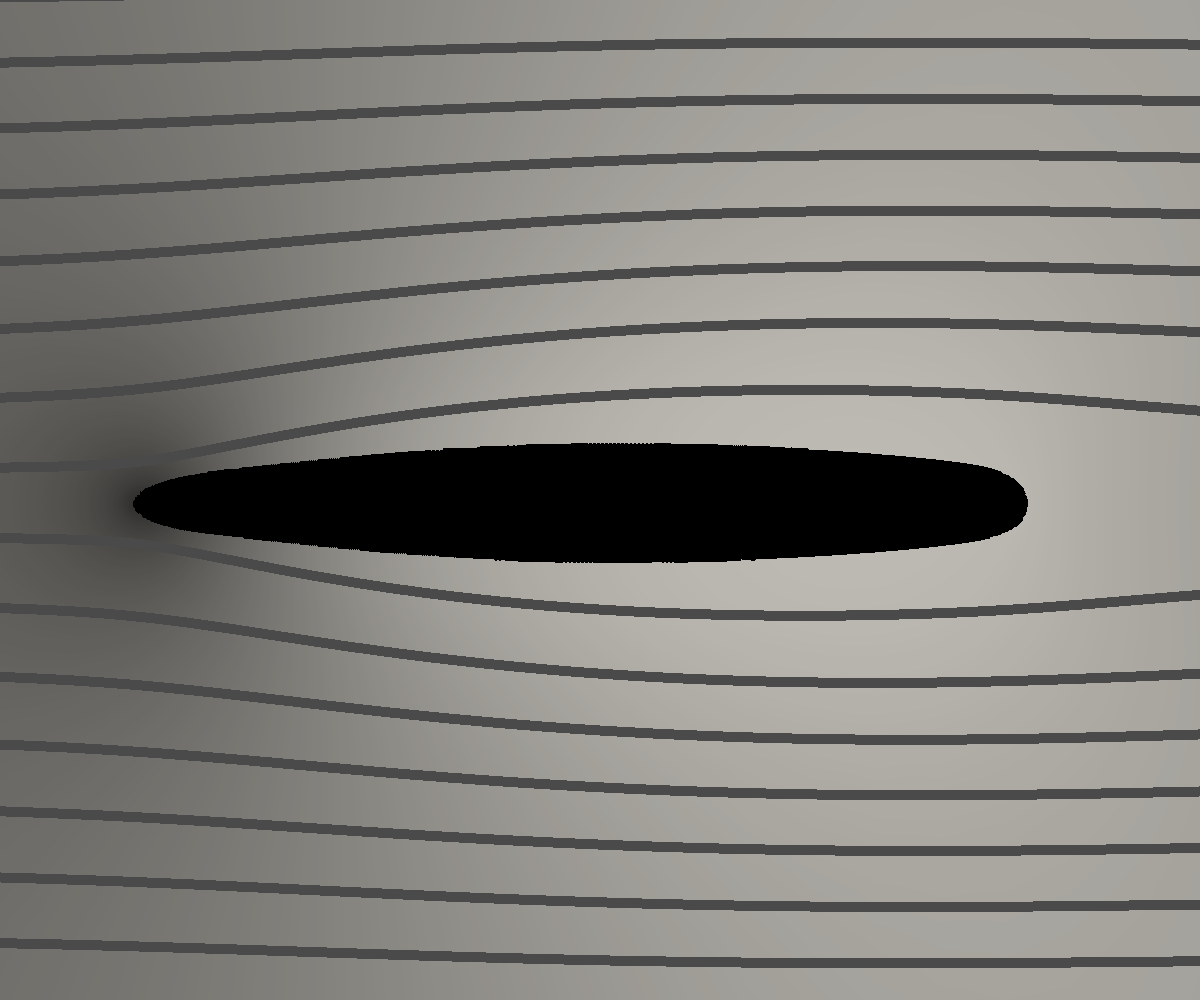}
        \includegraphics[width=0.25\textwidth]{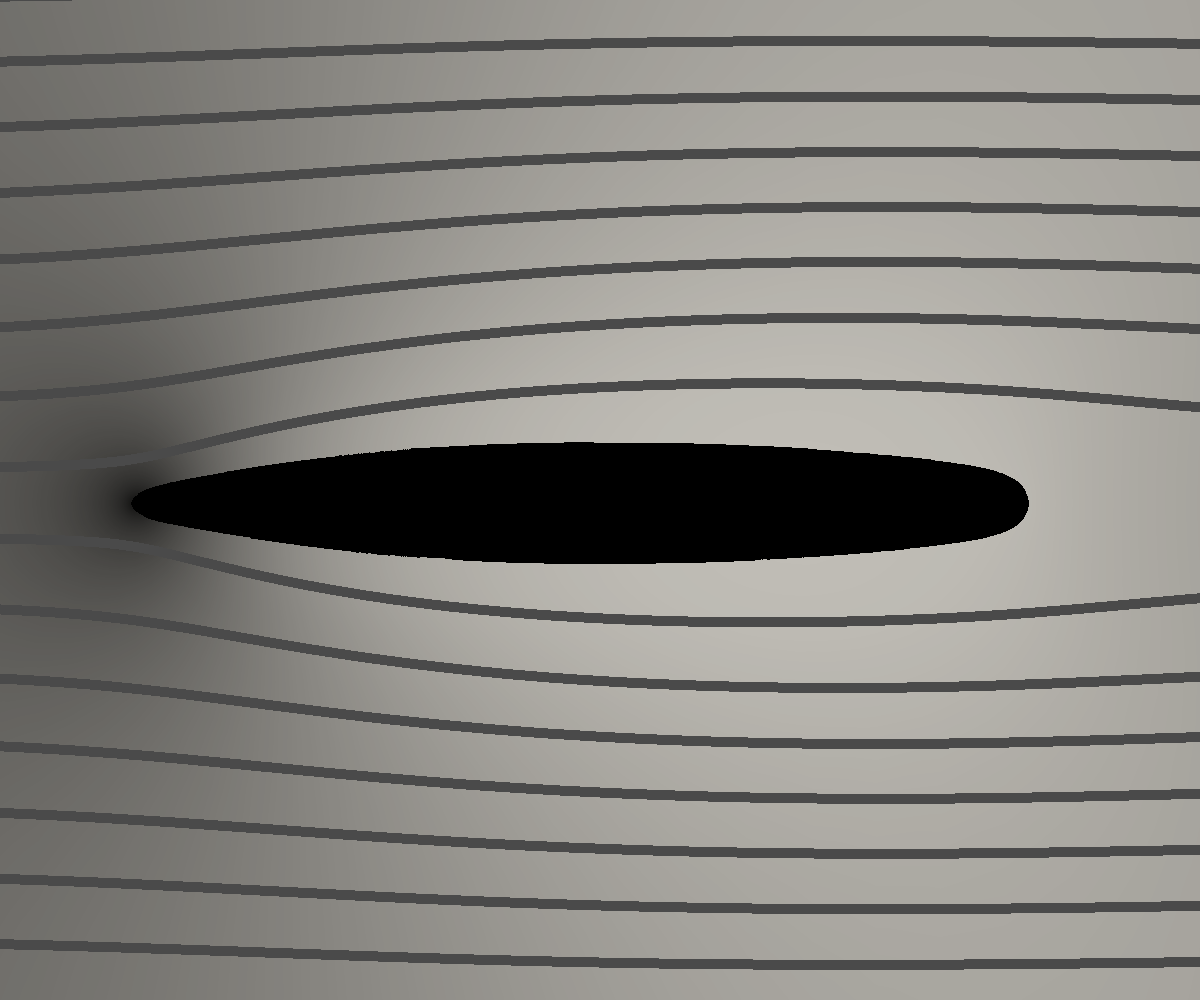}
        \includegraphics[width=0.25\textwidth]{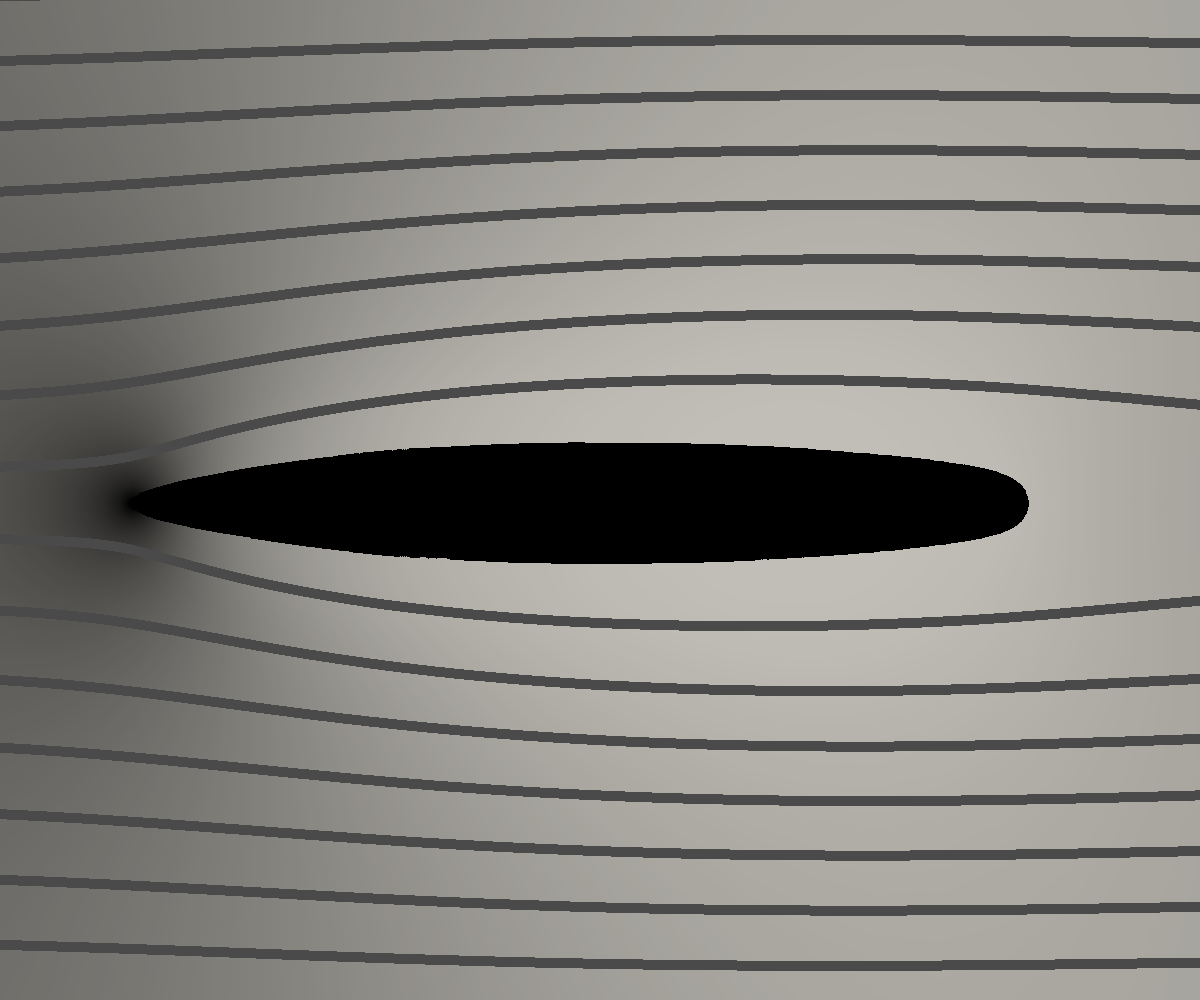}
        \caption{The optimal shapes for the minimization of drag in the surface formulation
with the
        parameters from \cite{GHHKL} for $\eps = 0.008,0.004,0.002,0.001,0.0005,0.00025$
(left upper
        to right lower). 
        Inflow from the left with $\bm{g} \equiv (1,0)^\top$.
        The shape is shown in black.
        The pressure is shown in gray, where darker gray means larger pressure, and some
streamlines of the velocity are shown in black.}
        \label{fig:num:repro_drag_surf}
\end{figure}

\begin{table}
\centering
\begin{tabular}{c|ccc}
 & & & \\
$\eps$ & 0.008  & 0.004 & 0.002\\
\hline
& & &  \\
$\e{F}^{D}$ & $1.0570 \times 10^{-2}$ & $1.9806  \times 10^{-2}$ & $2.8370  \times
10^{-2}$ \\
& & &  \\
$F^{D}$ & $1.1103  \times 10^{-2}$ & $2.0519 \times 10^{-2}$ & $2.9025 \times
10^{-2}$ \\
\hline
& & & \\
$\eps$ & 0.001 & 0.0005 & 0.00025\\
\hline
& & &  \\
$\e{F}^{D}$ & $3.4255 \times 10^{-2}$ & $3.8184 \times 10^{-2}$ & $4.0739 \times
10^{-2}$ \\
& & &  \\
$F^{D}$ & $3.4777 \times 10^{-2}$ & $3.8572 \times 10^{-2}$ & $4.1012 \times 10^{-2}$
\end{tabular}
\caption{The diffuse ($\e{F}^{D}$) and sharp ($F^{D}$) drag for the parameters from
\cite{GHHKL}
and different values of $\eps$.  Note that $\alpha_{\eps}(-1) \to \infty$ for $\eps
\to 0$, 
i.e., the object becomes less permeable and thus the drag increases with $\eps \to
0$.  
In \cite{GHHKL} for $\eps = 0.00025$ we observed $\e{F}^{D} = 3.9117 \times 10^{-2}$
and $F^{D} = 3.9499 \times 10^{-2}$. }
\label{tab:num:repro_drag_surf}
\end{table}

\subsection{Comparison of volume and surface formulations for drag}
\label{ssec:num:volsurf}
Let us point out that the hydrodynamic force component \eqref{SI:HydrodynamicForce} \AAA in its classical
representation as a surface integral over $\Gamma$ \BBB
can be expressed in terms of a volume integral over the fluid region $E$, and this reformulation has been used extensively in numerical simulations, see \cite[\S 5.1]{BLUU},  \cite[\S 2.2]{GLLS}, and \cite[\S 9]{HoffmanJohnson}.  Given the unit
vector $\bm{a} \AAA \neq \bm{0} \BBB$, let $\bm{\eta}$ be a smooth vector field such
that
\begin{align}\label{eta:vectorfield:property}
\bm{\eta} = \bm{a} \text{ on } \Gamma \text{ and } \bm{\eta} = \bm{0} \text{ on }
\pd \Omega.  
\end{align}
\AAA This can be done since it is assumed in the introduction that $\Gamma$ does not
intersect with $\pd \Omega$. \BBB  Then, by taking the scalar product of
\eqref{SI:state:1} with $\bm{\eta}$, we obtain
\begin{align*}
0 = \int_{E} - \div \left (\mu \nabla \bm{u} \right ) \cdot \bm{\eta} + (\bm{u}
\cdot \nabla) \bm{u} \cdot \bm{\eta} + \nabla p \cdot \bm{\eta} - \bm{f} \cdot
\bm{\eta} \dx.
\end{align*}
Integrating by parts and noting that the boundary integrals over $\pd \Omega$ vanish
owning to $\bm{\eta} = \bm{0}$ on $\pd \Omega$ yields
\begin{equation}\label{SI:HydrodynamicForce:Volume}
\begin{aligned}
& \int_{\Gamma} \bm{a} \cdot \left ( \mu \left ( \nabla \bm{u} + \left (
\nabla \bm{u} \right)^{\top} \right ) - p \id \right ) \bm{\nu} \dHaus = \int_{\Gamma} \bm{a} \cdot \left ( \mu  \nabla \bm{u} - p \id \right ) \bm{\nu} \\
& \quad = - \int_{E} \mu \nabla \bm{u} \cdot \nabla \bm{\eta} + (\bm{u} \cdot \nabla )
\bm{u} \cdot \bm{\eta} - p \div \bm{\eta} - \bm{f} \cdot \bm{\eta} \dx.
\end{aligned}
\end{equation}
Here we have also used that $\bm{u}$ has no tangential component on $\Gamma$ due to
the no-slip condition $\bm{u} = \bm{0}$ on $\Gamma$, and together with the
divergence-free condition, we obtain that $(\nabla \bm{u} )^{\top} \bm{\nu} =
\bm{0}$ on $\Gamma$ (see \cite[\S 2]{GHHKL} for more details).  This implies that we
can also consider the following function as the volume formulation of the drag (if $\bm{a}$ is parallel to the flow direction)
\begin{align}
\label{SI:VolumeDrag}
\int_{\Omega} -\tfrac{1}{2}(1+\varphi) \left (\mu \nabla \bm{u} \cdot \nabla \bm{\eta} + (\bm{u}
\cdot \nabla ) \bm{u} \cdot \bm{\eta} - p \div \bm{\eta} - \bm{f} \cdot \bm{\eta} \right ) \dx.
\end{align}
Alternatively, using integration by parts and the boundary conditions $\bm{\eta} =
\bm{0}$ on $\pd \Omega$ and $\bm{u} = \bm{0}$ on $\Gamma$, we see that 
\begin{align*}
\int_{E} (\bm{u} \cdot \nabla ) \bm{u} \cdot \bm{\eta} \dx = - \int_{E} (\bm{u}
\cdot \nabla) \bm{\eta} \cdot \bm{u} \dx,
\end{align*}
and so we may also use the function
\begin{align}\label{SI:VolumeDrag:alt}
\int_{\Omega} - \tfrac{1}{2}(1+\varphi) \left ( \mu \nabla \bm{u} \cdot \nabla \bm{\eta} -
(\bm{u} \cdot \nabla ) \bm{\eta} \cdot \bm{u} - p \div \bm{\eta} - \bm{f} \cdot
\bm{\eta} \right ) \dx,
\end{align}
as a volume formulation for the drag.  The corresponding phase field approximations of \eqref{SI:VolumeDrag} and \eqref{SI:VolumeDrag:alt} have the exact same form.  However, for our numerical investigations, we use the formation \eqref{SI:VolumeDrag:alt}
instead of \eqref{SI:VolumeDrag}.

The aim of this section is to compare results of Sec.~\ref{ssec:num:reproDragSurf} with the following drag minimization problem
\begin{align*}
\min_{(\varphi, \bm{u}, p)} \mathcal{J}_{\eps}(\varphi, \bm{u}, p) & = \int_{\Omega}
\frac{1}{2} \alpha_{\eps}(\varphi) \abs{\bm{u}}^{2} + \frac{\gamma}{\pi} \left ( \frac{1}{\eps} \Psi(\varphi) +
\frac{\eps}{2} \abs{\nabla \varphi}^{2} \right ) \dx \\
& + \int_{\Omega} - \tfrac{1}{2}(1+\varphi) \left ( \mu \nabla \bm{u} \cdot \nabla \bm{\eta} -
(\bm{u} \cdot \nabla ) \bm{\eta} \cdot \bm{u} - p \div \bm{\eta} - \bm{f} \cdot
\bm{\eta} \right ) \dx
\end{align*}
subject to $\varphi \in \Phi$, $\bm{u} \in \bm{H}^{1}_{\bm{g},
\sigma}(\Omega)$, $p \in L^{2}_{0}(\Omega)$ 
satisfying \eqref{NSweak} and the volume constraint \eqref{Drag:volCon}.  In particular, we compare the optimal shapes obtained with volume formulation \eqref{SI:VolumeDrag:alt} and those obtained with the surface formulation \eqref{PF:HydrdodynamicForce} for the drag.  For the above optimization problem, we consider the same setup as in Sec.~\ref{ssec:num:reproDragSurf} and set $\bm{\eta} \equiv \bm{a}$ on $(0.15,1.0)\times(0.13,0.27)$.

Using the surface formulation \eqref{PF:HydrdodynamicForce} we observe that for larger values of $\overline{\alpha}$ (the constant in \eqref{num:alpha}) an interfacial region $\{\abs{\e{\varphi}} < 1\}$ that is neither fluid nor object appearing in front of the object.  A similar behaviour was observed in the previous work \cite{GHHKL} with another minimization algorithm.  In any case, a sufficiently impermeable object can be obtained by using smaller values of $\eps$.  We stress that, in Sec.~\ref{ssec:num:reproDragSurf} for $\eps = 0.00025$ the the velocity $\abs{\e{\bm{u}}}$ inside the object is five orders of magnitude smaller than outside the object (see \cite[Fig.~1]{GHHKL}).

On the other hand, using the volume formulation
\eqref{SI:VolumeDrag:alt} we have to define the extension of the unit vector field $\bm{a}$, namely the vector field $\bm{\eta}$ which has to vanish at $\pd \Omega$.  We define $\bm{\eta}$ as the solution of a Poisson problem on $\Omega$ with $\bm{\eta} = \bm{a}$ on a square around the object and $\bm{\eta} = 0$ on $\pd \Omega$.  That is, let $S$ denote a square such that $\{\e{\varphi} = -1 \}
\subset S$ and $\pd S \cap \pd \Omega = \emptyset$, then we solve 
\begin{align}\label{eta:extension}
-\Laplace \bm{\eta} = \bm{0} \text{ in } \Omega \setminus S, \quad \bm{\eta} =
\bm{0} \text{ on } \pd \Omega, \quad \bm{\eta} = \bm{a} \text{ in } \overline{S}.
\end{align}
For small values of $\overline{\alpha}$, we observe that the object splits and the solid is collected close to the inflow outflow boundaries.  We believe this behavior is due to the following: On the one hand, due to the boundary condition $\bm{\eta} = \bm{0}$ on $\pd \Omega$, the magnitude $\abs{\bm{\eta}}$ is small close to the inflow and
outflow boundaries, which results in small drag forces.  On the other hand, for $\overline{\alpha}$ small, the porous-medium penalization term $\int_{\Omega} \e{\alpha}(\varphi) \abs{\bm{u}}^{2} \dx$ is small, and thus the value of the objective functional can be reduced by placing material in regions where $\abs{\bm{\eta}}$ is small.  Therefore, in contrast to the surface formulation \eqref{PF:HydrdodynamicForce}, large values of $\overline{\alpha}$ are needed for the volume formulation to obtain reasonable optimal shapes, which additionally allows us to construct sufficiently impermeable objects when we use larger values of
$\eps$.


We use the setup from Sec.~\ref{ssec:num:reproDragSurf} with only one
modification, that we set $\mu = 0.01$.  In Fig.~\ref{fig:num:volsurf:final} the optimal shapes of the
objects using the surface and the volume formulations of the drag are shown.  We observe that the front of the object with both formulations is rather similar, while the surface formulation leads to a less pronounced rear.  The corresponding drag values in sharp interface evaluation \eqref{SI:HydrodynamicForce} as defined in Sec.~\ref{ssec:num:reproDragSurf} are 
$F^{D} =0.106467052$ (volume formulation) and $F^{D} = 0.106470276$ (surface
formulation).

\begin{figure}
  \centering
  \includegraphics[width=0.3\textwidth]{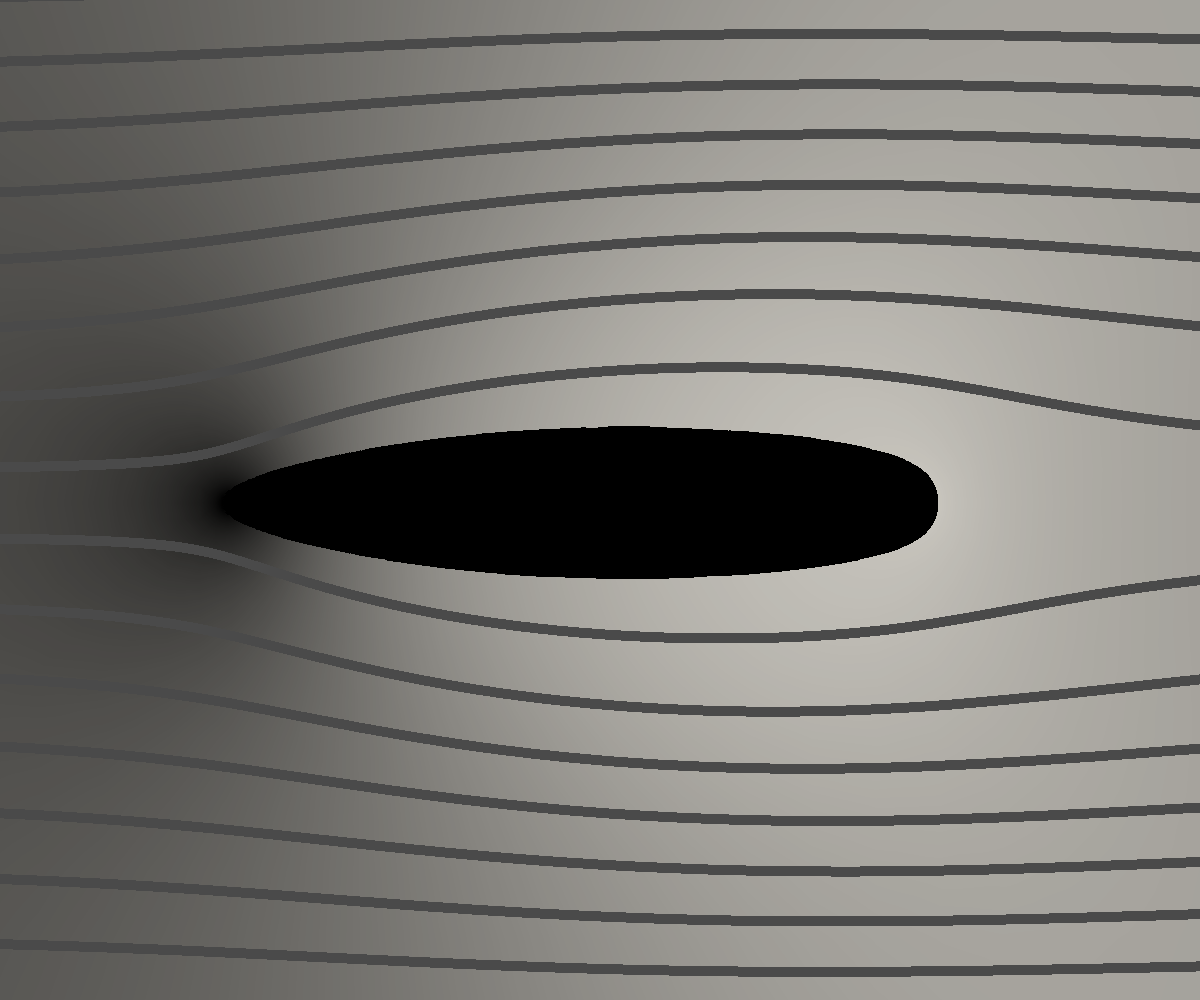}
  \hspace{1em} 
  \includegraphics[width=0.3\textwidth]{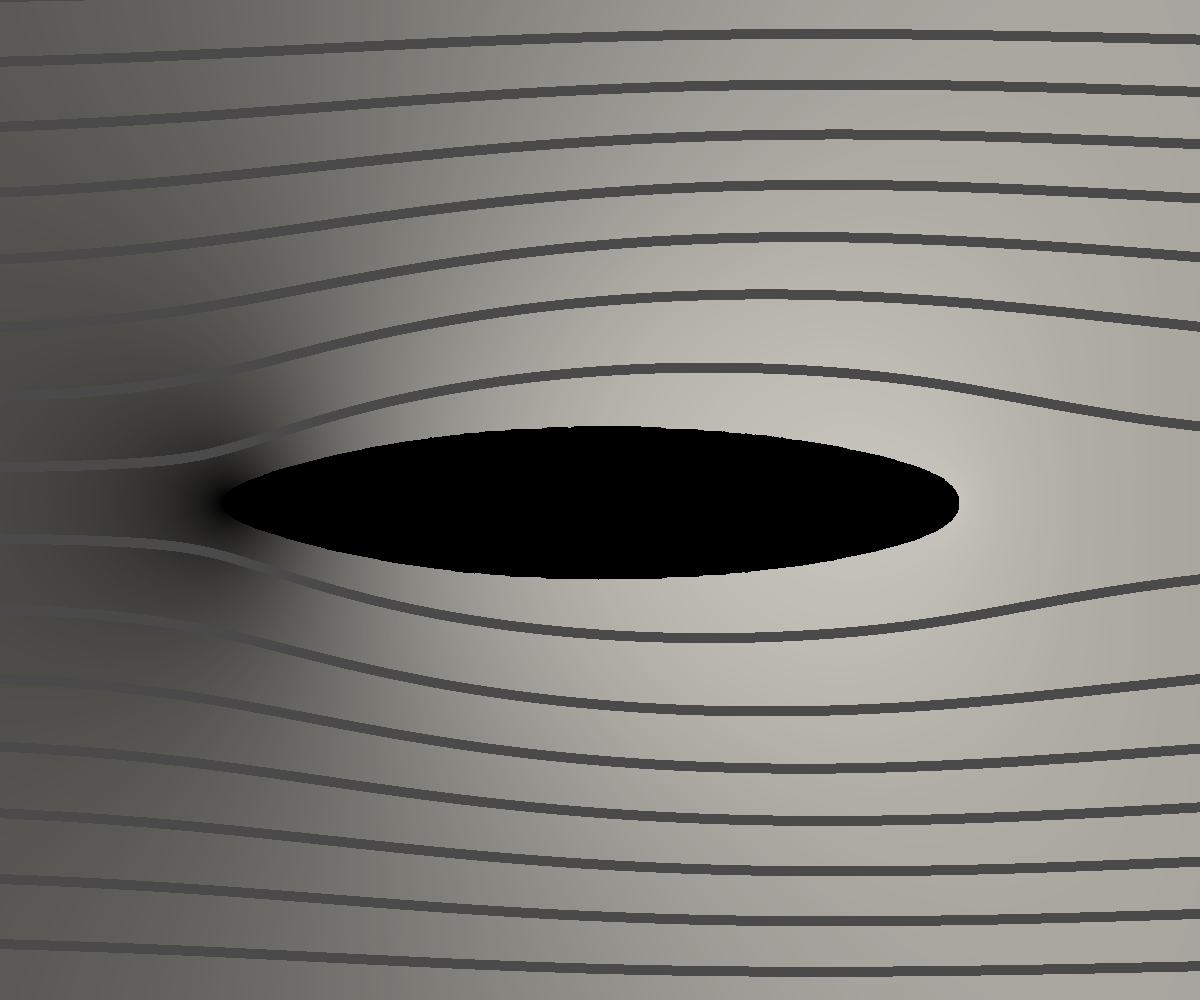} 
  \caption{The optimized shapes of the object using the surface formulation
(\eqref{PF:HydrdodynamicForce}, left) and the volume
  formulation (\eqref{SI:VolumeDrag:alt}, right) of the drag with
$\mu = 0.01$ and $\overline{\alpha} = 0.03$.  We observe that the rear of the
object is slightly more pronounced when the volume formulation is used, while the
drag measured on the zero level-line in both cases is nearly identical.}
  \label{fig:num:volsurf:final}
\end{figure}

As described above, using the volume formulation we can use larger values for $\overline{\alpha}$ to model objects with smaller permeability.  To show the influence of $\overline{\alpha}$ in Fig.~\ref{fig:num:volsurf:vol_large_aT} we show the optimal shape for the above parameters, but using a larger value $\overline{\alpha} = 1$ and $\mu = 0.01$ (left) and $\mu = 0.001$ (right).  For $\mu = 0.01$ we observe, that we get a sharper rear of the object, while the magnitude of the velocity inside the object is of order $10^{-4}$, which is two orders of magnitudes smaller than in the case $\overline{\alpha} = 0.03$.  We also mention that the shapes obtained here bear similarities to the optimized shape for the
minimization of the dissipative energy, as presented in \cite[Figs.~4 and
5]{GHHKNumerics}.  For $\overline{\alpha} = 1$, and $\mu = 0.001$ we observe a symmetric airfoil shape.

\begin{figure} 
  \centering
  \includegraphics[width=0.3\textwidth]{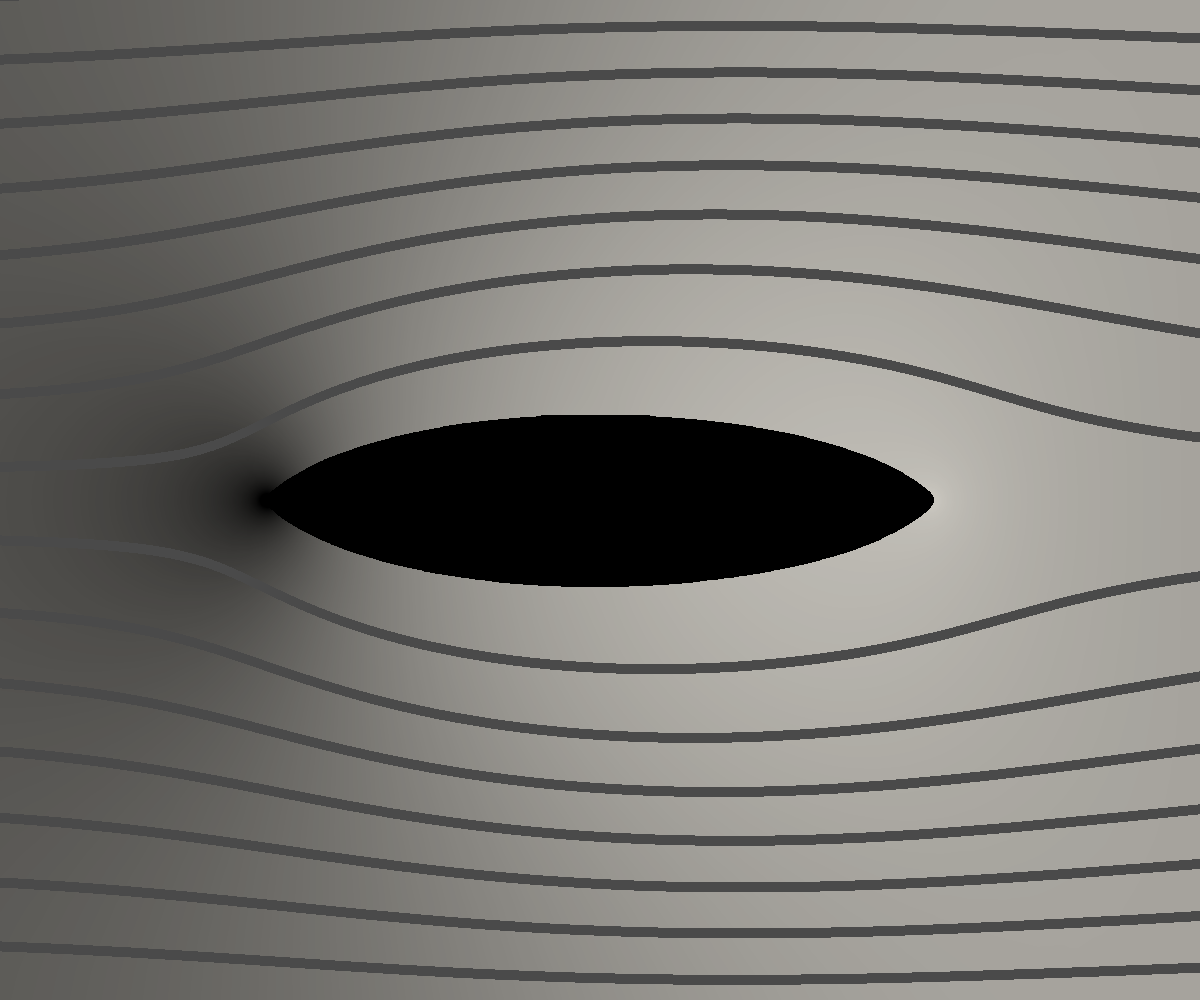}
  \hspace{1em}
  \includegraphics[width=0.3\textwidth]{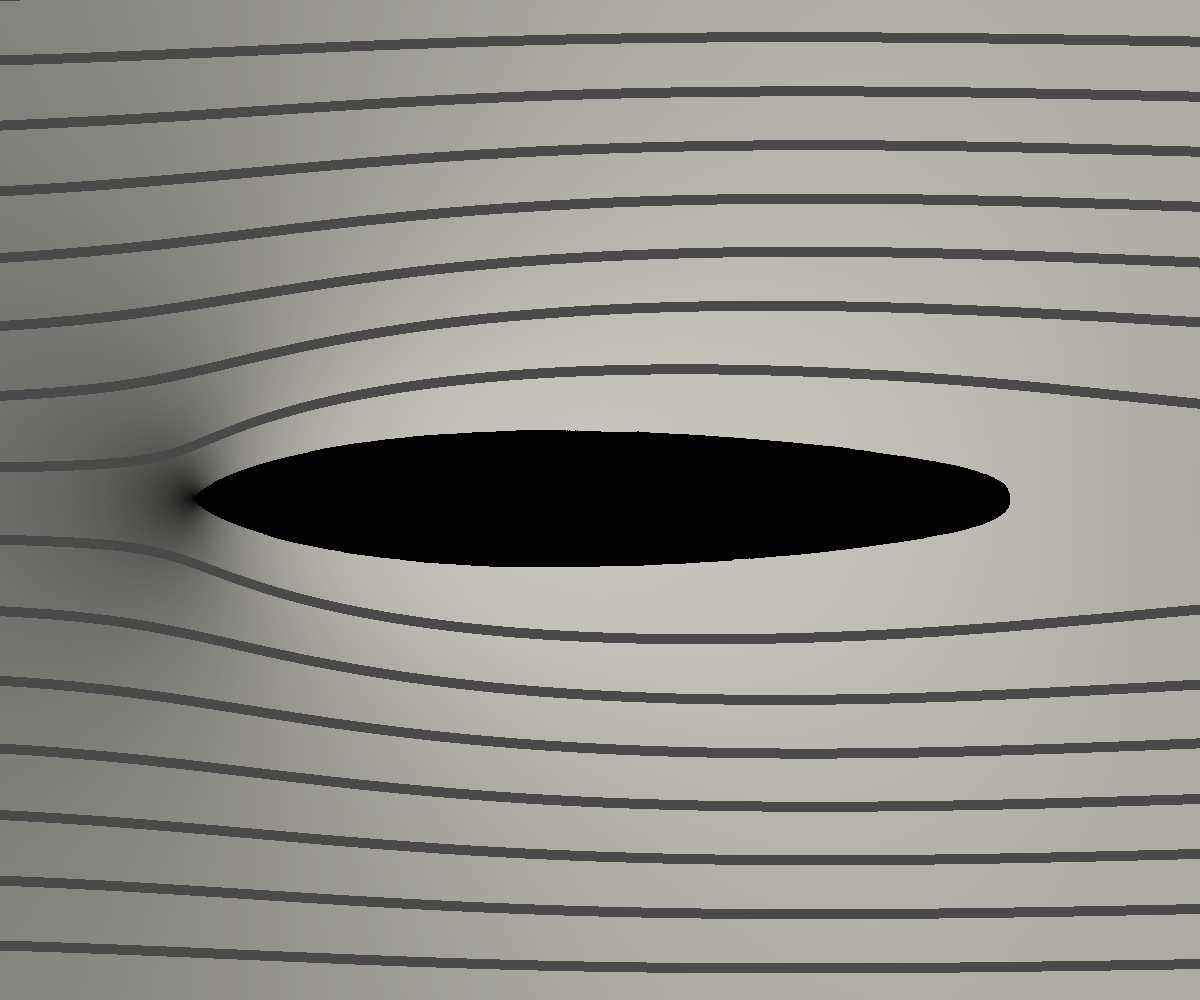}
  \caption{The optimized shape of the object using the volume formulation and
$\overline{\alpha} = 1$ with
  $\mu = 0.01$ (left) and $\mu = 0.001$ (right).    Compared to
Fig.~\ref{fig:num:volsurf:final} we observe a sharper rear and for $\mu = 0.001$
a symmetric airfoil shape emerges.  For $\mu = 0.01$ the drag is $F^{D} =
0.205542595$ and the velocity inside the object is of order $10^{-4}$, which is
two orders of magnitude smaller than in the case $\overline{\alpha} = 0.03$.  For
$\mu = 0.001$ the drag is $F^{D} = 0.041090517$ and the velocity inside the object
is of order $10^{-6}$. }
  \label{fig:num:volsurf:vol_large_aT}
\end{figure}

\subsection{Maximizing the lift with constraints on the total potential power}
\label{ssec:num:maxLift}
We give an example of dealing with a state constraint, namely we consider the maximization of the lift of an object under the constraint that the total potential power is bounded by some given value.  This is a non-linear constraint on the state variables of the constraint optimization problem, namely the velocity field.  To treat the highly non-linear potential power
constraint we use Moreau--Yosida relaxation.

The optimization problem we solve is
\begin{align*}
\begin{aligned}
\min_{(\varphi, \bm{u}, p)} \mathcal{J}_{\eps}^{\hp}(\varphi) & := 
\int_{\Omega} \frac{1}{2} \alpha_{\eps}(\varphi)  \abs{\bm{u}}^{2} + \frac{\bm{a}}{2}
 \cdot \left ( \mu \left ( \nabla \bm{u} + \left ( \nabla \bm{u} \right )^{\top} - p
\id \right ) \nabla \varphi \right ) \dx \\
& \quad + \int_{\Omega} \frac{\gamma}{\pi} \left ( \frac{1}{\eps}
\Psi(\varphi)  + \frac{\eps}{2} \abs{\nabla \varphi}^{2} \right ) \dx \\
&
+ \frac{\hp}{2} \max\left(0.0, \; \int_{\Omega} \frac{1+\varphi}{2} \frac{\mu}{2} \abs{\nabla \bm{u}}^{2}  - D
\abs{\Omega}^{-1} \dx \right)^{2},
\end{aligned}
\end{align*}
subject to $\varphi \in \Phi$, $\bm{u} \in \bm{H}^{1}_{\bm{g},
\sigma}(\Omega)$, $p \in L^{2}_{0}(\Omega)$ 
satisfying \eqref{NSweak}, where to realize the maximization of lift, we set $\bm{a} = (-1,0)^{\top}$ as the negative
unit vector perpendicular to the flow direction, and the parameter $\hp > 0$ penalizes violation of the constraint that the total potential power of the fluid region must be less than or equal to a prescribed value $D$:
\begin{align*}
\int_{\Omega} \frac{1+\varphi}{2} \frac{\mu}{2} \abs{\nabla \bm{u}}^{2}  \dx \leq D.
\end{align*}
\AAA The integral constraint for this optimization problem are volume constraints of the form $0.663 \leq \int_{\Omega} \e{\varphi} \dx \leq 0.665$ and the center of mass is fixed at $( 0.5,0.2)^\top$. \BBB

The set up is similar to that in Sec.~\ref{ssec:num:reproDragSurf}, where we set $\Omega = (0.0,1.7) \times (0.0,0.4)$, $\bm{g} = (1, 0)^{\top}$,
$\e{\varphi}^{0}(x) := -R[(0.5,0.2)^{\top},0.25,-1](x)$, i.e., a circle around $ \bm{m} = (0.5,0.2)^{\top}$
with radius $r = 0.25$.  For the penalization parameter we choose $\hp = 100$ and further numerical parameters are $\eps = 0.02$, $\overline{\alpha} = 2$, $\mu = 0.01$, $\gamma = 0.001$, and $D = 0.06$.
 
In Fig.~\ref{fig:num:maxLift:results} we show the resulting optimal shape of the object.  As expected we observe an inclined structure in order to maximize lift, but due to the constraint on the potential power, the angle of attack is restricted.  \AAA This is consistent with previous results in \cite[Fig.~2]{GHHKL}. \BBB
  
\begin{figure}
   \centering
\includegraphics[width=0.3\textwidth]{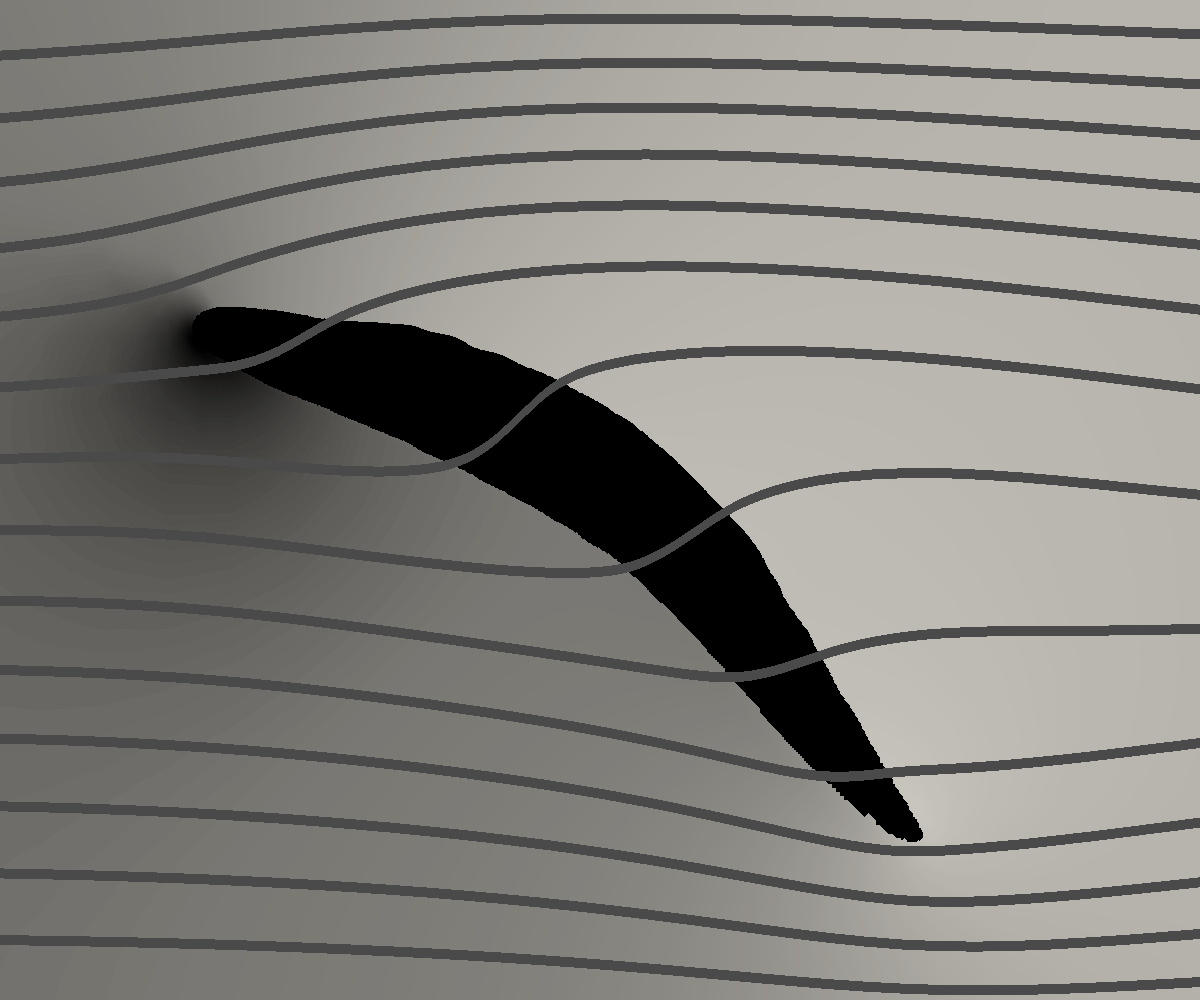}
\caption{The optimal shape for the maximization of the lift of an object, under a
constraint on
the dissipative power. We observe an inclined shape. }
\label{fig:num:maxLift:results}
\end{figure}

\AAA
\section{Conclusion}
In this paper, we formulate and analyze a phase field approximation for an abstract shape optimization problem subject to stationary Navier--Stokes flow with general objective functionals and integral state constraints.  We provide examples for the objective functionals and integral constraints that are of practical relevance, and we establish the existence of minimizers, and derive the first order optimality conditions.  A crucial point in the analysis is to show the existence of Lagrange multipliers corresponding to the integral constraints.  In the general setting we assume that the Zowe--Kurcyusz constraint qualification holds, and verify these assumptions for two specific examples.  The first involves integral constraints only in the variable $\varphi$, while the second involves the state variable $\bm{u}$.  The optimality conditions are solved using the VMPT method and several simulations are performed.  We demonstrate that the proposed phase field approach can handle topology optimization, and compare the results of drag minimization with previous works.  Lastly, we consider an example with an integral constraint involving the state variables, namely maximization of lift with constraint on the potential power.  The optimal shapes obtained are consistent with previous works on the lift-to-drag ratio for fluid flow with small Reynolds number.  
\BBB

\bibliographystyle{plain}  
\bibliography{GHKL_Ref}

\begin{thebibliography}{10}

\bibitem{Ambrosio}
L.~Ambrosio, N.~Fusco, and D.~Pallara.
\newblock {\em {Functions of Bounded Variation and Free Discontinuity
  Problems}}.
\newblock Oxford Mathematical Monographs. Oxford University Press, USA, 2000.

\bibitem{Becker}
R.~Becker and R.~Rannacher.
\newblock {An optimal control approach to a posteriori error estimation in
  finite element methods}.
\newblock {\em Acta Numerica}, 10:1--102, 2001.

\bibitem{Bello97}
J.~A. Bello, E.~Fern\'{a}ndex-Cara, J.~Lemoine, and J.~Simon.
\newblock {The differentiability of the drag with respect to the variations of
  a Lipschitz domain in a Navier--Stokes flow}.
\newblock {\em SIAM J. Control Optim.}, 35(2):626--640, 1997.

\bibitem{Blank}
L.~Blank, C.~Hecht, H.~Garcke, and C.~Rupprecht.
\newblock {Sharp interface limit for a phase field model in structural
  optimization}.
\newblock {\em SIAM J. Control Optim.}, 54:1558--1584, 2016.

\bibitem{Blank_Rupprecht}
L.~Blank and C.~Rupprecht.
\newblock {An extension of the projected gradient method to a Banach space
  setting with application in structural topology optimization}.
\newblock {\em SIAM J. Control Optim.}, 55:1481--1499, 2017.

\bibitem{Boisgerault}
S.~Boisg\'{e}rault and J.P. Zol\'{e}sio.
\newblock {Shape derivative of sharp functionals governed by Navier--Stokes
  flow}.
\newblock In W.~J\"{a}ger, J.~Ne\v{c}as, O.~John, K.~Najzar, and J.~Star\'{a},
  editors, {\em Partial Differential Equations: Theory and Numerical Solution},
  pages 49--63. Chapman and Hall/CRC, 1993.

\bibitem{BP03}
T.~Borrvall and J.~Petersson.
\newblock {Topology optimization of fluids in Stokes flow}.
\newblock {\em Internat. J. Numer. Methods Fluids}, 41(1):77--107, 2003.

\bibitem{BC03}
B.~Bourdin and A.~Chambolle.
\newblock {Design-dependent loads in topology optimization}.
\newblock {\em ESAIM Control Optim. Calc. Var.}, 9:19--48, 2003.

\bibitem{BLUU}
C.~Brandenburg, F.~Lindemann, M.~Ulbrich, and S.~Ulbrich.
\newblock {A Continuous Adjoint Approach to Shape Optimization for Navier
  Stokes Flow}.
\newblock In K.~Kunisch, J.~Sprekels, G.~Leugering, and F.~Tr\"oltzsch,
  editors, {\em {Optimal Control of Coupled Systems of Partial Differential
  Equations}}, volume 158 of {\em Internat. Ser. Numer. Math.}, pages 35--56.
  Birkh\"auser, 2009.

\bibitem{EvansGariepy}
L.C. Evans and R.F. Gariepy.
\newblock {\em {Measure Theory and Fine Properties of Functions}}.
\newblock Studies in advanced mathematics. CRC Press, Boca Raton, 1992.

\bibitem{GarckeHechtNS}
H.~Garcke and C.~Hecht.
\newblock {Applying a phase field approach for shape optimization of a
  stationary Navier-Stokes flow}.
\newblock {\em ESAIM: Control Optim. Calc. Var.}, 2015.

\bibitem{GarckeHechtStokes}
H.~Garcke and C.~Hecht.
\newblock {Shape and topology optimization in Stokes flow with a phase field
  approach}.
\newblock {\em Appl. Math. Optim.}, pages 1--48, 2015.

\bibitem{GHHKNumerics}
H.~Garcke, C.~Hecht, M.~Hinze, and C.~Kahle.
\newblock {Numerical approximation of phase field based shape and topology
  optimization for fluids}.
\newblock {\em SIAM J. Sci. Comput.}, 37(4):A1846--A1871, 2015.

\bibitem{GHHKL}
H.~Garcke, C.~Hecht, M.~Hinze, C.~Kahle, and K.F. Lam.
\newblock {Shape optimization for surface functionals in Navier--Stokes flow
  using a phase field approach}.
\newblock {\em Interfaces Free Bound.}, 18(2):219--261, 2016.

\bibitem{GLLS}
M.~Giles, M.~Larson, M.~Levenstam, and E.~S\"{u}li.
\newblock {Adaptive error control for finite element approximations of the lift
  and drag coefficients in viscous flow}.
\newblock {\em Technical Report NA-79/06, Oxford University Computing
  Laboratory}, 1997.

\bibitem{Giusti}
E.~Giusti.
\newblock {\em {Minimal surfaces and functions of bounded variation}},
  volume~80 of {\em Monographs in mathematics}.
\newblock Birkh\"{a}user Basel, 1984.

\bibitem{Goldberg92}
H.~Goldberg, W.~Kampowsky, and F.~Tr\"{o}ltzsch.
\newblock {On Nemytskij operators in $L_{p}$-spaces of abstract functions}.
\newblock {\em Math. Nachr.}, 155:127--140, 1992.

\bibitem{HechtThesis}
C.~Hecht.
\newblock {\em {Shape and topology optimization in fluids using a phase field
  approach and an application in structural optimization}}.
\newblock PhD thesis, University of Regensburg, 2014.

\bibitem{HHK}
M.~Hinterm\"{u}ller, M.~Hinze, and C.~Kahle.
\newblock {An adaptive finite element Moreau--Yosida-based solver for a coupled
  Cahn--Hilliard/Navier--Stokes system}.
\newblock {\em J. Comput. Phys.}, 235:810--827, 2013.

\bibitem{HHKK}
M.~Hinterm\"{u}ller, M.~Hinze, C.~Kahle, and T.~Keil.
\newblock {A goal-oriented dual-weighted adaptive finite elements approach for
  the optimal control of a Cahn--Hilliard--Navier--Stokes system}.
\newblock Preprint Hamburger Beitr\"{a}ge zur Angewandten Mathematik 2016-25,
  2016.

\bibitem{HPUU}
M.~Hinze, R.~Pinnau, M.~Ulbrich, and S.~Ulbrich.
\newblock {\em {Optimization with PDE Constraints}}.
\newblock Mathematical Modelling: Theory and Applications. Springer
  Netherlands, 2009.

\bibitem{HoffmanJohnson}
J.~Hoffman and C.~Johnson.
\newblock {Adaptive Finite Element Methods for Incompressible Fluid Flow}.
\newblock In T.J. Barth and H.~Deconinck, editors, {\em {Error Estimation and
  Adaptive Discretization Methods in Computational Fluid Dynamics}}, volume~25,
  pages 95--157. Springer Berlin Heidelberg, 2003.

\bibitem{KPTZ}
B.~Kawohl, O.~Pironneay, L.~Tartar, and J.-P. Zolesio.
\newblock {\em {Optimal Shape Design: Lectures Given at the Joint
  C.I.M./C.I.M.E. Summer School Held in Troia (Portugal), June 1-6, 1998}}.
\newblock Lecture Notes in Mathematics / C.I.M.E. Foundation Subseries.
  Springer-Verlag Berlin Heidelberg, 2000.

\bibitem{Kondoh}
T.~Kondoh, T.~Matsumori, and A.~Kawamoto.
\newblock Drag minimization and lift maximization in laminar flows via topology
  optimization employing simple objective function expressions based on body
  force integration.
\newblock {\em Struct. Multidiscip. Optim.}, 45(5):693--701, 2012.

\bibitem{Modica}
L.~Modica.
\newblock {The gradient theory of phase transitions and the minimal interface
  criterion}.
\newblock {\em Arch. Ration. Mech. Anal.}, 98(2):123--142, 1987.

\bibitem{Murat}
F.~Murat.
\newblock {Contre-exemples pour divers probl\`{e}mes o\`{u} le contr\^{o}le
  intervient dans les coefficients}.
\newblock {\em Ann. Mat. Pura Appl., Serie 4}, 112(1):49--68, 1977.

\bibitem{PRW}
P.~Penzler, M.~Rumpf, and B.~Wirth.
\newblock {A phase-field model for compliance shape optimization in nonlinear
  elasticity}.
\newblock {\em ESAIM: Control Optim. Calc. Var.}, 18:229--258, 2012.

\bibitem{Pironneau}
O.~Pironneau.
\newblock On optimum design in fluid mechanics.
\newblock {\em J. Fluid Mech.}, 64:97--110, 5 1974.

\bibitem{PlotSoko}
P.I. Plotnikov and J.~Sokolowski.
\newblock Shape derivative of drag functional.
\newblock {\em SIAM J. Control Optim.}, 48(7):4680--4706, 2010.

\bibitem{Robinson}
S.M. Robinson.
\newblock {Stability theorems for systems of inequalities, Part II:
  Differentiable nonlinear systems}.
\newblock {\em SIAM J. Numer. Anal.}, 13(4):497–--513, 1976.

\bibitem{SS10}
S.~Schmidt and V.~Schulz.
\newblock {Shape Derivatives for General Objective Functions and the
  Incompressible Navier--Stokes Equations}.
\newblock {\em Control Cybernet.}, 39(3):677--713, 2010.

\bibitem{Simon91}
J.~Simon.
\newblock {Domain variation for drag in Stokes flow}.
\newblock In {\em Control Theory of Distributed Parameter Systems and
  Applications}, volume 159 of {\em Lecture Notes in Control and Information
  Sciences}. Springer, Berlin, Heidelberg, 1991.

\bibitem{Sohr}
H.~Sohr.
\newblock {\em {The Navier-Stokes Equations: An Elementary Functional Analytic
  Approach}}.
\newblock Birkh{\"a}user Advanced Texts. Springer Verlag, 2001.

\bibitem{SHH}
K.~Sturm, M.~Hinterm\"{u}ller, and D.~H\"{o}mberg.
\newblock {Distortion compensation as a shape optimization problem for a sharp
  interface model}.
\newblock {\em Comput. Optim. Appl.}, 64:557--588, 2016.

\bibitem{Takezawa}
A.~Takezawa, S.~Nishiwaki, and M.~Kitamura.
\newblock {Shape and topology optimization based on the phase field method and
  sensitivity analysis}.
\newblock {\em J. Comput. Phys.}, 229:2697--2718, 2010.

\bibitem{Tartar}
L.~Tartar.
\newblock {Problemes de Controle des Coefficients Dans des Equations aux
  Derivees Partielles}.
\newblock In A.~Bensoussan and J.L. Lions, editors, {\em Control Theory,
  Numerical Methods and Computer Systems Modelling}, volume 107 of {\em Lecture
  Notes in Economics and Mathematical Systems}, pages 420--426. Springer Berlin
  Heidelberg, 1975.

\bibitem{Troltzsch}
F.~Tr{\"o}ltzsch.
\newblock {\em {Optimal Control of Partial Differential Equations: Theory,
  Methods, and Applications}}.
\newblock Graduate studies in mathematics. AMS, Providence, RI, 2010.

\bibitem{WangZhou}
M.Y. Wang and S.W. Zhou.
\newblock {Multimaterial structural topology optimization with a generalized
  Cahn--Hilliard model of multiphase transition}.
\newblock {\em Struct. Multidisc. Optim.}, 33:89--111, 2007.

\bibitem{ZoweKurcyusz}
J.~Zowe and S.~Kurcyusz.
\newblock {Regularity and Stability for the Mathematical Programming Problem in
  Banach Spaces}.
\newblock {\em Appl. Math. Optim.}, 5:49--62, 1979.

\end{thebibliography}

\end{document}